\documentclass{article}
\usepackage{amsmath,amsthm,amsfonts,amssymb,esint,mathtools,extarrows,witharrows,enumitem}

\renewcommand{\a }{\alpha }
\renewcommand{\b }{\beta }

\newcommand{\R}{\mathbb{R}}

\def\S{\Sigma} 
\def\n{\nabla}

\def\p{\partial}

\def\a{\alpha}
\def\b{\beta}
\def\<{\langle}
\def\>{\rangle}
\def\n{\nabla}

\def\p{\partial}

\def\a{\alpha}
\def\b{\beta}

\def\H{\mathcal{H}}

\def\s{\sigma}
\def\wt{\widetilde}

\def\wh{\widehat}

 \def\v{\varphi}
  \def\PP{\mathbb{P}}
  \def\NN{\mathbb{N}}
  \def\CC{\mathbb{C}}
  \def\RR{\mathbb{R}}
  
\makeatletter
\renewcommand*\env@matrix[1][*\c@MaxMatrixCols c]{%
  \hskip -\arraycolsep
  \let\@ifnextchar\new@ifnextchar
  \array{#1}} 
\makeatother

\newcommand\blfootnote[1]{
  \begingroup
  \renewcommand\thefootnote{}\footnote{#1}%
  \addtocounter{footnote}{-1}%
  \endgroup
}

\newtheorem{corollary}{Corollary}[section]
\newtheorem{remark}{Remark}[section]    
\newtheorem{lemma}{Lemma}[section]
\newtheorem{proposition}{Proposition}[section] 
\newtheorem{thm}{Theorem}

\newtheorem{thmx}{Theorem}
\numberwithin{equation}{section}

\title
{
(CMC) 1-immersions of surfaces  
into hyperbolic 3-manifolds.
}

\begin{document} 

\author{Gabriella Tarantello and Stefano Trapani}

\maketitle

\begin{abstract} 
Constant Mean Curvature (CMC) 1-immersions of surfaces into hyperbolic 3-manifolds are natural and yet rather unfamiliar objects in hyperbolic geometry, with 
curious features and interesting  applications. 

Firstly, Bryant in \cite{Bryant} reveled a surprising relation between (CMC) $1$-immersions of surfaces into the hyperbolic space $\mathbb H^3$ (now known as Bryant surfaces) and  minimal immersions into the euclidean space $\mathbb E^3.$ Ever since, the description of Bryant surfaces has been actively pursued in relation to their "cousins" minimal immersions, see e.g. \cite{Rosenberg} and references therein. 

In addition, the interest to constant mean curvature immersions of a surface $S$ (closed, orientable and of genus $\mathfrak{g} \geq2$) into hyperbolic 3-manifolds was motivated for example in  
 \cite{Uhlenbeck} and \cite{Goncalves_Uhlenbeck} in connection to irreducible representations of the fundamental group $\pi_{1}(S)$ into the Mobious group $PSL(2,\mathbb{C}).$ 
However on the basis of \cite{Bryant}, we see that a (CMC) 1-immersed compact surface  might develop singularities (punctures at finitely many points), and indeed in our analysis the prescribed value 1 of the mean curvature enters as a "critical" parameter. \\ 
More precisely from \cite{Huang_Lucia_Tarantello_2}  we know that, when  $\vert c \vert <1$ then (CMC) $c$-immersions of $S$ into hyperbolic 3-manifolds are always available and their moduli space can be parametrized by elements of the tangent bundle of the Teichm\"uller space $\mathcal{T}_{\mathfrak{g}}(S)$ of the surface $S.$ More importantly, (CMC) $1$-immersions can be attained only as "limits" of such (CMC) $c$-immersions, 
as $|c| \to 1^-,$ see \cite {Tar_2}.

On the other hand, the passage to the limit can be prevented by possible blow-up phenomena. Thus (after scaling) at the limit we may end up with a (CMC) 1-immersion into a tridimensional hyperbolic cone-manifold  (\cite{KS}), and the induced metric on the immersed surface will admit (finitely many) conical singularities (consistently with the presence of "smooth ends" described  in \cite{Bryant}) see Remark \ref{conical} for details.\\ 
In \cite{Tar_2} and \cite{Tar-Tra_1} it was proved that actually the passage to the limit can be ensured in terms of the  Kodaira map \eqref{Kodaira1} (cf. \cite{Tar_2}) and its suitable extension (cf. \cite{Tar-Tra_1}) respectively for surfaces of genus $\mathfrak{g}=2$ and  $\mathfrak{g}=3,$ see Theorem \ref{thmB}, Theorem \ref{thmC} and Theorem \ref{thmD} below and \cite{Tar-Tra_1}. \\
In this note we are able to handle the case of surfaces of any genus $\mathfrak{g}\geq 2.$ As initiated in \cite{Tar_2} and \cite{Tar-Tra_1}, we  capture the blow up situation in terms of a suitable "orthogonality" condition, and we refer to Theorem \ref{mainth} for details.\\
Subsequently, we can provide the existence and uniqueness of (CMC) 1-immersions under an appropriate "generic" condition, see Theorem \ref{thm1} for the precise statement.

\end{abstract}

%\tableofcontents

\section{Introduction}\label{introduction} 

\blfootnote
{
MSC: 35J50, 35J61, 53C42, 32G15, 30F60.
Keywords: 
Blow-up Analysis, Minimiser of a Donaldson functional,  CMC 1-immersions, Grassmannian, hyperelliptic curves. 
}

Let $S$ be an oriented closed surface with genus $\mathfrak{g}\geq 2$ and denote by $\mathcal{T}_{\mathfrak{g}}(S)$ the Teichm\"uller space of $S.$ 

We shall consider Constant Mean Curvature (CMC) $c$-immersions of $S$ into hyperbolic $3$-manifolds, i.e. immersions  with  prescribed value $c$ of the mean curvature.\\ In this context, the value $c=1$ plays a significant role. This fact was pointed out first by Bryant in \cite{Bryant},
where (CMC) $1$-immersions of surfaces into the hyperbolic space $\mathbb H^3$ were shown to share striking analogies with the (cousins) minimal immersions into the Euclidean space $\mathbb E^3$,  see \cite{Rosenberg} and also \cite{Rossman_Umehara_Yamada}, \cite{Umehara_Yamada}. 

Prompted by \cite{Goncalves_Uhlenbeck}, we aim to identify (CMC) 1-immersions of $S$ into hyperbolic 3-manifolds in terms of elements of $T(\mathcal{T}_{\mathfrak{g}}(S))$ the tangent bundle of  $\mathcal{T}_{\mathfrak{g}}(S).$\\ 
To this purpose we recall that in \cite{Huang_Lucia_Tarantello_2} it was shown that, for $|c| < 1$ the moduli space of (CMC) $c$-immersions into hyperbolic 3-manifolds can be parametrized by $T(\mathcal{T}_{\mathfrak{g}}(S)).$ Subsequently, in \cite{Tar_2} it was observed that (CMC) $1$-immersions can be detected only as "limits"  of the (CMC) $c$-immersions (obtained in  \cite{Huang_Lucia_Tarantello_2})  as $c \to 1^-,$ see Theorem \ref{thmprimobis} below.\\
On this basis our main effort will be to control 
the asymptotic behavior of (CMC) $c$-immersions in order to carry them out at the limit, as $c \to 1^-.$

To be more precise we follow  \cite{Uhlenbeck} and \cite{Goncalves_Uhlenbeck}, and for given $X\in \mathcal{T}_{\mathfrak{g}}(S)$ let us suppose for a moment that the Riemann surface $X$ is  immersed with constant mean curvature  $c$ into the hyperbolic 3-dimensional manifold $(N, \hat{g}).$ To attain hyperbolicity, the metric tensor  
$\hat{g}=(\hat{g}_{ij})$  and the Riemann curvature tensor $R_{ijlk}$ of $(N, \hat{g})$  must satisfy the following relation:
\begin{equation}\label{0.1a}
R_{ijlk}=-(\hat{g}_{il}\hat{g}_{jk}-\hat{g}_{ik}\hat{g}_{jl}) \,\mbox{ with }\, 1\leq i, j, k, l \leq 3,
\end{equation}
and the system \eqref{0.1a} expresses actually six independent equations for the six independent components of the Riemann tensor.\\ 
For a more explicit interpretation of \eqref{0.1a}, we introduce Fermi coordinates: $(z, r)\in X\times(a, -a),$ with holomorphic $z$-coordinates in X and $a>0$ small. Thus, in a tubular neighborhood of X in $N,$  we have: 
$\hat{g}_{i3}(z, r)=\delta_{i3},$ for $i=1,2,3$,  and in view of (\ref{0.1a}), the remaining (three) components: $\hat{g}_{ij}(z, r)$, $1\leq i \leq j \leq 2$ satisfy:
\begin{equation}\label{0.2a}
R_{i3j3}=-\hat{g}_{ij}.
\end{equation}
By explicit calculation we see that, in the $(z, r)-$coordinates, the (three) equations in (\ref{0.2a}) defines a $2^{nd}$ order system of ODE's for $\hat{g}_{ij}$ with respect to the variable $r$ (and $z$ fixed), see \cite{Uhlenbeck} and \cite{Huang_Lucia_Tarantello_1} for details.\\
So for $|r|< a$ (and every $z \in X$) the metric $(\hat{g}_{ij})$ is uniquely identified  by its initial  data at $r=0.$ Clearly, such Cauchy data are expressed in terms of the pullback metric $g$ on $X$ and the second fundamental form $II_g,$ and they are constrained by the remaining (three) independent equations
in (\ref{0.1a}), as given by:
\begin{align}
&R_{ijl3}=0 \label{0.4a}\\
&R_{1212}=-(\hat{g}_{11}\hat{g}_{22}-\hat{g}_{12}^2) \label{0.5a}
\end{align}
see \cite{Uhlenbeck} and \cite{Goncalves_Uhlenbeck} for details.
In fact, by Bianchi identity it suffices that  (\ref{0.4a})  and (\ref{0.5a}) are satisfied at $r=0$ in order to hold for any $r\neq0$, see \cite{Lawson} for details. \\
To proceed further, we denote by $g_X$ the unique hyperbolic metric on $X$ (i.e. with constant Gauss curvature -1) as given by the uniformisation theorem. Then by compatibility, the pullback metric $g$ must  be conformally equivalent to $g_X,$ namely: $g=e^{u}g_X,$ with a suitable function $u$ smooth in $X.$ 

In addition, we note that the second fundamental form $II_g$ relative to a (CMC) $c$-immersion is completely identified by its $(2,0)$-part, together with $g$ and $c$. In other words, if we let $\alpha := (2,0)-part$ of $II_g,$ then the pair: $(u, \alpha)$ completely identify the Cauchy data, and the equations \eqref{0.4a}, \eqref{0.5a} expressed in terms of $(u, \alpha)$ define the well known Gauss-Codazzi equations.   \\
To be more precise, let $E=T^{1,0}_X$ be the holomorphic tangent bundle of $X$ with dual $E^*=K_X$ defining the canonical bundle of $X.$\\ Both holomorphic line bundles $E$ and $E^*$ (and their tensor products) inherit the complex structure induced by $X$ and the hermitian product induced by $g$ or $g_X,$  with corresponding norm denoted by: $\|\cdot\|_g$ and $\|\cdot\|$ respectively.  
Notice that,  $\alpha = (2,0)-part$ of $II_g$ defines a (1, 0)-form valued in $K_X,$ and (as already known to Hopf) for $r=0$ the two (independent) equations in (\ref{0.4a})
combine into the following (complex) \underline{Codazzi equation}:
\begin{equation}\label{0.6a}
\bar{\partial}\alpha=0 
\end{equation}
where $\bar{\partial}$ corresponds to the d-bar operator in the complex structure of $K_X \otimes K_X$ (induced by $X$) see 
\cite{Uhlenbeck} for details.
Equivalently, if $C_2(X)$ denotes the finite dimensional (complex) space of holomorphic quadratic differentials on $X,$ then 
$$\a\, \mbox{satisfies \eqref{0.6a}} \iff \a\in C_2(X) \quad \mbox{and} \quad \mbox{dim} _{\mathbb{C}}( C_2(X)) = 3(\mathfrak{g}-1),$$ see \cite{Miranda}.  

Equation (\ref{0.5a}) at $r=0$ yields to the \underline{Gauss equation} for $u$, and it states compatibility of the  Gauss curvature $K_g$ of $(X, g) $ with its extrinsic expression computed in terms of the given immersion, namely:
\begin{equation}\label{0.7a}
K_g=-1+c^2-4\|\alpha\|_g^2.
\end{equation}

By recalling that for $g=e^ug_X$ we have: $\|\alpha\|_g=\|\alpha\|e^{-u}$ and  $\ K_g=e^{-u}(-\frac12\Delta_{X} u-1)$, with  $\Delta_{X}$ the Laplace Beltrami operator in $(X, g_X)$, we can formulate (\ref{0.7a}) in terms of the conformal factor $u,$ as the following elliptic equation of Liouville type:
 \begin{equation}\label{0.8a}
 -\Delta_{X} u=2-2(1-c^2)e^u-8\|\alpha\|^2e^{-u}.
\end{equation}
Conversely, as shown by Taubes in \cite{Taubes}, any solution $(u, \alpha)$ of the Gauss-Codazzi equations (\ref{0.8a})-(\ref{0.6a}) provides an appropriate set of initial data at $r=0$  (and fixed $z$) for a suitable second order system of O.D.E. in the $r$-variable. 
 Thus, by the local solvability of the corresponding Cauchy problem (in the spirit of system (\ref{0.2a})) we obtain  a metric for a hyperbolic 3-manifold: $(N,\hat{g})$ ($N\simeq X\times\mathbb{R}$ not necessarily complete) where $X$ is immersed as a surface with constant mean curvature $c.$  Details for the construction of $(N, \hat{g})$ is provided in \cite{Taubes}, where $(N, \hat{g})$ is referred as a ``germ" of hyperbolic 3-manifolds around $X$ (corresponding to the solution pair $(u, \alpha)$  of (\ref{0.8a})-(\ref{0.6a})). Such an immersion (into a germ of hyperbolic 3-manifolds) is unique up to   diffeomorphims of small tubular neighborhoods of $X,$ and therefore can be taken as a representative for elements of the moduli space of  (CMC) $c-$immersion of $S.$ In this way, we find a parametrization for the moduli space as soon as we find a parametrization for the solution set of the Gauss-Codazzi equation (\ref{0.8a})-(\ref{0.6a}).

It may be tempting to describe the solution set of the Gauss-Codazzi equations (\ref{0.8a})-(\ref{0.6a})
in terms of 
elements of the cotangent bundle of 
$\mathcal{T}_{\mathfrak{g}}(S)$ as given by the pairs:
$
(X,\alpha) \in \mathcal{T}_{\mathfrak{g}}(S) \times C_{2}(X),
$ see \cite{Jost}.
However, 
as discussed in \cite{Huang_Lucia} and \cite{Huang_Lucia_Tarantello_1},  for a fixed $\a\in C_2(X)$ a solution of \eqref{0.8a} 
may not exist, or (when it exists) 
it may not be unique (see also \cite{Huang_Loftin_Lucia}).
So, in general, the pair $(X,\a)$ is not suitable to parameterized (CMC) $c$-immersions.\\ 
Instead, Goncalves and Uhlenbeck  \cite{Goncalves_Uhlenbeck} proposed a (more successful)   "dual" approach, and suggested 
to parametrize the moduli space of (CMC) $c$-immersions of $S$
into hyperbolic $3$-manifolds, 
by elements $(X,[\beta]) \in T (\mathcal{T}_{\mathfrak{g}}(S))$ the tangent bundle  of the Teichm\"uller space $\mathcal{T}_{\mathfrak{g}}(S).$ Hence, the class $[\beta] \in \mathcal{H}^{0,1}(X,E),$ with $E=T^{1,0}_{X}$ and  $\mathcal{H}^{0,1}(X,E)$ the Dolbeault (0,1)-cohomology group (see \eqref{cohomologygroup}, \eqref{classbeta} below) and we have: $C_{2}(X) \simeq (\mathcal{H}^{0,1}(X,E))^{*}$ 
 (cf \cite{Griffiths_Harris}).
 \\
Interestingly, accordingly to  \cite{Goncalves_Uhlenbeck}, the datum $(X, [ \beta ])$ should  identify the \underline{unique} solution 
$(u, \, \a)$ of the Gauss-Codazzi equations (\ref{0.8a})-(\ref{0.6a}) subject to the constraint:
\begin{eqnarray}\label{constraint0}
   *_E^{-1}(e^{-u}\a) \in[\b],
\end{eqnarray} where $*_E$ is the Hodge star operator relative to the metric $g_X,$ acting between (dual) forms valued on $E$ and $E^*$ respectively. As well known, the map  $*_E$ defines an isometry with inverse $(*_E)^{-1},$  for details see \eqref{2.8*hodge_operator_intro} below.
This program was rigorously carried out in \cite{Huang_Lucia_Tarantello_2}, and for $|c|<1$  
(as anticipated by \cite{Goncalves_Uhlenbeck}) the following holds:

\begin{thmx}[\cite{Goncalves_Uhlenbeck},\cite{Huang_Lucia_Tarantello_2}]
\label{thm_A}
For given $c\in (-1,1)$ there is a one-to-one correspondence between the space of constant mean curvature $c$-immersions of $S$ into a (germ of) hyperbolic $3$-manifolds  and the 
tangent bundle of 
$\mathcal{T}_{\mathfrak{g}}(S)$, parametrized by the pairs: 
$
(X,[\beta])\in \mathcal{T}_{\mathfrak{g}}(S) \times \mathcal{H}^{0,1}(X,E),
\; 
E=T_{X}^{1,0}.  
$ 
\end{thmx}	
As discussed in \cite{Uhlenbeck} and \cite{Taubes} (and more generally in \cite{Huang_Lucia_Tarantello_2}) from Theorem \ref{thm_A} one can deduce useful algebraic information about all possible irreducible representations of the fundamental group $\pi_{1}(S)$ into the Mobious group $PSL(2,\mathbb{C})$ 
(or PU(2,1)).
Also we mention \cite{Loftin_Macintosh_3} for analogous results in the context of Lagrangean immersions and  \cite{Loftin_Macintosh_1}, \cite{Loftin_Macintosh_2}  concerning minimal immersions in various contexts via the Higgs bundle approach of Hitchin's selfduality theory discussed below.  
For further details see \cite{Huang_Lucia_Tarantello_2} and \cite{Tar_2} .
\\ \\
To see that \eqref{constraint0} is a "natural" constraint for the Gauss-Codazzi equations (\ref{0.8a})-(\ref{0.6a}), we recall that according to Dolbeault decomposition, any Beltrami differential $\beta$ (i.e. a $(0,1)$-form valued in $E$) admits the following unique decomposition: 
$$
\beta = \beta_{0} + \bar{\partial}\eta
$$
with  $\beta_{0}$ \underline{harmonic} (with respect to $g_{X}$) and  
$\eta$ a smooth section of X valued on E.
Hence, the corresponding  (0,1)-cohomology class  $[\beta] \in \mathcal{H}^{0,1}(X,E)$ is uniquely identified
by its harmonic representative: $\beta_{0} \in [\beta].$\\ 
In this way, for a fixed pair $(X, [\beta])$ we can formulate the Gauss-Codazzi equations constrained by \eqref{constraint0} by letting:
\begin{equation}\label{constraint}
g=e^u g_X \quad  \a=e^{u}*_E(\beta_{0} + \bar{\partial}\eta),\,\,\mbox{with harmonic}\,\, \beta_0 \in [\beta]\,\,\mbox{and }\, \eta\in A^{0}(X),
\end{equation}
and see (by recalling that $*_E$ is norm preserving,) that $(u\, , \eta)$ must satisfy:
\begin{equation}\label{system_of_equations_introbis}
\left\{
\begin{matrix*}[l]
\Delta_{X} u +2 -2te^{u} -8e^{u}\Vert \beta_{0}+\overline{\partial}\eta \Vert^{2} =0  &  \;\text{ in }\;  &  X  \\
\overline{\partial}(e^{u}*_{E}(\beta_{0}+\overline{\partial}\eta))=0 &  \;\text{}\;  &   \\ 
\end{matrix*}
\right.
\end{equation} and $t=1-c^2.$
\\
It is interesting to notice that system \eqref{system_of_equations_introbis}
can be formulated in terms of  Hitchin's self-duality equations \cite{Hitchin}
with respect to a suitable nilpotent $SL(2,\mathbb{C})$ Higgs bundle, we refer to \cite{Alessandrini_Li_Sanders} and \cite{Huang_Lucia_Tarantello_2} for details. 
Therefore, on the ground of Hitchin's sefduality theory, the existence and uniqueness for \eqref{system_of_equations_introbis} is equivalent 
 to the "stability" of the given Higgs bundle (cfr \cite{Hitchin} and \cite{Wentworth}). The "stability" property has been succcessfully verified in the context of minimal immersions (see e.g. \cite{Li}, \cite{Alessandrini_Li_Sanders}, \cite{Hitchin}, \cite{Donaldson}, \cite{Loftin_Macintosh_1} and \cite{Loftin_Macintosh_2}) but it appears difficult to be directly checked in our context. 

However, it is easy to check that 
for any given pair $(X,[\beta])$ then (weak) solutions of the "constraint" Gauss-Codazzi equations (\ref{system_of_equations_introbis}) correspond to  critical points of the following Donaldson functional  introduced (and so called) in \cite{Goncalves_Uhlenbeck}:
\begin{equation}\label{F_t}
F_{t}(u,\eta)
= 
\int_{X}
\left(\frac{\vert \nabla_{X} u \vert^{2}}{4}
-
u
+
te^{u}
+
4e^{ u}\Vert \beta_{0} + \overline{\partial} \eta \Vert^{2}
\right)
\,dA.
\end{equation}
\\
Indeed, Theorem \ref{thm_A} 
is established in \cite{Huang_Lucia_Tarantello_2} by showing precisely that, for $t>0$ the functional $F_t$ admits a \underline{unique} critical point $(u_t , \eta_t)$ given by its global minimum.

On the other hand, for $t \leq 0$ (or equivalently $|c|\geq 1$) it is not at all clear weather the functional $F_t$ admits critical points, as we have an evident non-existence situation when $[\beta]=0$
(see Section \ref{Asymptotics} for details). 

Therefore for $t\leq0,$ the crucial issue is to identify the pairs: $(X,[\beta])$ (with $[\beta]\neq 0$) yielding to a functional $F_t$ (possibly unbounded from below) which admit critical points.

For the geometrically meaningful case: $t=0$ (i.e.  $|c|=1$) such a task demands  a detailed asymptotic analysis, since a critical point for $F_{t=0}$ exists (and is unique) only as "limit" of the pair $(u_t, \eta_t) \,\,\mbox{as }\, t\to 0^+,$ see Theorem 8  of \cite{Tar_2} (or Theorem \ref{thmprimobis} below).
However, such passage to the limit can be prevented by a "blow-up" situation involving the conformal factor $u_t.$ \\  
In this respect, we recall from \cite{Bryant} (see also 
 \cite{Rossman_Umehara_Yamada}, \cite{Umehara_Yamada}) that
(CMC) $1$-immersions of surfaces into the hyperbolic space $\mathbb{H}^{3}$ develop "smooth end", which in a compact setting manifest as "punctures" at finitely many points. As discussed in Remark \ref{conical}, typically such singularities  correspond to conical singularities 
and in our analysis they will occur naturally as blow-up points  of the function: $$ \xi_t:= -u_t + log(\|\alpha_t\|^2) \,\,\mbox{with} \,\, \a_t=e^{u_t}*_E(\beta_{0} + \bar{\partial}\eta_t)\,\mbox{and}\, t \to 0^+.$$

Indeed, in view of the (constrained) Gauss equation  in \eqref{system_of_equations_introbis}, we see that $\xi_t$ satisfies a Liouville type equation (see \eqref{1.19a} below) so that, according to \cite{Brezis_Merle}, \cite{Li_Harnack}, 
 \cite{Li_Shafrir},  \cite{Bartolucci_Tarantello} \cite{Tar_1} and together with \cite{Tar_2}, we find that the following alternative holds:\\ \\ 
\medskip
(1) either (\underline {Compactness}) : $\limsup_{t \to 0^+}max_{X} \ \xi_t < + \infty$ and  then $(u_{t},\eta_{t})\rightarrow (u_{0},\eta_{0})$ uniformly in $X$
as $t\rightarrow 0^{+}$; and $(u_{0},\eta_{0})$ is the unique critical point of $F_0$ corresponding to its global minimum;\\ \\
\medskip
(2) or (\underline{Blow-up}):  $\limsup_{t \to 0^+} max_{X} \ \xi_t = \liminf_{t \to 0^+} max_{X} \ \xi_t = + \infty, $ and along \underline{any} sequence $t_k \to 0^+,$ we have that    
 $ \xi_k := \xi_{t_k}$ admits a finite set $\mathcal{S}$ of \underline{blow-up points} (depending possibly on the sequence  $t_k$) and any blow up point $x \in \mathcal{S}$ satisfies: 
 $$\lim_{k \to + \infty} (\max_{B(x;r)} \ \xi_k) = + \infty, \, \, \text{ \  for all small \ } r > 0;$$   
with \underline{blow-up mass at $x$} given as follows:
\begin{equation}\label{blowup mass}
m_x:= \frac{1}{8\pi}\lim_{r \to 0^{+} }
\left(
\lim_{k \to +\infty }8 \int_{B(x;r)} \|\widehat\a_{t_{k}}\|^2 e^{\xi_k}dA\right) \in \mathbb{N}
\end{equation}
(quantization property of the blow-up mass) and
\begin{equation}\label{total mass} 
1 \leq \sum_{x \in \mathcal{S}}m_x \leq \mathfrak{g-1};
\end{equation}
see \cite{Tar_1}, \cite{Tar_2} and \cite{Tar_3} for details and also Theorem \ref{thm_blow_up_global_from_part_1} below for the precise statement.

Therefore, to obtain (CMC) $1$-immersions,  we must identify those pairs $(X,[\beta])$  for which "blow-up" can be rule out and the passage to the limit ensured.
On the other hand, by a simple scaling argument, for $[\beta]\neq 0,$  we see that: \\ 
 $(X,[\beta]) \implies \exists$ (CMC) $1$-immersion subject to the constraint \eqref{constraint} $\iff$   $(X,[\lambda \beta]) \implies \exists$ (CMC) $1$-immersion subject to the relative  \eqref{constraint}, $\, \forall\lambda \in \mathbb{C}\setminus \{  0 \}. $  \\ 
Thus, we are naturally lead to consider the projective space:

$\mathbb{P}(\mathcal{H}^{0,1}(X,E)) \simeq \mathbb{P}^{3\mathfrak{g}-4}$ of $\mathcal{H}^{0,1}(X,E)$ (with $E=T^{1,0}_{X}$),\\ 
where,
$\dim_{\mathbb{C}}\mathbb{P}(\mathcal{H}^{0,1}(X,E)) \geq 2$ for $\mathfrak{g}\geq 2.$\\ 
For given $[\beta] \in \mathcal{H}^{0,1}(X,E)\setminus \{ 0\},\,\mbox{ we let }\, [\beta]_{\mathbb{P}} \in  \mathbb{P}(\mathcal{H}^{0,1}(X,E))$ be the projective class identified by the class $[\beta],$ namely:
\begin{equation} \label{betap}
[\beta]_{\mathbb{P}} =\{\, [\lambda\beta] \in C_2(X), \quad \forall \, \lambda \in \mathbb{C} \setminus \{0\}\, \}\in \mathbb{P}(\mathcal{H}^{0,1}(X,E)).
\end{equation}

For genus $\mathfrak{g}=2,$ in \cite{Tar_2} and \cite{Tar-Tra_1}  the existence of (CMC) $1$-immersions was formulated in terms of the Kodaira map:

\begin{equation}\label{Kodaira_map_Intro}
\tau:X\longrightarrow \mathbb{P}(V^{*})\quad V=C_{2}(X)
\end{equation}
see section 12.1.3 of \cite{Donaldson_Book} for details. Here we 
recall only that, for genus $\mathfrak{g} =2$ the Kodaira map $\tau$ defines a two to one  holomorphic map of $X$ into the projective space: $\mathbb{P}(V^*)\simeq \mathbb{P}(\mathcal{H}^{0,1}(X,E)).$ 
Since the image 
$\tau(X)$ defines a complex curve  into $\mathbb{P}(\mathcal{H}^{0,1}(X,E)),$ we get that: $\tau(X) \subsetneq \mathbb{P}(\mathcal{H}^{0,1}(X,E)),$ and 
actually $\mathbb{P}(\mathcal{H}^{0,1}(X,E))\setminus \tau(X)$ defines a non empty  Zariski open (hence dense) subset of $\mathbb{P}(\mathcal{H}^{0,1}(X,E)).$

\begin{thmx}[\cite{Tar_2} ]\label{thmB}
Let $\mathfrak{g}=2$ and suppose that for the pair  
$(X,[\beta]) 
\in 
\mathcal{T}_{\mathfrak{g}}(X) \times (\mathcal{H}^{0,1}(X,E)\setminus \{  0 \} )$
$E=T^{1,0}_{X}$, blow-up occurs (in the sense of (2) above). 
Then, for given $t_k \to 0^{+},$ the sequence $\xi_{t_k}$ admits a unique blow-up point, i.e. 
$\mathcal{S} =\{x_0\}$ and
\begin{equation}\label{Kodaira1}
 [\beta]_{\mathbb{P}}=\tau(x_0).   
\end{equation}
\end{thmx}

Next, we recall that every Riemann surface of genus $\mathfrak{g}=2$ is hyperelliptic. Hence it admits a unique non trivial  bi-holomorphic hyperelliptic involution: 
\begin{eqnarray*}
    j: X\to X
\end{eqnarray*} such that $\tau \circ j = \tau,$ and the map $j$  has  exactly $2(\mathfrak{g}+1)=6$ (for $\mathfrak{g}=2$) distinct fixed points and they coincide with the \underline{Weierstrass points} of $X$  (see \cite{Miranda}, \cite{Griffiths_Harris}).\\
In \cite{Tar-Tra_1} we observed that the functional $F_0$ is equivariant with respect to bi-holomorphisms and more importantly, when $\mathfrak{g}=2$ the following holds:
\begin{thmx}[\cite{Tar-Tra_1} ]\label{thmC}
Under the assumptions of Theorem \ref{thmC}, the blow-up point $x_0$ must be the same along \underline{any} sequence $t_k \to 0^{+},$ and it must coincides with one (of the six) Weierstrass points of X, namely: $j(x_0)=x_0.$  
\end{thmx}

As a consequence of Theorem \ref{thmB} and Theorem \ref{thmC} the following holds:

\begin{thmx}[\cite{Tar_2} \cite{Tar-Tra_1}]\label{thmD}
If $\mathfrak{g}=2$ then to every  
$(X,[\beta]) 
\in 
\mathcal{T}_{\mathfrak{g}}(X) \times (\mathcal{H}^{0,1}(X,E)\setminus \{  0 \} )$
$E=T^{1,0}_{X}$, satisfying: 
\begin {equation} \label {weierstrass} [\beta]_{\mathbb{P}}\not \in \{\tau(q), \text{ with }\, q\in X: j(q)=q\}, \end{equation}
there  correspond a \underline{unique} (CMC) 1-immersion of $X$  
into a (germ of) hyperbolic $3$-manifold $N(\simeq S \times \R$), with pull back metric $g$ and $(2,0)$-part $\a$ of the second fundamental form $II_g$ satisfying:
\begin{eqnarray}\label{constraint1}
g=e^u g_X \quad \text{  and  } \quad *_E^{-1}(e^{-u}\a)\in [\b],
\end{eqnarray}where $*^{-1}_E$ the inverse of the Hodge star operator $*_E.$
\end{thmx}	

We expect that the condition \eqref{weierstrass} 
is sharp in the sense that, when it fails then blow up occurs ( as in (2) above) and the (CMC) $c$-immersions in Theorem \ref{thm_A} do not pass to the limit, as $|c| \to 1^-.$\\
For completeness we recall from \cite{Tar_2} that, if $\mathfrak{g}=2$ then the functional $F_0$ is always bounded from below, and when \eqref{weierstrass} holds then $F_0$ attains its global minimum at a point corresponding to its unique critical point.
\\
Our main contribution will be to establish  suitable extensions of Theorem \ref{thmC} and Theorem \ref{thmB} for higher genus.\\ 
In this case multiple blow-up points are possible, (with no particular restraint) carrying integral blow-up mass (see \eqref{blowup mass}), and 
so we are naturally lead to consider effective divisors over X.\\ To be more precise, let $X^{(\nu)}$ the symmetric product  of $\nu$-copies of $X$ modulo permutations, which defines a smooth complex manifold of dimension $\nu$ (see Section 2 of Chapter 2 in \cite{Griffiths_Harris}). 

As well known, $X^{(\nu)}$ can be identified with the space of non zero effective divisors of degree $\nu \geq  1$ on $X$. Indeed, a given $\nu$-ple representing an element in $X^{(\nu)}$ is formed by distinct points: $\{p_j\in X, \,\,j=1,\ldots, k \}$ appearing $n_j$ times in the $\nu$-ple, and so it identifies the divisor: $D=\sum_{j=1}^k n_j p_j$, having degree: $deg\,(D)=\sum_{j=1}^k n_j=\nu $ and support: $supp \,D := \{p_1,\ldots,p_k\}\subset X.$ 

In connections with holomorphic quadratic differentials, we recall that every $\a\in C_2(X)\setminus\{0\}$ admits $4(\mathfrak{g}-1)$ zeroes  counted with multiplicity, {(cf \cite{Miranda})}. Therefore the zero set of $\alpha$ identifies in a natural way an effective divisor: $\text{div}(\a) \in X^{(4(\mathfrak{g}-1))}.$ Thus, for $1\leq \nu\leq 4( \mathfrak{g}-1)$ and $D\in X^{(\nu)}$ we let, 
$$
Q(D)=\{\a\in C_2(X): \text{div}(\a) \geq D\}
$$ namely: $\a\in Q(D)\,\
 \iff \, \alpha$ vanishes at each point of $supp \,D$ with greater or equal multiplicity.
 Notice in particular that, for $D=x_0$ then,
 $$Q(x_0) =\{ \alpha \in C_2(X): \alpha (x_0)=0 \}.$$
and, by the very definition of the Kodaira map (see \cite{Donaldson_Book}) we have:
\begin{equation} \label{orthogonality} [\beta]_{\mathbb{P}}=\tau(x_0) \iff \int_X \b\wedge \,\a = \int_X \b_0\wedge \,\a=0, \quad \forall 
 \a\in Q(x_0).
\end{equation}
Clearly, the above "orthogonality" condition is independent of the chosen representative in the projective class $[\beta]_{\mathbb{P}},$ (and obviously on the chosen representative of the cohomology class $[\beta]$). 

On the ground of \eqref{Kodaira1}, for genus $\mathfrak{g}\geq 2$  we need to identify an appropriate version of the "orthogonality" condition \eqref{orthogonality}, which identifies an analytic sub-variety (possibly reducible) of higher dimension, but still \underline{properly} contained in $\mathbb{P}(\mathcal{H}^{0,1}(X,E)),$ and also provides the "natural" replacement (for higher genus) of the complex curve $\tau(X).$\\

In fact, our  main effort will be to establish the following:
\begin{thm}\label{mainth}
Let $\mathfrak{g}\geq 2$ and suppose that for the pair  
$(X,[\beta]) 
\in 
\mathcal{T}_{\mathfrak{g}}(X) \times (\mathcal{H}^{0,1}(X,E)\setminus \{  0 \} )$
$E=T^{1,0}_{X}$, blow-up occurs (in the sense of (2) above).\\
For a given sequence $t_k \to 0^{+},$ 
let $\mathcal{S}$ be the blow-up set of $\xi_{t_k}.$ \\  Then
 for every $x \in \mathcal{S}$ with blow-up mass $m_x \in \mathbb{N},$ there exists: 
\begin{equation} \label{vincoli} N_x \in \mathbb{N} \cup \{0\} :
0\leq N_x \leq 2(m_x -1),
\end{equation}
so that, for the divisor $D := \sum_{x \in \mathcal{S}} (N_x + 1) x$ the following holds:

\begin{equation}\label{ortogonale} \int_X \beta \wedge \alpha = 0, \quad \forall \alpha \in Q(D). \end{equation} \end{thm} 
\medskip

Since for genus $\mathfrak{g}=2$ we have:  $\mathcal{S}=\{x_0\} \,\,\mbox{and } \, m_{x_0}=1,$ we see that \eqref{ortogonale} is a direct extension of \eqref{orthogonality}. On the other hand, for genus $\mathfrak{g}=3,$  in  \cite{Tar-Tra_1} we were able to provide a shaper "orthogonality" condition, by showing that \eqref{ortogonale} actually holds with the choice:  $N_x + 1 =m_x, \,\, \forall x\in \mathcal{S}.$ Such an improvement was possible on the basis of a very accurate blow-up analysis, describing the asymptotic profile of  $\xi_{t_k} \,\,\mbox{as}\, k\to +\infty.$ \\ At the moment, it seems extremely  difficult (or even impossible) to extend to higher genus the description of the asymptotic blow up profile of  $\xi_{t_k}$ with the same accuracy as in \cite{Tar-Tra_1}, along the lines of \cite{Chen_Lin_1, Chen_Lin_2, Chen_Lin_3, Chen_Lin_4} and \cite{Wei_Zhang_1, Wei_Zhang_2, Wei_Zhang_3}. In addition, we face a new and delicate situation where blow-up occurs at a point of "collapsing" zeroes of 
$\alpha_{t_k}\,\,\mbox{as}\, k\to +\infty,$ where the phenomenon of "blow-up without concentration" (see Theorem \ref{thm_blow_up_global_from_part_1}) may manifest.  \\   Instead, to established Theorem \ref{mainth} we change completely point of view 
and relay on an appropriate approximation property (see Lemma \ref{approximation}) of "global" nature rather than  the "local" viewpoint of \cite{Tar_2} and \cite{Tar-Tra_1} focusing on  description of   $\xi_{t_k},$ around a blow-up point. 
\\    
Finally, in the spirit of \cite{Tar-Tra_1}, 
in Section \ref{Preliminaries} we discuss the crucial role played by the  constraints \eqref{vincoli} and \eqref{total mass}, which unable us to show that the "orthogonality" condition \eqref{ortogonale} identifies precisely a (possibly reducible) complex analytic sub-variety:\\ \\
$ \qquad \tilde{\Sigma}_{\mathfrak{g}} \subset \mathbb{P}(\mathcal{H}^{0,1}(X,E)):\quad \tilde{\Sigma}_{\mathfrak{g}=2}=\tau(X)\quad \mbox{and}\quad \mbox{dim} (\tilde{\Sigma}_{\mathfrak{g}}) \leq 2\mathfrak{g} - 3. $ \\ 

In particular, codim$ (\tilde{\Sigma}_{\mathfrak{g}})\geq {\mathfrak{g}}-1$ and therefore $\mathbb{P}(\mathcal{H}^{0,1}(X,E))\setminus \tilde{\Sigma}_{\mathfrak{g}}$ defines a non-empty Zariski open set (thus dense) in $\mathbb{P}(\mathcal{H}^{0,1}(X,E))$ where \underline{compactness} holds, see Corollary \ref{dimproj} for details. \\
Thus we may conclude: 
\begin{thm}\label{thm1}
For $\mathfrak{g}\geq 2$ there exist a closed complex analytic sub-variety:\\
$\tilde{\Sigma}_{\mathfrak{g}} \subset \mathbb{P}(\mathcal{H}^{0,1}(X,E)) $ with codim $ (\tilde{\Sigma}_{\mathfrak{g}})\geq {\mathfrak{g}}-1$ such that for every $(X,[\b])\in \mathcal{T}_{\mathfrak{g}}(X) \times \mathcal{H}^{0,1}(X,E)$ satisfying:
$$
[\b]\neq0 \text{ and } [\b]_\PP\notin \wt\S_{\mathfrak{g}}
$$ there  correspond a \underline{unique} (CMC) 1-immersion of $X$  
into a (germ of) hyperbolic $3$-manifold $N(\simeq S \times \R$) 
satisfying \eqref{constraint1}.
\end{thm}

Finally, we comment again about the case of genus $\mathfrak{g}=3$, where in view of the improvement in \eqref{ortogonale} mentioned above, it is possible to sharpen the conclusion of Theorem \ref{thm1} by replacing the whole $\tilde{\Sigma}_{\mathfrak{g}}$ with one of its  
irreducible components of dimension $2\mathfrak{g}-3=3$ (for $\mathfrak{g}=3$), we refer to \cite{Tar-Tra_1}  for details.

\section{Preliminaries}\label{Preliminaries}

Let us fix $X\in \mathcal{T}_{\mathfrak{g}}(S).$ Then we can consider $X$ as a  Riemann surface with (unique) hyperbolic metric $g_{X}$ and induced scalar product
$\langle\cdot,\cdot\rangle$, norm $\Vert \cdot \Vert$ and volume 
element $dA$. 

So, around any point $x \in X,$ 
we can introduce holomorphic $\{ z \}$-coordinates  centred at the origin (namely $x$ is mapped to $0$) and we denote: 
\begin{equation}\label{palla} 
\begin{array}{l}
B(x;r), \,\, \mbox{the geodesic ball centred at $x$ with radius } r>0,  \\   
\Omega_{r,x}  \,\, \mbox{the image of $B(x;r)$ in $\mathbb {C}$ with } \, 0\in \Omega_{r,x}, \\
B_{\delta} \mbox{ the disc in $\mathbb {C}$ of center the origin and radius $\delta>0$.}
\end{array} 
\end{equation}
\medskip

Since $X$ is compact, we can pick sufficiently small radii: $r>0$ and $\delta>0$ independent on $x,$ such that holomorphic $z$-coordinates at $x$ are well defined in $B(x;r)$ and $B_{\delta}  \subseteq \Omega_{r,x}, \,\, \forall x \in X.$

Thus, for $z=x+iy\in \Omega_{r,x},$ the local expression of the conformal and Riemannian structure of $X$ (around $x$) are given by:
\begin{equation}\label{coord} 
\begin{array}{l}
\partial=\frac{\partial}{\partial z}=
\frac{1}{2}(\frac{\partial}{\partial x} - i \frac{\partial}{\partial y})
\; \text{ and } \; 
\bar{\partial}=\frac{\partial}{\partial \bar{z}}
=
\frac{1}{2}(\frac{\partial}{\partial x} + i \frac{\partial}{\partial y}) \\
dz=dx + i dy,\; d \bar{z} = dx-idy,\\
g_{X}=e^{2u_X} dzd\bar{z}\quad
u_X \; \text{smooth,} \quad  
dA = \frac{i}{2}e^{2u_{X}} dz \wedge d \bar{z},\\

*dz=i d\bar{z},\; *d \bar{z}=-idz;\quad
*(d \bar{z}\otimes\frac{\partial}{\partial z})=\frac{-i}{2}e^{2u_{X}}(dz)^2
\end{array} 
\end{equation}

Furthermore, without loss of generality, we can consider the so called "normal" coordinates at $x,$ by assuming further that $u_X$ satisfies:
\begin{equation}\label{hypcoord} 
u_X(0) = |\nabla u_X(0)|=0. 
\end{equation}

In addition, in such local coordinates, the Laplace-Beltrami operator on $(X, g_X)$ can be expressed (locally) as follows:
$\Delta_{X}=4e^{-2u_X} \partial \bar{\partial}$ 
and in particular we have:
$4 \partial \bar{\partial} u_{X} = e^{2u_{X}}$ in $\Omega_{r,x}$ . 

In the sequel we also denote 
the \underline{flat} Laplacian by  
$\Delta=4 \partial \bar{\partial}.$ 

\vskip0.5cm
Throughout this paper, we let:
\begin{eqnarray}\label{E}
E=T_X^{1,0} \text{ the holomorphic tangent bundle of $X$} 
\end{eqnarray}
with dual:
$$
E^*=(T_{X}^{1,0})^{*}=K_{X} \text{ the canonical bundle of $X.$}
$$

The holomorphic line bundles $E$ and $E^*$ will be equipped with the complex structure induced by $X$ and with an hermitinan product induced by a given metric $g$ (typically conformal to the metric $g_X$) defined in $X$.  
Therefore, on sections and forms valued on $E$ or $E^*$ (or their tensor products) we have a well defined  $\bar\p$ operator. 
\\ \\
We introduce the spaces: 
$$ \begin{array}{l}
   A^{0}(E)=
\{\text{smooth sections of $X$ valued on $E$}\}, \\  A^{0,1}(X,E)
=
\{ \text{$(0,1)$-forms valued on $E$ }  \} = A^{0,1}(X,\mathbb{C}) \otimes E, \\
A^{1,0}(X,E^*)
= \{ \text{$(1,0)$-forms valued on $E^*$ }  \} = A^{1,0}(X,\mathbb{C}) \otimes E^*.

\end{array} $$
The elements in $A^{0,1}(X,E)$ are known as the Beltrami differentials. \\ 
Thus, on those spaces we have a well defined fiberwise hermitian product $\langle \cdot,\cdot \rangle_{g}$ and norm $\Vert \cdot \Vert_{g}.$ In the sequel, we shall drop the subscript $g$ in the hermitian product and norm induced by $g=g_X.$ \\
Hence, for $p\geq 1,$  we can define
the corresponding $L^{p}$-spaces:
\begin{align}
&
L^{p}(X,E)
=
\{ \eta:X\longrightarrow E 
\; : \; 
\Vert  \eta \Vert_{L^{p}}:=
(\int_{X}\Vert \eta \Vert^{p}dA)^{\frac{1}{p}} < +\infty
\},
\notag
\\
&
L^{p}(A^{0,1}(X,E))
=
\{ 
\beta \in A^{0,1}(X,E)
\; : \; 
\Vert \beta \Vert_{L^{p}}
:=
(\int_{X}\Vert \beta \Vert^{p}dA)^{\frac{1}{p}} < +\infty
\},
\notag
\end{align}	
which define Banach spaces equipped with the given norm: 
$\Vert \cdot  \Vert_{L^{p}}.$ 

Also for $p\geq 1,$ we have the Sobolev space:
\begin{equation}\label{W_1_p}
W^{1,p}(X,E)
=
\{ \eta \in L^{p}(X,E)
\; : \; 
\bar{\partial} \eta \in L^{p}(A^{0,1}(X,E))
\}, 
\end{equation}
defining a Banach space equipped with the norm:
$$
\Vert \eta \Vert_{W^{1,p}}
=
\Vert \eta \Vert_{L^{p}} + \Vert \bar{\partial}\eta \Vert_{L^{p}},
\; \forall \; \eta \in W^{1,p}(X,E).
$$
Incidentally, we recall that, for the holomorphic line bundle $E=T_{X}^{1,0}$ in \eqref{E},  
the following Poincar\'e inequality holds,
\begin{equation}\label{poincare}
\Vert \eta \Vert_{L^{p}} \leq C_{p}\Vert \bar{\partial}\eta \Vert_{L^{p}},
\; \forall \; \eta \in W^{1,p}(X,E).
\end{equation}
for suitable $C_p>0$, see \cite{Huang_Lucia_Tarantello_2}.

Letting:
$$\bar\partial
:
A^{0}(E)
\longrightarrow 
A^{0,1}(X,E).
$$  
we can define the $(0,1)$-Dolbeault cohomology group, as given by the following quotient space:
\begin{equation}\label{cohomologygroup}
\mathcal{H}^{0,1}(X,E)
=
A^{0,1}(X,E)/\overline{\partial}_E(A^{0}(E)),
\end{equation} 
where any Beltrami differential $\beta \in A^{0,1}(X,E)$ identifies the cohomology class: 

\begin{equation}\label{classbeta}
[\beta]=\{ \beta + \bar{\partial}\eta,
\; \forall \; \eta \in A^{0}(E) \}\in \mathcal{H}^{0,1}(X,E).
\end{equation}
Since, by Dolbeault decomposition, any $\beta \in A^{0,1}(X,E)$
can be uniquely decomposed as follows:
$$
\beta = \beta_{0} + \bar{\partial}\eta
\; \text{ with  $\beta_{0}$ \underline{harmonic} (with respect to $g_{X}$) and } \; 
 \eta \in A^{0}(E),
$$
we see that every class  $[\beta] \in \mathcal{H}^{0,1}(X,E)$ is uniquely identified by its harmonic representative $\beta_{0} \in [\beta].$

Next, we consider the \underline{wedge product}:
$$\wedge: A^{0,1}(X,E) \times A^{1,0}(X,E^{*})\to A^{1,1}(X,\mathbb{C})$$ 
(see \cite{Griffiths_Harris}), and obtain the bilinear form: 
\begin{equation}\label{wedge_product_map}
\begin{split}
 A^{1,0}(X,E^{*}) \times A^{0,1}(X,E) \longrightarrow \mathbb{C}
\;:\;
(\alpha,\beta)
\longrightarrow 
\int_{X}   \beta \wedge \alpha,
\end{split}
\end{equation}
which, by Serre duality (see \cite{Voisin}),  is non-degenerate and induces the isomorphism:
\begin{equation}\label{isomorphism_A_1_0_To_A_0_1}
A^{1,0}(X,E^{*})
\simeq 
(A^{0,1}(X,E))^{*}.
\end{equation}
Furthermore, we can express the (metric dependent) \underline{isomophism}  between
$A^{1,0}(X,E^{*})$ and $ A^{0,1}(X,E),$ in terms of the anti-linear Hodge * operator (with respect to the metric $g_X$) acting on forms.
To be more precise, for $x \in X$ we consider the usual Hodge * operator defined on real valued forms and its anti-linear extension:  $$ *_{x}: \,  [T_{x}(X)^{*}]^{0,1} \to  [T_{x}(X)^{*}]^{1,0}\quad \mbox{( anti-linear Hodge  operator)}$$ see \cite{Wells}; where we recall that a  map $L : V \to W$ between complex vector spaces is called anti-linear, if it is $\mathbb{R}$ linear and $L(iv) = -i L(v),  \ \forall v \in V.$

Also recall that, for   $\varphi \in [T_{x}(X)^{*}]^{0,1}\,\
 \mbox { and }  e, f \in E_x$ we have the following  anti-linear isomorphisms: $$\sharp_x : E_x \to E^*_x\quad  \mbox{ and } \quad  *_{x}: \,  [T_{x}(X)^{*}]^{0,1} \otimes E_x \to  [T_{x}(X)^{*}]^{1,0}\otimes E^*_x $$ defined as follows: 

$\sharp_x(e)(f) = <f,e>_x,$ and $*_{x}(\varphi \otimes e) = *_{x}(\varphi) \otimes \sharp_x(e),$  for every $x \in X.$\\
\medskip
In this way, we can extend the Hodge * operator  on forms, valued on $E$ and $E^*$ respectively, as follows:  
\begin{equation}\label{2.8*hodge_operator_intro}
*_{E} : A^{0,1}(X,E) \longrightarrow  A^{1,0}(X,E^{*}),
\end{equation}
where, for given $\beta \in A^{0,1}(X,E)$  the form
 $*_{E}\beta \in A^{1,0}(X,E^*)$ is uniquely identified by the  condition: 
$$ 
\xi \wedge *_{E}\beta  
=
\langle \xi,\beta \rangle \,dA,
\quad \forall \; \xi \in A^{0,1}(X,E).$$ 
As well known, the Hodge operator $*_{E}$ defines an isometry (with inverse $*_E^{-1}$) and more precisely there holds: 
\begin{equation}\label{antiso}
< *_{E}(\beta_1), *_{E}(\beta_2)> = \overline{ < \beta_1, \beta_2>} =  < \beta_2, \beta_1> \,\,\, \mbox{ and }  *_E(i \beta) = -i *_E(\beta). 
\end{equation}
In local holomorphic coordinates, for $*=*_E$ we have:
 \begin{equation}\label{**2}
  * dz = i d \bar{z}, \,\,\,   *d \bar{z} = - i d z, \quad  \sharp(\frac{\partial}{\partial z}) = \frac{e^{2u_{X}}}{2} dz, \end{equation} 
and in particular, 
 \begin{equation}\label{**bis}
 *(d \bar{z}\otimes\frac{\partial}{\partial z})=\frac{-i}{2}e^{2u_{X}}(dz)^2
 \end{equation}

Moreover, for the local expression 
of the (fiberwise) norm (induced by $g_X$) of sections and forms there holds:
\begin{equation}\label{norme}
\begin{split}
\eta = \eta(z) (\frac{\partial}{\partial z}) \quad \implies \quad \Vert \eta \Vert
= \; 
\vert \eta(z) \vert \frac{e^{u_{X}(z)}}{\sqrt{2}},
\;\; 
\eta\in A^{0}(E); 
\\ 
\beta = \beta(z) (d \bar{z} \otimes\frac{\partial}{\partial z}) \quad \implies \quad \Vert \beta  \Vert = \vert \beta(z) \vert,  
\;\;
\beta \in A^{0,1}(X,E); \\
\a= h(z)(dz)^{2} \,\, \implies \, \Vert \alpha \Vert =2 \vert h  \vert e^{-2 u_X}, \,\,\, \a \in A^{1,0}(X,E^*).
\end{split}
\end{equation}

Next, we let $C_{2}(X)$ be the complex linear space of holomorphic quadratic differentials, or equivalently:
$$ \begin{array}{l}
C_{2}(X)
=\{ \alpha \in A^{1,0}(X,E^{*})
\; : \;
\bar\partial
 \alpha=0 \}. \end {array}$$
 Hence, for $\alpha \in C_2(X)$ we have the following local expression at $x\in X$: $$\a= h(z)(dz)^{2}\quad \mbox{with } h\,\,\mbox{holomorphic in} \,\, \Omega_{r,x}.$$
In this way it is clear what we mean by a zero of $\alpha \in C_2(X)$ and corresponding multiplicity, since those notions are independent on the chosen coordinate system.
In particular, if $q$ is a zero of $\alpha$ with multiplicity $n$, then 
in  local $z$-coordinates at $q$ we have:
 $$\alpha = z^{n} \psi(z)(dz)^{2}, \quad \mbox{ $\psi$ holomorphic and never vanishing
  in $\Omega_{r,q}.$}$$ 
Moreover we have: $\Vert \alpha \Vert=2 \vert z \vert^{n} \vert \psi(z) \vert e^{-2 u_X}$ 
and  $\partial \bar{\partial} \ln \vert \psi \vert^{2}=0$ in $\Omega_{r,q}$, a property we shall use in the sequel.\\  
By the Riemann-Roch Theorem we know that,
\begin{equation}\label{dim}
\dim_{\mathbb{C}}C_{2}(X)= 3(\mathfrak{g}-1),
\end{equation}
and more importantly we have:
\begin{eqnarray}\label{2.03}
\text{any } \, \alpha\in C_2(X)\setminus\{ 0 \} \,\, \text {admits}\, \,  4(\mathfrak{g}-1) \text{ zeroes counted with multiplicity, } 
\end{eqnarray}
(see \cite{Miranda} and \cite{Narasimhan}).

By Stokes theorem, we see that actually the bilinear form \eqref{wedge_product_map} is 
well defined and non degenerate
when restricted on the space:
$C_{2}(X) \times \mathcal{H}^{0,1}(X,E)$, 
and it induces the isomorphism:
\begin{equation}\label{C_kappa_X_isometry}
\begin{split}
C_{2}(X) \simeq (\mathcal{H}^{0,1}(X,E))^{*}.
\end{split}
\end{equation}
Furthermore, for 
$[\beta]\in \mathcal{H}^{0,1}(X,E)$ with harmonic representative $\beta_{0} \in [\beta]$, we have: $*_E\beta_0\in C_2(X),$ and 
in analogy to \eqref{2.8*hodge_operator_intro}, we obtain the isomorphism:
\begin{equation}\label{isoC_2}
\mathcal{H}^{0,1}(X,E) \longrightarrow C_{2}(X)
:
[\beta] \longrightarrow  *_{E}\beta_{0}.
\end{equation}

In addition, to the class $[\beta]$ we can associate also an (unique) element in $(C_{2}(X))^{*}$ defined as follows:
\begin{equation}\label{dual}
C_{2}(X)\longrightarrow \mathbb{C}
:
\alpha \longrightarrow \int_{X} \beta_{0} \wedge \alpha = 
\int_{X} (\beta_{0} +\bar{\partial}\eta) \wedge \alpha .
\end{equation}
 Consequently, the space $\mathcal{H}^{0,1}(X,E)$ or equivalently the space of harmonic Beltrami differentials (with respect to $g_X$) can be identified with the dual space $(C_2(X))^*$  .\\

At this point, (in view of \eqref{dim})  we have 
a well-known parametrization of $T^*(\mathcal{T}_{\mathfrak{g}}(S)),$ the cotangent bundle of 
$\mathcal{T}_{\mathfrak{g}}(S),$   given by the pairs:
$$(X,\alpha) \in \mathcal{T}_{\mathfrak{g}}(X)\times C_{2}(X),$$
see e.g. \cite{Jost} for details.
In view of the isomorphism \eqref{C_kappa_X_isometry}, 
we derive also a local trivialization of  $T(\mathcal{T}_{\mathfrak{g}}(S))$ the  tangent bundle of $\mathcal{T}_{\mathfrak{g}}(S)$, parametrized by the pairs: 
$$(X,[\beta]) \in \mathcal{T}_{\mathfrak{g}}(S) \times \mathcal{H}^{0,1}(X,E).$$
Finally, by virtue of \eqref{dim}, we have the equivalence of  all norms in $C_2(X),$ and it is usual
(by recalling the Weil-Patterson form \cite{Jost}) to consider the following  $L^{2}$-hermitian product and $L^{2}$-norm (in terms of $g_X$):
\begin{equation}\label{norm}
\begin{array}{l}
\langle \alpha_1 , \alpha_2\rangle_{L^{2}} = \int_{X} \langle  \alpha_1, \alpha_2\rangle dA, \quad \alpha_1 , \alpha_2 \in C_{2}(X)\\
\\
\Vert \alpha \Vert_{L^{2}}
:=
(\int_{X} \langle  \alpha, \alpha\rangle dA)^{\frac{1}{2}},
\; \text{ for } \; 
\alpha \in C_{2}(X).
\end{array}
\end{equation}
Then, for a given \underline{orthonormal basis} in
$C_{2}(X)$: 
\begin{equation}\label{basis}
\{s_{1},\ldots,s_M \} \subset C_{2}(X)
\; \text{ with } \; 
M=3(\mathfrak{g}-1)
\text{ and }\;
\int_{X}\langle s_{j} , s_{k}\rangle dA
=
\delta_{j,k}
\end{equation}
($\delta_{j,k}$ the Kronecker symbols)
and $\alpha \in C_{2}(X)$ 
we write:
$$\alpha=\sum_{j=1}^{N}b_{j}s_{j}, \;b_{j}\in \mathbb{C}
\; \text{ and compute } \; 
\Vert \alpha \Vert_{L^{2}}^{2}=\Vert \beta_{0} \Vert^{2}_{L^{2}}
=
\sum_{j=1}^{N}\vert b_{j} \vert^{2}.
$$ 
Clearly, any closed and bounded subsets of $C_2(X)$ is compact. Thus, for example, if $\alpha_{k}\in C_2(X)$ satisfies 
$\Vert \alpha_{k} \Vert_{L^{2}}=1$ then it admits a convergent subsequence 
$\alpha_{k_{n}}\longrightarrow \alpha_{0}  \in C_{2}(X),\,\, n\to +\infty;$ 
with 
$\Vert \alpha_{0} \Vert_{L^{2}}=1$. 
\medskip

\vskip 0.2cm
Next, we recall some well-known facts about divisors and some useful consequences of the Riemann-Roch theorem 
that will
be useful in the sequel, see e.g. \cite{Jost}, \cite{Miranda} and \cite{Narasimhan}.
\vskip 0.1cm
For given $\nu\in \mathbb{N},$ let $X^{\nu} =X\times \cdots\times X$ be the Cartesian product of $\nu$-copies of $X$ and consider the quotient of $X^\nu$ modulo the action of the symmetric  group $S_\nu$ by permutations, namely:
\begin{eqnarray*}
    X^{(\nu)}:=X^{\nu} /\{ \text{permutations}\}.
\end{eqnarray*} Then $X^{(\nu)}$ with the quotient topology, defines 
   a compact complex manifold of dimension $\nu,$ where the natural projection: 
   $$\pi : X^\nu \to X^{(\nu)}$$ is surjactive and holomorphic. 
\medskip  
   
   As already mentioned in the Introduction, the manifold $X^{(\nu)}$ is identified with the space of effective divisors in $X$ of degree $\nu.$
More precisely, a non zero  \underline{effective divisor $D$ of degree $\nu$} on $X$ is given by the formal expression:
\begin{eqnarray}\label{divis}
    D=\sum_{j=1}^n n_j x_j,  \quad\, \text{ with } \, x_j\in X \quad \,  n_j \in \mathbb{N} \,  \quad \, \text{ and } \quad \sum_{j=1}^n n_j=\nu > 0,
\end{eqnarray}
and it identifies the unique element of $X^{(\nu)}$ associated with the $\nu$-ple containing $n_j$ copies of the point $x_j$, $j=1,\ldots, n$. 

We denote by: $\text{supp } D= \{x_1,\ldots, x_n\}$ the support of the divisor $D$ in \eqref{divis} (formed by finitely many \underline{distinct} points of $X$) while  the positive integer $n_j\in \mathbb{N}$ defines the multiplicity of $x_j \in \text{supp }D$ for $j=1,\ldots,n,$ and the integer $\nu$ defines the  degree of $D,$ namely: $deg (D):=\sum_{j=1}^n n_j=\nu$.

\begin{remark}\label{topoldiv}
In view of the identification of $X^{(\nu)}$ with the space of non zero  effective divisors of degree $\nu,$  we have the following notion of convergence for divisors: if $\mathbf{x_k} \in X^{\nu}$ is such that: $D_k = \pi (\mathbf{x_k}) \in X^{(\nu)}$ then $D_k \to D \mbox{ in } \, X^{(\nu)}$ as $k \to + \infty,$ if and only if for every subsequence of $\mathbf{x_k}$ converging to some $\mathbf{x} \in X^\nu,$ we have: $\pi(\mathbf{x}) = D.$ \end{remark}

To any  $\alpha \in C_2(X)\setminus \{0 \}$ we associate the so called  \underline{zero divisor} of $\alpha,$ denoted by $div(\alpha),$ and defined as follows:
$$
div(\alpha)=\sum_{j=1}^n n_j q_j,
$$
where $\{q_1,\ldots, q_n\}$ are the \underline{distinct} zeroes of $\alpha$, and $n_j$ is the multiplicity of the zero $q_j$, for $j=1,\ldots,n.$ 
From \eqref{2.03}, we have: $div (\alpha) \in X^{4(\textit{g}-1).}$

\vskip0.2cm
Given $D_1\in X^{(\nu_1)}$ and $D_2\in X^{(\nu_2)}$ we set $D_1\leq D_2$ if $\nu_1\leq \nu_2$, $\text{supp} D_1\subset \text{supp} D_2$ and the multiplicity at $x\in \text{supp} D_1$ is smaller or equal than the multiplicity of $x\in \text{supp } D_2$. 

For an effective divisor $D$, we define:
$$
    Q (D)=\{\alpha\in C_2(X): \text{div}(\alpha)\geq D\},
$$
and for the trivial divisor $D=0$ we set Q(0)=$C_2(X).$
As a direct consequence of the Riemann-Roch theorem we have,
\begin{eqnarray}\label{2.07*}
\begin{array}{l}
\mbox{if}\,\,1\leq \nu < 2(\mathfrak{g}-1) \, \mbox{and deg}(D)=\nu \implies    \dim_{\mathbb{C}} Q(D)= 3(\mathfrak{g}-1)-\nu\\
\mbox{where } \ \, 3(\mathfrak{g}-1) =\dim_\mathbb{C}(C_2(X)). 
\end{array} 
\end{eqnarray} 
Consequently, if $D_1\in X^{(\nu_1)}$ and $D_2\in X^{(\nu_2)}$  satisfy: $D_1\leq D_2$ and $1\leq \nu_1<\nu_2<2(\mathfrak{g}-1)$ then
$\exists \, \a\in Q(D_1) \text{ but } \a\notin Q(D_2).$
In particular,
\begin{corollary}\label{alphaspecial}
If $0< deg(D) < 2(\mathfrak{g}-1)$ then for every $x_0 \in supp\,D$ there exists $\alpha \in Q(D-x_0)$ but $\alpha \notin Q(D).$
\end{corollary} 
Moreover we have:
\begin{lemma}\label{divalfa}
The map $\alpha \to  div(\alpha)$ from $C_2(X) \setminus \{0\}$ to $X^{(4(g-1))}$ is continuous.
\end{lemma}
\begin{proof}
Let $\alpha_k \to \alpha \in C_2(X) \setminus\{0\},$ we set $D_k := div(\alpha_k)$ and $D_0 = div(\alpha).$ Up to a subsequence, we may assume that $D_k:=\pi(\mathbf{x_k}) \to \bar{D}$ with suitable $\mathbf{x_k} \to \mathbf{x}$ in $X^{4(g-1)}$ as $k \to + \infty$ and $\pi(\mathbf{x}) = \bar{D}.$ We need to show that $\bar{D} = D_0.$ 
Let $\nu:=4(g-1)$ and set  $\mathbf{x_k} = (\bar{x}_{1,k}, \ldots,\bar{x}_{\nu_k}).$ We know that, for $j=1,\ldots, \nu,$ the point $\bar{x}_{j,k}$ is a zero of $\alpha_k$ which appears in the $\nu-ple \,\, \mathbf{x_k}$ according to its multiplicity, . 
By setting: $\mathbf{x} = (\bar{x}_{1}, \ldots,\bar{x}_{\nu}),$ then for $j=1,\ldots, \nu$ the point $\bar{x}_{j}$ is a zero of $\alpha$  and it appears in the $\nu-ple\,\,\mathbf{x}$ with a multiplicity given by the sum of the multiplicities of all components of $\mathbf{x_k}$ which converge to $\bar{x}_j.$   Therefore, the sum of the multiplicities of each $x_j$  add up to $\nu,$ and consequently there are no other zeros of $\alpha$ except those in $\mathbf{x}.$ Thus we conclude that necessarily: $div( \alpha) = \bar{D},$ as claimed.
\end{proof}

Next, for an $M$ dimensional complex vector space $V,$ we denote by $\mathbb{P}(V)$ the projective space relative to $V,$ with $dim( \mathbb{P}(V)) = M-1.$\\ Moreover, for $1\leq l \leq M$
we let $Gr(l,V)$ be the Grassmanian of $l$ dimensional complex subspaces of $V.$ We know that $Gr(l,V)$ is  a compact complex manifold of complex dimension $l(M-l)$, see  e.g. \cite{Tar-Tra_1}.

In our context, we are naturally interested to the vector space $V=C_{2}(X)$ and so  $M=3(\textit{g} -1)$ and $\mathbb{P}(V^*)\simeq \mathbb{P}(\mathcal{H}^{0,1}(X,E))$ (recall  \eqref{dual}). By virtue of \eqref{2.07*}
we have:
\begin{equation}\label{trapmap}
\begin{array}{l}
\mbox{if} \qquad 1\leq \nu < 2(\textit{g}-1)\,\,\mbox{and}\,\, l=3(\textit{g} -1)-\nu, 
\mbox{then it is well defined the map:}\\ \\
\quad \Psi: X^{(\nu)} \to Gr(l,V)  \quad D \to Q(D)\\ 
\end{array}
\end{equation}
and actually it is shown in  \cite{Tar-Tra_1} that $\Psi$ is \underline{holomorphic} (see \cite{Tar-Tra_1} for details). Thus we derive: $$D_k \to D \,\, \mbox{ in } \, X^{(\nu)} \implies Q(D_k) \to Q(D),\, k \to + \infty,$$ 
$$\mbox{if}\,\alpha_k \in Q(D_k): \,\,\alpha_k \to \alpha \mbox{ in } C_{2}(X)\,\, \implies \alpha \in Q(D).$$ 
where in the second statement above we have used  Lemma \ref{divalfa}, and actually also its reverse statement holds as follows:
\begin{lemma}\label{approximation}
Let $D_k \to D \,\, \mbox{ in } \, X^{(\nu)} \,\,\mbox{as }\, k \to + \infty \,\mbox{ with }\, 0< \nu < 2(g-1).$ \\
If $\alpha \in Q(D)$ then there exists $\alpha_k \in Q(D_k)$ such  that  $\alpha_k \to \alpha,$  as $ k \to + \infty.$   
\end{lemma}

\begin{proof}
According to the hermitian product \eqref{norm} for $V=C_{2}(X)$ and with respect to a given orthonormal basis as in \eqref{basis}, we can identify $C_{2}(X)$ with $\mathbb{C}^M \,\, \mbox{ and }\,M=3(\textit{g} -1). $ In this way, each space $W \in Gr(\mathbb{C}^M,l)$ can be seen as the kernel of a  $(M-l) \times M$ complex matrix of rank 
$M-l.$ Moreover, for $Gr(\mathbb{C}^M,l)$ we obtain local charts in $\mathbb{C}^{(M-l)l}$ by considering  all $(M-l) \times M$ complex matrices with a fixed  $(M-l) \times (M-l)$ sub-matrix equal to the identity. In this way, around $W=Q(D)$  (hence $l= 3(\mathfrak{g}-1)-\nu$)  we obtain local holomorphic coordinates in $\mathbb{C}^{(M-l)l},$ and
since $Q(D_k) \to Q(D)$ as $k \to + \infty,$ we have that $Q(D_k)$ lies in such coordinate neighborhood of $Q(D)$ for $k$ large.  At this point, we can identify in a canonical way a basis $E_k = \{s_{1,k}, \ldots,s_{l,k} \}$ for $Q(D_k)$ among all solutions of the corresponding homogeneous linear system. Analogously, we adopt the same canonical choice to obtain a basis $E = \{s_1, \ldots,s_l\}$ for $Q(D).$  As a consequence, we find that: $s_{j,k} \to s_j \,\,\mbox{as }\, k \to + \infty.$  
Hence if $\alpha \in Q(D)$ then we can write: $\alpha = \sum_{j=1}^{l} \lambda_j s_j,$ and it suffices to take:  $\alpha_{k}=\sum_{j=1}^{l} \lambda_j s_{j,k} \in Q(D_k)$
to find: $\alpha_k \to \alpha,$ as $k \to + \infty,$ as claimed. 
\end{proof}
\medskip
\medskip
For later use, we point out the following well known convergence properties for holomorphic functions:
\begin{lemma}\label{C_k}
Let $n_1, \ldots, n_s$ be positive integers and let $a_k(z) = (z-z_{1,k})^{n_1}(z-z_{2,k})^{n_2} \ldots (z- z_{s,k})^{n_s}C_k(z),$ with $z_{j,k} \in B_{\delta} : z_{j,k} \to 0,  \forall j=\{1 \ldots,s\}$ and $C_k$ holomorphic in $B_\delta$. If  $a_k \to a, $ uniformly on compact sets of $B_{\delta}$\, as $k \to + \infty,$ \, then for $n = \sum_{j=\{1 \ldots,s\}} n_j$ we have: $a(z) = z^n C(z)$ with $C_k \to C$ uniformly on compact sets of $B_{\delta},$  $C$ holomorphic in $B_{\delta}$ and $C(0) =\frac{1}{n!}\frac{\partial^n a}{\partial z^n}(0).$
\end{lemma}

\begin{proof}
 
 Although well known we sketch the proof of the above statement for completeness. Clearly we need to  prove only that, $C_k \to C$ as $k \to + \infty,$ uniformly on compact sets of $B_{\delta}.$ To this purpose we use simply an induction argument first with respect to the index $s$ and then over the index $n$. Hence, we need only to treat the case where $s =1$ and $n=n_1= 1.$ Thus we set: $ z_{1,k}=z_k.$ Since $a_k$ converges to $a$ uniformly on compact sets of $B_{\delta}$  we immediately see that $a_k(z_k) =0= a(0).$ Moreover $a$ is holomorphic on $B_{\delta}$ and $a(z)=zC(z)$ for suitable $C$ holomorphic in $B_{\delta}.$ Also, the complex derivative $a_k'$ converges uniformly to $a'$ on compact sets of $B_{\delta}.$   Thus,  $C_k(z) =  \frac{a_k(z) - a_k(z_k)}{z-z_k} = \int_{0}^{1} a_k'(t z +(1-t)z_k)dt$ and 
$C(z) = \frac{a(z) - a(0)}{z} = \int_{0}^{1} a'(tz)dt,$ by which we immediately derive uniform  convergence of $C_k$ to $C$ on compact sets of $B_{\delta}.$ Finally, it suffices to consider the Taylor expansion of $a$ at $z=0$ to find:  
$C(0) = \frac{1}{n!}\frac{\partial^n a}{\partial z^n}(0),$ as claimed. \end{proof}
\vskip.1cm

Recall the holomorphic map $\Psi$ in \eqref{trapmap}. By the proper mapping theorem (see Chapter II 8.2 of \cite{Demailly})   we know that the image $\Psi(X^{(\nu)})$ is a closed analytic sub-variety of $Gr(M-\nu, V)$ of dimension at most $\nu,$ ( recall: $ V=C_2(X)\,\,\mbox{and }\,\, M = 3(\mathfrak{g}-1)).$ \\
Moreover, for $[\beta] \in \mathcal{H}^{0,1}(X,E)\setminus \{ 0\}\,\mbox{ we recall that }\, [\beta]_{\mathbb{P}} \in  \mathbb{P}(\mathcal{H}^{0,1}(X,E))$ denotes the projective class identified by $[\beta],$ namely:
$$
[\beta]_{\mathbb{P}} =\{\, [\lambda\beta] \in C_2(X) \,\, \forall \lambda \in \mathbb{C} \setminus \{0\} \}.$$ In view of the duality in \eqref{dual},  it makes sense to define:\\ \\ $[\beta]_{|Q(D)}\equiv 0 \iff
\int_X \beta \wedge \alpha = \int_X \beta_0 \wedge \alpha =0, \quad \forall \alpha \in Q(D)$\\ \\ 
with harmonic $\beta_0 \in [\beta].$ 
Clearly: $[\beta]_{|Q(D)}\equiv 0 \iff [\lambda\beta]_{|Q(D)}\equiv 0\,\, \forall \lambda \in \mathbb{C} \setminus \{0\},$ and so such an "orthogonality"  condition is well defined in terms of the projective class: $[\beta]_{\mathbb{P}}.$ Therefore, we can let: 
$$  [\beta]_{\mathbb{P}}( Q(D)) \equiv 0 \iff [\beta]_{|Q(D)}\equiv 0,\,\, \mbox{for any } [\beta]\in [\beta]_{\mathbb{P}} $$
and consider the set: 
\begin{equation}\label{Sigma} \Sigma_\nu := \{ ([\beta]_{\mathbb{P}}, Q(D)) : [\beta]_{\mathbb{P}}( Q(D)) \equiv 0 \} \subseteq \, \mathbb{P}(V^*) \times \Psi(X^{(\nu)}) \end{equation}
(recall $\mathbb{P}(\mathcal{H}^{0,1}(X,E)) \simeq \mathbb{P}(V^*))$ which defines a closed analytic sub-variety of  $\mathbb{P}(V^*) \times \Psi(X^{(\nu)}),$ see \cite{Tar-Tra_1} for details.
Let,
$$p_1 :  \mathbb{P}(V^*) \times \Psi(X^{(\nu)}) \to \mathbb{P}(V^*)\,\,\mbox { and }\,   p_2 :  \mathbb{P}(V^*) \times \Psi(X^{(\nu)}) \to \Psi(X^{(\nu)})$$ be the canonical projections. From \cite{Tar-Tra_1} we know also that, $p_2(\Sigma_{\nu}) =  \Psi(X^{(\nu)})$ and  the fibers of the map $p_2$ restricted to $\Sigma_{\nu}$ are $\nu-1$ dimensional projective spaces. 
Therefore,  if $Y$ is a closed irreducible analytic sub-variety of  $\Psi(X^{(\nu)})$ then $\Sigma_\nu \cap p_2^{-1}(Y)$ is also an irreducible closed analytic sub-variety of  dimension: dim$(Y)$+ $\nu - 1.$

Now, we fix $k\in \mathbb{N}$ such that: $1 \leq k \leq \mathfrak{g}-1$ and  consider the \underline{finite} set: 
$$\mathfrak{I}_k = \left\{ \begin{array}{l}({\bf{m}},{\bf{N}})\in \mathbb{N}^k\times\mathbb{N}^k\, \mbox{ with } {\bf{m}} = (m_1, \ldots,m_k),\, {\bf{N}} = (N_1,\ldots,N_k): \\ 
 m := \sum_{j \in \{1, \ldots,k \}} m_j \leq \mathfrak{g}-1 \ \mbox{and} \ 1 \leq N_j \leq 2m_j-1, \, \,   1 \leq j \leq k. \end{array} \right. $$
For given $1 \leq k \leq \mathfrak{g}-1$ and $ ({\bf{m}},{\bf{N}}) \in \mathfrak{I}_k$ we set $N := \sum_{ j \in \{1, \ldots,k\}}N_j,$ 
so that: $1\leq N\leq 2m-k,$ and we define the set of corresponding effective divisors: 
\begin{equation}\label{Ykmn} Y_{ (k,{\bf{m}},{\bf{N}})} = \{ N_1 x_1 + N_2 x_2 + \ldots N_k x_k \in X^{(N)}   \ \mbox{ with $(x_1, \ldots, x_k) \in X^k $}  \},\end{equation}
note for instance that,
$Y_{(k,(1,1 \ldots,1), (1,1 \ldots 1))} = X^{(k)}.$ The following holds:
\begin{lemma}\label{dimk}

For given $k\in \mathbb{N}:$ $1 \leq k \leq g-1$ and $ ({\bf{m}},{\bf{N}}) \in \mathfrak{I}_k$ we have:\\
(i) the set $\Psi(Y_{ (k,{\bf{m}},{\bf{N}})})$ is an irreducible closed analytic sub-variety of $\Psi(X^{(N)})$ of dimension at most $k,$ \\ 
(ii) the set $\Sigma_{(k,{\bf{m}},{\bf{N}})} := \Sigma_N \cap p_2^{-1}(\Psi(Y_{ (k,{\bf{m}},{\bf{N}})}))$ is an irreducible closed analytic sub-varieties of  $\mathbb{P}(V^*) \times \Psi(X^{(\nu)})$ of  dimension at most $k + N -1, $\\  
(iii) the set $$\tilde{\Sigma}_{(k,\bf{m},\bf{N})}:= p_1(\Sigma_{(k,\bf{m},\bf{N})}) $$ is an irreducible closed analytic sub-variety in $\mathbb{P}(V^*)$ of dimension at most $k+N-1$ and $\tilde{\Sigma}_{(k,\bf{m},\bf{N})}\subsetneq \mathbb{P}(V^*).$ 
\end{lemma}

\begin{proof}

The set $Y_{(k,{\bf{m}},{\bf{N}})}$ is the image of the following holomorphic proper map: $$X^k \to X^{(N)}, \quad (x_1, \ldots,x_k) \to N_1x_1 + \ldots + N_k x_k.$$ Again, by the proper mapping  theorem  we see that, $Y_{k,{\bf{m}},{\bf{N}}}$ defines a closed irreducible analytic sub-variety of dimension at most $k,$ and consequently also the set $\Psi(Y_{k,{\bf{m}},{\bf{N}}}),$ admits the same properties. This implies that,  $\Sigma_{(k,{\bf{m}},{\bf{N}})} := \Sigma_N \cap p_2^{-1}(\Psi(Y_{ (k,{\bf{m}},{\bf{N}})}))$ is an irreducible closed analytic sub-variety of  $\mathbb{P}(V^*) \times \Psi(X^{(\nu)})$ of  dimension at most $k + N -1.$ Thus, by using once more the proper mapping theorem, we get in addition that,  
$p_1(\Sigma_{(k,\bf{m},\bf{N})}) \subseteq \mathbb{P}(V^*)$ 
is also an irreducible closed analytic subset of dimension at most $k+N-1.$ Since, $1 \leq k+N-1\leq 2g-3 < dim \mathbb{P}(V^*)=3g-4,$ we may conclude that: $\tilde{\Sigma}_{(k,\bf{m},\bf{N})}\subsetneq \mathbb{P}(V^*),$ as claimed.

\end{proof}

\begin{remark}\label{sigma}
Clearly, if for a  Beltrami differential $\beta\neq 0$ we have:\\
$\int_X \beta \wedge \alpha =0 \quad \forall \alpha \in Q(D),$ with a divisor $D\in Y_{ (k,{\bf{m}},{\bf{N}})}$ then  $[\beta]_{\mathbb{P}} \in \tilde{\Sigma}_{(k,\bf{m},\bf{N})}. $ 
\end{remark}

In this way, by varying $k \in \{ 1,...,\mathfrak{g}-1 \}$ and $ ({\bf{m}},{\bf{N}}) \in {\mathfrak{I}}_k$ we obtain a finite family of irreducible analytic sub-varieties of $\mathbb{P}(V^*),$ which may or may not be distinct and in any event the following holds: 
\begin{corollary}\label{dimproj}
The set $$ \tilde{\Sigma}_{\mathfrak{g}}:= \bigcup_{1 \leq k \leq g-1} \bigcup_{({\bf{m}},{\bf{N}}) \in \mathfrak{I}_k} \tilde{\Sigma}_{(k,\bf{m},\bf{N})}$$
defines a  closed complex analytic  sub-variety of $\mathbb{P}(V^*)$ (possibly reducible) of dimension at most $2\mathfrak{g}-3,$ and so of codimension at least $g-1$ in $\mathbb{P}(V^*).$ In particular $\tilde{\Sigma}_{\mathfrak{g}} \subsetneq \mathbb{P}(V^*)$ and $\mathbb{P}(V^*)\setminus{\tilde{\Sigma}_{\mathfrak{g}}}$ defines a non empty Zariski open set.
\end{corollary}

\begin{proof}
It follows immediately    from Lemma \ref{dimk} \end{proof}
Please recall that a non-empty Zariski open set in the (connected) complex manifold $\mathbb{P}(V^*)$ is open and dense in the usual topology of $\mathbb{P}(V^*),$ and more precisely it has full mass with respect to any smooth volume form on $\mathbb{P}(V^*).$

\begin{remark}
It can be proved that 
$ \tilde{\Sigma}_{(\mathfrak{g}-1,\bf{1},\bf{1})}= p_1(\Sigma_{(\mathfrak{g}-1,(1,1\ldots,1),(1,1,\ldots,1))})$ has dimension exactly $2g-3.$ \end{remark} 
\section{Asymptotics and the proof of the Main Results.}\label{Asymptotics}

Given the  pair: $(X,[\beta]) \in \mathcal{T}_{\mathfrak{g}}(S) \times \mathcal{H}^{0,1}(X,E),$ let 
$\beta_{0} \in [ \beta ]$ be the harmonic representative of the class $[\beta].$ As observed in the Introduction that, if $g$
(the pullback metric  on $X$) and $\alpha$ (the $(2,0)$-part of the second fundamental form $II_g$) satisfy:
\begin{equation}\label{vincolo1}
g=e^u g_X\quad \text{ and } \, \a=e^u*_E(\b_0+\bar\p \eta), \text{ with suitable }\eta \in A^0(E),
\end{equation}
then the pair $(u,\a)$ is a solution of the Gauss Codazzi equations \eqref{0.8a} \eqref{0.6a} subject to the constraint:
$*_E^{-1}(e^{-u}\a) \in [\b] \,\, \iff \,\,(u,\eta)$ satisfies:
\begin{equation}\label{system_of_equations_intro}
\left\{
\begin{matrix*}[l]
\Delta_{X} u +2 -2te^{u} -8e^{u}\Vert \beta_{0}+\overline{\partial}\eta \Vert^{2} =0  &  \;\text{ in }\;  &  X,  \\
\overline{\partial}(e^{u}*_{E}(\beta_{0}+\overline{\partial}\eta))=0, &  \;\text{}\;  &   \\ 
\end{matrix*}
\right.
\end{equation} with $t=1-c^2$.\\
We shall refer to \eqref{system_of_equations_intro} as the ``constrained" Gauss-Codazzi equations by the pair $(X,[\beta])$.

In particular from \eqref{system_of_equations_intro},  we see that the Beltrami differential 
$\beta_{0}+\bar{\partial}\eta \in [\beta]$ is harmonic with respect to the metric 
$h=e^{\frac{u}2}g_{X}.$ With this point of view, it follows that the system \eqref{system_of_equations_intro}
can be formulated in terms of  Hitchin's self-duality equations (see \cite{Hitchin})
with respect to a suitable nilpotent $SL(2,\mathbb{C})$ Higgs bundle, we refer to \cite{Alessandrini_Li_Sanders}, \cite{Huang_Lucia_Tarantello_2} for details, see also \cite{Li} for related issue concerning  minimal immersions. 

As a consequence of Hitchin's selfduality theory \cite{Hitchin}, we would obtain readily existence and uniqueness for \eqref{system_of_equations_intro} provided the given Higgs bundle is stable, a property which is hard to check in our context.  

On the other hand, it is easy to check that (weak) solutions of \eqref{system_of_equations_intro} correspond to critical points of the following \underline{Donaldson functional} (in the terminology of \cite{Goncalves_Uhlenbeck})
\begin{equation}\label{F_t2}
F_{t}(u,\eta)
= 
\int_{X}
\left(\frac{\vert \nabla_{X} u \vert^{2}}{4}
-
u
+
te^{u}
+
4e^{ u}\Vert \beta_{0} + \overline{\partial} \eta \Vert^{2}
\right)
\,dA
, 
\end{equation}
$t\in \R$, with ``natural" (convex) domain:
$$
\begin{array}{l}
\Lambda
=
\left\{  
(u,\eta)\in H^{1}(X) \times W^{1,2}(X,E)
\; : \; 
\int_{X}e^{u}
\Vert \beta_{0} + \bar{\partial}\eta \Vert^{2}\,dA < \infty
\right\} 
\end{array}
$$
where $H^{1}(X)$ is the usual Sobolev spaces of function defined on $X$
and $W^{1,2}(X,E)$ is the Sobolev space of sections of $E$ (see \eqref{W_1_p}).  

We refer to \cite{Tar_2} for a detailed discussion about the Gateaux differentiability  of $F_t$ along "smooth" directions and the corresponding notion of "weak" critical point and related regularity.

For $t>0$ the functional $F_{t}$ is clearly bounded from below in $\Lambda,$ 
and as anticipated in \cite{Goncalves_Uhlenbeck}, we know from \cite{Huang_Lucia_Tarantello_2} that  if $t>0$ then  $F_{t}$ admits a unique (smooth) critical point $(u_t, \eta_t)$ corresponding to its global minimum in $\Lambda.$ 
Theorem \ref{thm_A} in the Introduction is a direct consequence of this fact. 
On the other hand, for $t \leq 0$ it can happen that the functional $F_t$ admits no critical points in $\Lambda,$ i.e. \eqref{system_of_equations_intro} admits no solutions. Indeed, this is the case if we take $[\beta]=0$ (i.e. $\beta_0=0$) 
where the second equation in \eqref{system_of_equations_intro} implies that necessarily: $\bar{\partial}\eta =0$ (or equivalently: $\eta=0,$ see \eqref{poincare} and \cite{Huang_Lucia_Tarantello_2} ) and as a consequence the first equation \eqref{system_of_equations_intro} cannot admit a solutions for $t \leq 0.$

Thus when $t \leq 0,$ we need to identify the pairs $(X,[\beta])$ which insure the existence of critical points for $F_{t}.$ This is a delicate task even for $t=0.$
Indeed, from \cite{Tar_2} we know about the continuous dependence of the pair $(u_t,\eta_t)$ with respect to the parameter $t\in(0,+\infty)$, and letting:
$$
F_{0}(u,\eta)
=
\int_{X}
\left(
\frac{1}{4}\vert \nabla_{X} u \vert^{2}
-
u
+
4e^{u}\Vert \beta_{0} + \overline{\partial} \eta \Vert^{2}
\right)
\,dA,
$$
we know that the existence and uniqueness of a (smooth) critical point for $F_0$ is actually equivalent to the continuous extension of $(u_t,\eta_t)$ at $t=0.$ More precisely, the following holds (see Theorem 8 in \cite{Tar_2}):
\begin{thmx}[Theorem 8 \cite{Tar_2}]\label{thmprimobis}
If $(u_{0},\eta_{0})$ is a solution for the system
\eqref{system_of_equations_intro} with $t=0$, then
\begin{enumerate}[label=(\roman*)]
\item $(u_{t},\eta_{t})\rightarrow (u_{0},\eta_{0})$ uniformly in $C^{\infty}(X), $
as $t\rightarrow 0^{+}$;
\item
$F_{0}$ is bounded from below in $\Lambda$ and attains its global minimum at $(u_{0},\eta_{0})$ which defines its only critical point.\\ 
Hence,
$(u_{0},\eta_{0})$ is the \underline{only} solution of \eqref{system_of_equations_intro} with $t=0.$

\end{enumerate}
\end{thmx}
 
Just to clarify the above result note that, 
for $[\beta]=0$ and $t>0$ we have: $u_{t}=\ln \frac{1}{t}\rightarrow +\infty$, $\eta_{t}=0$ and
$F_{t}(u_{t},\eta_{t})\rightarrow -\infty$, 
as $t\rightarrow 0^{+},$ and indeed for $t=0$ the system \eqref{system_of_equations_intro} admits no solutions, consistently with Theorem \ref{thmprimobis}.
\\   

Therefore, to identify possible critical points for $F_{0}$, 
we must investigate  when the pair $(u_{t},\eta_{t})$ 
survives the passage to the limit, as $t\to 0^{+}$.  

For this purpose, we recall from \cite{Tar_2} that the map:
\begin{eqnarray*}
    t\to 4\int_X e^{u_t} \|\b_0+\bar \p \eta_t\|^2 dA=4\pi(\mathfrak{g}-1)-t\int_X e^{u_t} dA
\end{eqnarray*} is decreasing in $(0,+\infty)$ (see Lemma 3.6 of \cite{Tar_2}), and so it is well defined the value:
\begin{eqnarray}\label{rho}
    \rho([\b])=\rho([\b_0]):=4 \lim_{t\to0^+} \int_X e^{u_t} \|\b_0+\bar\p \eta_t\|^2 dA =4\lim_{t\to0^+} \int_X e^{\xi_t} \|\wh\a_t\|^2 dA,
\end{eqnarray}and,
\begin{eqnarray}\label{3.1*}
\rho([\b])\in [0, 4\pi(\mathfrak{g}-1)] \text{ and } \rho([\b])=0 \Longleftrightarrow [\b]=0.
\end{eqnarray}
Furthermore, in case $F_0$ is bounded from below then it was shown in \cite{Tar_2} that necessarily: 
$\rho ([\b])= 4\pi(\mathfrak{g}-1) \;
 \text{ i.e. } \, \lim_{t\to0^+} t\int_X e^{u_t}=0.$
\\

For given $[\beta] \in \mathcal{H}^{0,1}(X,E) \setminus \{ 0 \} $
with harmonic representative $\beta_{0} \in [\beta]\neq 0,$ and $t > 0,$ we set:
$$
\begin{array}{l}
\beta_{t}=\beta_{0}+\overline{\partial}\eta_{t} \in A^{0,1}(X,E)
\; \text{ and } \; 
\alpha_{t}=e^{ u_{t}}*_{E}\beta_{t} \in C_{2}(X) \setminus \{ 0 \} \; \\ \\
D_t=\text{div}(\a_t) \quad\text{ and } \; supp\, D_{t}=: Z_{t}.
\end{array}
$$  
Namely, $D_t$ is the zero divisor of $\a_t$ and its support $Z_{t}$ is the finite set of \underline{distinct} zeroes of $\alpha_{t}$, 
whose multiplicities adds up to $4(\mathfrak{g}-1)$ (see \eqref{2.03}). 
In terms of the fiberwise norm for $\alpha_{t}$ we have:
$\Vert \alpha_{t} \Vert(q)=\Vert \alpha_{t} \Vert_{E^{*}}(q)>0,
\; \forall \; q\in X\setminus Z_{t}.$ 

Moreover we let,
$$
s_{t}\in \R \; : \; 
e^{ s_{t}}
=
\Vert \alpha_{t} \Vert_{L^{2}}^{2}
\; \text{ and } \; 
\hat{\alpha}_t
=
\frac{\alpha_{t}}{\Vert \alpha_{t} \Vert_{L^{2}}}
=
e^{-\frac{s_{t}}{2}}\alpha_{t},
$$
where $\Vert \alpha_t \Vert_{L^{2}}$ is the $L^{2}$-norm of $\alpha_t\in C_{2}(X)$ (see \eqref{norm}) and we have: $\text{div}(\widehat\a_t)=\text{div}(\a_t)=D_t$.

In order to control
the asymptotic behavior of $(u_{t},\eta_{t})$, 
as $t\rightarrow 0^{+}$, 
we shall need to account for possible blow-up phenomena (cf. \cite{Brezis_Merle}) of the function,
\begin{equation}\label{u_to_xi}
\xi_{t}:=-u_{t}+s_{t},
\end{equation} satisfying the Liouville-type equation:
\begin{eqnarray}\label{1.19a}
-\Delta_{X} \xi_t=8\|\widehat\a_t\|^2 e^{\xi_t} -f_t \text{ in } X,
\end{eqnarray}with $f_t= 2(1-t e^{u_t})$ satisfying: $0\leq f_t \leq 2$ in $X$.
\vskip0.5cm

As already mentioned in the Introduction, by combining Theorem 3 of \cite{Tar_1} and Theorem \ref{thmprimobis} above, we know that either "compactness" or "blow-up" holds for $\xi_t$ along any sequence $t_k\to 0^+. $ This fact will be described in details in Theorem \ref{thm_blow_up_global_from_part_1} below, and for this purpose we let :

\begin{equation*}\label{6.11_prime}
u_{t}
=
w_{t} 
+
d_{t}
,
\; \text{ with } \; 
\int_{X} w_{t}dA=0
\; \text{ and } \;  
d_{t}=\fint_{X}u_{t}dA.
\end{equation*}

The following easy bounds where derived in \cite{Tar_2}:
\begin{lemma} For any $t>0$ the following holds:
\begin{equation}\label{property_v} \begin{array}{l}
\forall \, q  \in [ 1,2 )  \; \exists \; 
C_{q}>0 \; : \; \Vert w_{t} \Vert_{W^{1,q}(X)}\leq C_{q} \quad \text{ and } \, \; te^{d_t}\leq 1,\\
\\
w_{t}\leq C \; \text{ in  } \, X, \quad s_t \leq d_t + C \,\,  \mbox{and}\,\, \int_{X}e^{- u_{t}}dA \geq
C \fint_{X} \Vert \beta_{0} \Vert^{2}dA,\\ \\
\mbox{for a suitable constant $C>0.$}
\end{array} \end{equation}

\end{lemma}	
\begin{proof}
See Lemma 3.7 and Remark 3.1 of \cite{Tar_2}.    
\end{proof}

In view of the estimates in \eqref{property_v}, along a (positive) sequence
$t_{k}\longrightarrow 0^{+}$,  
for $$d_{k}:=d_{t_{k}}, \,\; u_{k}=u_{t_{k}}, \,\; w_k:=w_{t_k},$$  
we may assume that,
\begin{equation*}
w_{k} \longrightarrow w_{0}
\; \text{ and } \; 
e^{w_{k}}\longrightarrow e^{w_{0}} 
\; \text{ pointwise and in   } \; 
L^{p }(X),\; 
\end{equation*}
\begin{equation}\label{6.20_prime}
t_{k}e^{d_{k}} \longrightarrow \mu \geq 0
\; \text{ and so } \;
t_{k}e^{u_{k}}
\longrightarrow 
\mu e^{w_{0}}
\; \text{ pointwise and in  } \; 
L^{p}(X), 
\end{equation}
for any $p>1,$ and as $k\longrightarrow +\infty$.

In addition, it follows from  
\eqref{2.03}
that, for  $k$ sufficiently large and possibly along a subsequence,
we can find a suitable integer $N \in \{1, ..., 4(\mathfrak{g}-1) \}$ such that, for $\wh\a_k:=\wh \alpha_{t_{k}}\in C_{2}(X) \setminus \{ 0 \}$ we have:
\begin{eqnarray}\label{zeroalfakappa}
    \text{div}(\wh \a_k)=\sum_{j=1}^N n_j q_{j,k}\;\, \text{and} \,\, \sum_{j=1}^{N} n_{j}=4(\mathfrak{g}-1),
\end{eqnarray}
where $q_{j,k}$ is a zero of $\wh \a_k$ with multiplicity $n_{j}\in \mathbb{N}, \quad j\in \{1,\ldots,N \}$. 

Moreover, up to subsequences,  $\text{ as } k\to +\infty,$ we may let,
$$
\wh\a_k\to \wh\a_0, \;\; q_{j,k}\longrightarrow q_{j},
\; \text{ with } \; \wh\a_0(q_{j})=0, \;\, j \in \{ 1,\ldots,N \}.
$$
Although the zeroes of $\wh\a_0$ in $\{q_1,\ldots, q_N\}$ may \underline{not} be distinct, we know  however that the sum of the multiplicities carried by each  $q_{j}$'s for $j \in \{ 1,\ldots,N \},$ adds up to the value: $4(\mathfrak{g}-1)$, and
therefore $\wh\a_0$ cannot vanish anywhere else.
We set,
\begin{equation}\label{zetakappa}
Z^{(k)}:= supp (\text{div}(\wh \a_k)) =\{ q_{1,k},\ldots,q_{N,k} \} \quad \mbox{and }\,\,Z^{(0)}:= supp (\text{div}(\wh \a_0)), 
\end{equation}
so that, $Z^{(0)}$ collects the distinct zeroes in  $\{q_1,\ldots, q_N\}$ of $\wh\alpha_0.$ 
In other words, we have:
$$\text{div}(\wh \a_0)=\sum_{q \in Z^{(0)}}n_{q}q \quad \mbox{and }\, \sum_{q \in Z^{(0)}}n_{q} = 4(\mathfrak{g}-1).  $$
Moreover, by setting:
$$I_q=\{j \in \{ 1,\ldots,N \}: q_j=q \}, \quad \mbox{for}\,\, q \in Z^{(0)}, $$
we can identify the set $Z_0$ of elements in $Z^{(0)}$ (possibly empty) corresponding to the limit points of distinct zeroes in $Z^{(k)},$ as given by:
\begin{eqnarray}\label{Z_0}
    Z_0:=\{q\in Z^{(0)}: |I_q|\geq2 \}, \quad \mbox{($|I_q|$=cardinality of $I_q$)}
\end{eqnarray}
and we shall refer to the elements in  $Z_0$ as the zeroes of $\wh\alpha_0$ of \underline{"collapsing"} type.
\\
 
We define:
$$
\xi_{k}=- (u_{t_{k}}-s_{t_{k}})
$$
and let,
\begin{equation}\label{6.22}
R_{k} = 8 \Vert \wh{\alpha}_{k} \Vert^{2}
\end{equation}
so that $R_k$ and $|\nabla_{X} R_k|$ are uniformly bounded in $X$. Moreover, we have:
\begin{equation}\label{6.23}
-\Delta_{X} \xi_{k}
=
R_{k}e^{\xi_{k}}-f_{k}
\; \text{ in  } \; 
X \qquad \text{ and } \,\,
\int_{X}R_{k}e^{\xi_{k}}\, \leq C
\end{equation}
with 
$f_{k}
: =
2 (1-t_{k}e^{u_{t_{k}}})>0
$ satisfying: 
\begin{equation}\label{f_0}
\begin{array}{l}
f_{k}
\rightarrow
f_{0}
=:
2 (1-\mu e^{w_{0}})
\, \text{ in } \, 
L^{p}(X),\; p>1;
\\
\int_{X}f_{0}=2 \rho([\beta])>0 
\; \text{ for } \; [\beta]\neq 0, \, \text{ (recall \eqref{3.1*}).}
\end{array} 
\end{equation} 
 
Also notice that,
\begin{equation}\label{form_of_R_k}
R_{k}(z)
=
8 
\prod_{j=1}^{N}(d_{g_{X}}(z,q_{j,k}))^{2 n_{j}}G_{k}(z)
,\;
z \in X,
 \end{equation}
where $d_{g_{X}}$ defines the distance relative to the metric $g_X$. From (\ref{6.22}) we have:
$$G_{k}\in C^{1}(X) \;\; 0 < a \leq G_{k} \leq b	 \mbox{ and  \ }  \vert \nabla_{X} G_{k} \vert \leq A \mbox{ \  in   }  X, $$
 with suitable positive constants $a,b$ and $A$.
Hence (by taking a subsequence if necessary) we may assume that, 
\begin{equation}\label{6.27}
G_{k}\rightarrow G_{0}
\; \text{ in   } \; 
C^{0}(X) 
\; \text{ and so } \; 
R_{k}\rightarrow R_{0}
\; \text{ in } \; 
C^{0}(X),
\; \text{ as } \; 
k\rightarrow +\infty,
\end{equation}
with
\begin{equation}\label{3.57a}
R_{0}(z)
=
8 
\prod_{q\in Z^{(0)}}  (d_{g_{X}}(z,q))^{2n_{q}}G_{0}(z)
=
8 \Vert \hat{\alpha}_0 \Vert^{2}.
\end{equation}
 
With the information above, we can apply  Theorem 3 of \cite{Tar_1}, which extends to the case of blow-up point at a zero point of "collapsing" type, the analysis of \cite{Brezis_Merle}, \cite{Li_Harnack}, 
 \cite{Bartolucci_Tarantello},  \cite{Lin_Tarantello} and 
\cite{Lee_Lin_Tarantello_Yang},
to deduce the following alternatives
about the asymptotic behavior of $\xi_{k}$:
\begin{thmx}[Theorem 3 \cite{Tar_1}]\label{thm_blow_up_global_from_part_1} 
Let $\xi_{k}$ satisfy \eqref{6.23} and  assume
\eqref{f_0}-\eqref{3.57a}. Then
one of the following alternatives holds (along a subsequence):

\

\noindent
(i) \quad (compactness)\; : \; 
$\xi_{k}\longrightarrow \xi_{0}$ in $C^{2}(X)$ with
\begin{equation}\label{thm_referenced_compactness}
-\Delta_{X} \xi_{0}
=
R_{0}e^{\xi_{0}}-f_{0}
,\;
\; \text{ in   } \; 
X
\end{equation}

\noindent
(ii)  \quad (blow-up)\; : \;There exists a \underline{finite} blow-up set 
$$
\mathcal{S}
=
\{ x\in X
\; : \; 
\; \exists \; 
x_{k}\rightarrow x
\; \text{ and } \; 
\xi_{k}(x_{k})
\rightarrow 
+ \infty,
\; \text{ as } \; 
k\rightarrow +\infty
 \} 
$$

\quad \,
such that, 
$\xi_{k}$ 
is uniformly bounded from above on compact sets of $X\setminus \mathcal{S}$

\quad \;
and,
as $k\rightarrow +\infty.$ Furthermore,
\begin{enumerate}[label=(\roman*)]
\item[a)] either (blow-up with concentration)\;:\;
\end{enumerate} 
$$
\begin{array}{l}
\xi_{k}\longrightarrow -\infty
\; \text{ uniformly on compact sets of  } \;
X\setminus \mathcal{S},
\\ \\
R_{k}e^{\xi_{k}}
\rightharpoonup
\sum_{x\in \mathcal{S}}\sigma(x)\delta_{x}
\; \text{ weakly in the sense of measures, } \end{array} $$ \\
where
\begin{equation}\label{sigma_q}
\begin{array}{l} 
\sigma(x)
:=
\lim_{r \to 0^{+} }
\left(
\lim_{k \to +\infty }8 \int_{B(x;r)} \|\widehat\a_{t_{k}}\|^2 e^{\xi_k}dA\right)\in 8\pi\mathbb{N},\\ \\
x\not \in Z^{(0)}\,\implies \sigma(x)=8\pi \,\mbox{and } 
\,
x=  z_{j} \in Z^{(0)}\setminus Z_{0}\,\implies \sigma(x) = 8\pi(1+n_{j}). 
\end{array}
\end{equation}
Such an alternative always holds when $\mathcal{S} \setminus Z_{0} \neq \emptyset$.
\item[b)] or (blow-up without concentration)\;:\; 
\begin{align}
&
\xi_{k}\rightarrow \xi_{0}
\; \text{ in } \; C^{2}_{loc}(X\setminus \mathcal{S}),
\label{3.58c}
\\
&
R_{k}e^{\xi_{k}}
\rightharpoonup
R_{0}e^{\xi_{0}}
+
\sum_{x\in \mathcal{S}}\sigma(x)\delta_{x}
\; \text{ weakly in the sense of measures, } \;
\notag
\\
&
\mbox{with }\sigma(x)\in 8\pi \mathbb{N}, \;\; \mathcal{S} \subset Z_{0} \,\mbox{and }\, \xi_{0}\,\mbox{satysfying:}
\notag
\end{align}	
\begin{equation*}
\quad\quad\quad
-\Delta_{X} \xi_{0}
=
R_{0}e^{\xi_{0}}
+
\sum_{x\in \mathcal{S}}\sigma(x)\delta_{x}-f_{0}
\; \text{ in  } \; 
X.
\end{equation*}

\end{thmx}
We point out that \eqref{sigma_q} is based on \cite{Chen_Li_1} and \cite{Prajapat_Tarantello}. While we refer the reader to \cite{Tar_1}, \cite{Suzuki_Ohtusuka}, \cite{Lin_Tarantello}, 
\cite{Lee_Lin_Tarantello_Yang}, \cite{Lee_Lin_Wei_Yang} and \cite{Lee_Lin_Yang_Zhang} for a more detailed discussion about blow up at a zero of "collapsing" type in connection with the phenomenon of "blow up without concentration". 
\begin{remark} \label{rho1}
\medskip
If alternative (i) holds then (by Theorem \ref{thmprimobis}) $F_{0}$ is bounded from below and
$(u_{t},\eta_{t})\rightarrow (u_{0},\eta_{0})$ in $\Lambda$
as $t\rightarrow 0^{+}$,
with $(u_{0},\eta_{0})$ the global minimum and only critical point of $F_{0}$ and  
$\rho([\beta])=4\pi(\mathfrak{g}-1),$ see \cite{Tar_2} for details. Hence in this case, the (CMC) $c$-immersions given by Theorem \ref{thm_A} pass to the limit as $c\to 1^-$ and yield to the desired (CMC) $1$-immersion.
\end{remark}
On the contrary, we observe the following:
\begin{remark} \label{conical}
\medskip
When alternative (ii)-b) holds then we may consider the family of (scaled) (CMC)-immersions of $X$ into hyperbolic 3-manifolds relative to the Cauchy data: $(u_t-s_t,\hat{\alpha}_t).$ Then, by taking into account \eqref{essek} below, along the  sequence $t=t_k\to 0^+$  as $ k \to +\infty,$ we obtain a "limiting" configuration (in the sense of Gromov-Hausdorff) given by a (CMC)-immersion of $X$ into a hyperbolic cone-manifold of dimension 3 (\cite{KS}). Roughly speaking, 3-dimensional hyperbolic cone-manifolds are characterized by the presence of conical singularities along lines. They were introduced by Krasnov-Schlenker in \cite{KS} to obtain a Hamiltonian description of 3D-gravity. 

 In particular in this case, the induced metric on $X$ admits fintely many conical singularities (at blow-up points corresponding to zeroes of "collapsing" type for $\hat{\alpha_0}$) with conical angles an integral multiple of $8\pi$ (and not the usual $4\pi$ due to our normalization of the conformal factor, see e.g. \cite{Mazzeo_Zhu_1, Mazzeo_Zhu_2}, \cite{Mondello_Panov_1, Mondello_Panov_2}). 
 
This situation is likely to captures the analogue in the compact setting of the "smooth ends" present in (CMC) 1-immersions into $\mathbb{H}^3$ as described by Bryant in \cite{Bryant}.

\end{remark}

Therefore, in the following, we shall investigate the sequence $\xi_k$ in case of blow-up (in the sense of alternative (ii) of Theorem \ref{thm_blow_up_global_from_part_1}) with the purpose to establish the orthogonality relation (\ref{ortogonale}) for the given class $[\b]\in \H^{0,1}(X,E)\setminus\{0\}$.  \\
Let,  
\begin{equation}\label{blowup-set}
\mathcal{S} \neq \emptyset\,\, \mbox{be the (finite) blow-up set of $\xi_k,$}
\end{equation}
so that,
\begin{equation}\label{blowupmass} 
m_{x}:= \frac{1}{8\pi}\sigma(x) \in 
\mathbb{N}\quad\mbox{(the \underline{blow-up mass} at}\,\, x \in \mathcal{S})
\end{equation} \\ satisfies: \begin{equation}\label{massatotale}
1\leq \sum_{x \in \mathcal{S}} m_{x}\leq \mathfrak{g}-1,     
\end{equation} \\
(recall \eqref{rho} and \eqref{3.1*}). 
\\
 As already observed in \cite{Tar_2}, and in view of \eqref{antiso},  we find:
\begin{lemma}\label{lemmaexpand} For any $r>0$ sufficiently small and for every $\alpha \in C_2(X)$ we have:
\begin{equation}\label{betawedge} 
\begin{array}{l}
\int_X \beta \wedge \alpha = \int_X \beta_0 \wedge \alpha = \\ e^{\frac{-s_k}{2}}\left( \sum_{x \in \mathcal{S}} \int_{B(x; r)} e^{\xi_k}<*^{-1} \hat{\alpha_k},*^{-1} \alpha> d A  \right) +o(1) \\ =  e^{\frac{-s_k}{2}}\left( \sum_{x \in \mathcal{S}} \int_{B(x; r)} e^{\xi_k}<\alpha , \hat{\alpha_k}> d A  \right) +o(1) \\ \mbox{as \, $k \to + \infty.$ } \end{array} \end{equation}
\eqref{antiso}.
\end{lemma}

\begin{proof} 
By formula (3.75) in \cite{Tar_2} and by using  \eqref{antiso} we find:
\begin{eqnarray*}
&& \int_X \beta_0 \wedge \alpha = e^{\frac{-s_k}{2}} \int_X e^{\xi_k} <*^{-1} \hat{\alpha_k}, *^{-1} \alpha> d A =\\
&&\; = e^{\frac{-s_k}{2}} \int_X e^{\xi_k} <\alpha, \hat{\alpha_k}> d A =\\
&&\; = e^{\frac{-s_k}{2}} \left( \sum_{l=1}^{m} \int_{B(x_{l}; r)}  e^{\xi_k} <\alpha, \hat{\alpha_k}> dA \right.\\
&& \; + \left. \sum_{l=1}^{m} \int_{X \setminus \bigcup_{l=1}^{m}  B(x_{l}; r)}  e^{\xi_k} <\alpha, \hat{\alpha_k}> dA     \right).
\end{eqnarray*}
Since
$$
c_{k}=F_{t_{k}}(u_{k},\eta_{k})
=
\frac{1}{4}\int_{X}\vert \nabla w_{k} \vert^{2}dA
-
4\pi(\mathfrak{g}-1)d_{k}+O(1),
$$
we see that, in case of blow-up, necessarily:
$d_{k}\rightarrow +\infty \, \, \text{as} \,\,  k\to +\infty. $

Moreover,  $\Vert w_{k} \Vert_{L^{2}(X)}\leq C$ and 
we can use elliptic estimates to derive that the sequence $|w_{k}|$  is uniformly bounded away from the blow-up set $\mathcal{S}$ and therefore, 
\begin{equation}\label{xi_k_versus_d_k_minus_s_k}
\xi_{k}
=-  (d_{k}-s_{k})+O(1)
\; \text{ on compact sets of  } \;
X\setminus \mathcal{S}.
\end{equation}
We can use the last estimate in \eqref{property_v} together with \eqref{xi_k_versus_d_k_minus_s_k} and find a suitable constant $C=C_r > 0 )$ to obtain:
$$  e^{\frac{-s_k}{2}}\left\vert \sum_{l=1}^{m} \int_{X \setminus \bigcup_{l=1}^{m}  B(x_{l}; r)}  \!\!\!\!\!\!\!\!\!\!\!\!\!\!\!\! e^{\xi_k} <\alpha, \hat{\alpha}_k> dA    \right\vert   \leq C_r e^{\frac{-s_k}{2} - (d_k- s_k)} \leq C_r e^{\frac{-d_k}{2}} \to 0 $$ as $k \to + \infty,$
and \eqref{betawedge} is established.
\end{proof}

\begin{remark} \label{rho1bis}
\medskip
In view of \eqref{xi_k_versus_d_k_minus_s_k}, we may conclude that,
\begin{equation}\label{essek}
\mbox{``blow-up with concentration"  
occurs if and only if}\,\, d_{k}-s_{k}\longrightarrow +\infty.
\end{equation}
\end{remark}
\vskip.1cm
In order to establish Theorem \ref{mainth}, our effort in the following will be to estimate each of the integral terms in \eqref{betawedge}. \vskip.1cm 
To this purpose, we can fix $r>0$ sufficiently small, so that 
for any $x\in \mathcal{S}$ we can consider
local holomorphic z-coordinates at $x \in \mathcal{S}$ defined in $B(x;r),$ (as specified in \eqref{coord} and
 \eqref{hypcoord})  and write:  
\begin{equation}\label{local_coord}
\begin{array}{l}
\hat{\alpha}_k=\hat{a}_{k,x}(z)dz^2,\,\,\hat{\alpha}_0=\hat{a}_{0,x}(z)(dz)^2\,\mbox{and }
\hat{a}_{k,x}(z),\,\,\hat{a}_{0,x}(z)\,\mbox{ holomorphic in $\Omega_{r,x}$}\\ \\ \hat{a}_{k,x} \to{\hat{a}_{0,x}} \,\, k\to +\infty, \, \mbox{ uniformly in } \Omega_{r,x}.
\end{array}
\end{equation}
Furthermore,
\begin{equation}\label{coordalfa}
\alpha \in C_2(X)\implies \alpha = a_{x}(z)dz^2 \quad \mbox{$a_{x}(z)$ holomorphic in $\Omega_{r,x}.$}
\end{equation}
So, by means of formula \eqref{norme}, in $B(x;r)$ the following local expression in $z$-coordinates holds:
\begin{equation}\label{localcoordqdiff}  
\begin{array}{l} 
<\alpha \,,\, \hat{\alpha}_k>dA = a_x (\overline{\hat{a}_{k,x}}) |dz^2|^2 e^{2u_X} \frac{i}{2} dz \wedge d\bar{z}= 4 a_x (\overline{\hat{a_{k,x}}}) e^{-2u_X} \frac{i}{2} dz \wedge d\bar{z}. \\
 \end{array} \end{equation}  

\vspace{0.3cm}
We start our ``local" analysis around a given blow-up point, say $x_0\in \mathcal{S}.$ 

\medskip

For small $r>0,$ in \eqref{local_coord} we set, 
\begin{equation}\label{cappuccio a}
\hat{a}_{k}:=\hat{a}_{k,x_0} \, \mbox{ and } \, \hat{a}_{0}:=\hat{a}_{x_0} \\\\ \mbox{ with }\,\, \hat{a}_{k} \to \hat{a}_{0} \, \,\, \mbox{ uniformly in $\Omega_r:=\Omega_{r,x_0},$ } 
\end{equation}
as $k\to +\infty.$\\ 
Moreover we let, 
\begin{equation}\label{xkappa}
  x_k=x_{k,x_0}\in B(x_0; r): \xi_k(x_k):=\max\limits_{B(x_0; r)} \xi_k\to +\infty \, \text{and} \, x_k\to x_0, \, \text{as} \, k\to +\infty,
\end{equation}
and define:
\begin{equation}\label{zetakappabis} 
\begin{array}{l} 
z_k\in \Omega_r \,\, \mbox{the expression of $x_k$ in the given z-coordinates (at $x_0$)}, \\ 
\mbox{so that: }\,\,\,  
z_k \to 0 \,\,\mbox { as } k \to \infty. \end{array}
\end{equation}

As usual, to simplify notation,  we shall not distinguish between a function and its local expression in terms of the given $z-$coordinates defined in $\Omega_r$. 

Therefore, by using a translation and  by replacing: 
\begin{eqnarray}\label{translation}
 \xi_k(z) \to \xi_k(z+z_k) \,\,\, \mbox{ defined in } \Omega_r - z_k ,  
\end{eqnarray} 
for $\delta >0$ sufficiently small: $\bar{B}_{\delta} \subset  (\Omega_r - z_k ),$ we are reduced to analyse the local problem:
\begin{eqnarray}\label{equation_xi}
-\Delta \xi_k = W_ke^{\xi_k} - g_k 
    \, \text{ in $B_\delta,$  } \,
\int_{B_\delta} W_k e^{\xi_k} \frac{i}{2} dz \wedge d\bar z \leq C,
\end{eqnarray} 
where $\Delta
:=4\p_z\p_{\bar z}$ is the flat Laplacian in $\CC$ (or $\R^2$), and we have:  
\begin{equation}\label{equation_W}
\begin{array}{l}
W_{k}(z):= R_{k}(z+z_k)e^{2u_X(z+z_k)}=32|\hat{a}_{k}(z+z_k)|^2e^{-2u_X(z+z_k)}\\ g_{k}(z):= e^{2u_X(z+z_k)}f_k (z + z_k).
\end{array}
\end{equation} 
Thus, in view of \eqref{translation}, there holds:
\begin{eqnarray}\label{blow up seq}
    \xi_k(0)=\max_{B_{\delta}} \xi_k \to +\infty, \quad \text{ as } k\to +\infty,
\end{eqnarray}
and we may let the origin be the only blow-up point of $\xi_k$ in $\bar B_{\delta}$,   namely:
\begin{eqnarray}\label{blow-up seq bounds cpt set}
\forall    K \Subset \bar B_{\delta} \setminus \{0\} ~~ \max_K \xi_k \leq C ~~ \text{ with suitable } C=C(K)>0.
\end{eqnarray} 
By well known potential estimates (see \cite{Li} 
and \cite{BCLT}) we know also that,
$$
    \max_{\p B_{\delta}}\xi_k-\min_{\p B_{\delta}} \xi_k \leq C \quad \mbox{ \  \ } 
$$ for suitable $C=C(\delta)>0$. 

By the convergence properties in \eqref{f_0}, \eqref{6.27}, \eqref{cappuccio a} and by recalling \eqref{equation_W}, as $ k\to +\infty,$  we have: 
\begin{eqnarray}\label{convergenceW}
W_k \to W_0 \,\,\mbox{uniformly in } \bar B_{\delta}\,\, \mbox{ with } \,\,\, W_0:=32 |\hat{a}_0|^{2}e^{-2u_{X}}, 
\end {eqnarray} 
and for any $p\geq1,$
\begin{eqnarray}\label{g_k}
g_k  \to e^{2u_X} f_0:=g_0,  \quad \text{pointwise and in }  L^p(B_{\delta}). 
\end{eqnarray} 

\begin{remark}\label{massquantization_local} In view of the above properties we can apply Proposition 2.1 and Theorem 1 of \cite{Tar_1} to the "local" problem \eqref{equation_xi} and conclude the analogous blow-up alternatives and mass "quantization" property  as  stated in Theorem \ref{thm_blow_up_global_from_part_1} for the "global" problem \eqref{6.23}.
\end{remark}
By recalling \eqref{blowupmass}, we let: 
\begin{equation}\label{blowup mass2}
\begin{array}{l}
\sigma_0:=\sigma(x_0), \quad  m_0:=m_{x_0}= \frac{1}{8\pi}\sigma(x_0) \in \mathbb{N}, 
\qquad 1\leq m_0\leq(\mathfrak{g}-1). 
\end{array} \end{equation}

\medskip
i.e. $m_0$ is the (quantized) blow up mass at $x_0.$ 

The case $m_0=1$ has been handled in \cite{Tar_2} on the basis of the local pointwise estimates for the blow-up profile of $\xi_k$ around $x_0$ as established in Corollary 3.1 of \cite{Tar_1}. The following holds,
\begin{proposition}\label{local1}
Let $x_0 \in \mathcal{S}$ with $m_0=1.$ Then, for $r>0$  sufficiently small and for every $\alpha \in C_2(X),$ according to the local expressions in \eqref{local_coord} and \eqref{cappuccio a} at $x_0,$ we have:
\begin{equation}\label{total asymp behavior1}
  \begin{array}{l}   \int_{B(x_0; r)} e^{\xi_k} <\a\, \, , \,\widehat \a_k > dA= \frac{\pi}{|\hat{a}_k(z_k)|}( a_{x_0}(0)\frac{\overline{ \hat{a}_k(z_k)}}{|\hat{a}_k(z_k)|}
     +o(1)) + o_r(1) \\ \\\mbox{as $k \to + \infty$} \ \mbox{and where $o_r(1) \to 0$  as $r\to 0^+,$ uniformly on $k.$} \end{array}\end{equation}
\end{proposition}
\begin{proof}
We could refer to \cite{Tar_2} but we sketch the proof for completeness. Since $m_0=1,$ then by \cite{Li_Shafrir} and \cite{Bartolucci_Tarantello} we know that, either $x_0 \notin Z^{(0)}$ (i.e. $x_0$ is not a zero of $\hat{\alpha}_0$) or $x_0 \in Z_{0}$ (i.e. $x_0$ corresponds to a zero for $\hat{\alpha}_0$ of "collapsing" type). In either case, we can rely on the point-wise estimates established in \cite{Li} and \cite{Tar_1} respectively, to obtain: 
 \begin{equation}\label{profile}
 \begin{array}{l}
        \xi_k(z +z_k) =\ln \left( \frac{ e^{\xi_k(z_k)}}{ \left( 1+\frac 1 8 W_{k}(0)e^{\xi_k(z_k)}|z|^2\right)^2 }\right) +O(1) \,\,\mbox{ with}\quad W_{k}(0) >0, \\ \text{for } z \in \Omega_{r,k}:=\Omega_r-z_k.
    \end{array} \end{equation}
In addition, in the "collapsing" case (where $W_{k}(0) \to 0^+\,\,\mbox{as}\, k\to +\infty$), in view of \eqref{profile} we know that:
\begin{equation}\label{profilo}
W_{k}^2(0)e^{\xi_k(z_k)} \geq C \,\,\mbox{ and }\, W_{k}(0)e^{\xi_k(z_k)}  \to +\infty\,\,\mbox{ as }\,k\to +\infty,
\end{equation}
for a suitable constant  $C>0.$ Moreover, in case blow-up occurs with the "concentration" property then:
$W_{k}^2(0)e^{\xi_k(z_k)} \to +\infty\,\,\mbox{ as }\,k\to +\infty,$ see Corollary 3.1 in \cite{Tar_1} for details.
\\
Next, we recall that: $W_k(0)=32|\hat{a}_k(z_k)|^2( 1 + o(1)),\,\,k\to +\infty.$ Thus, by setting:
$\varepsilon_k=(\frac{8}{W_k(0)e^{\xi_k(z_k)}})^{1/2} \to 0,\,\,k\to +\infty;$
by means of \eqref{profile} and \eqref{profilo}, for $\delta>0$ sufficiently small, we compute:
\begin{equation}
          \begin{array}{l}
              \int_{B(x_0; r)} e^{\xi_k} <\a\, \, , \,\widehat \a_k > dA=
              \\ \\=4 \int_{\Omega_{k,r}} e^{\xi_k(z + z_k)} a_{x_0}(z + z_k) \overline{\hat{a_k}}(z+z_k) e^{-2u_X(z+z_k)} \frac{i}{2} dz \wedge d\bar z 
               \\\\ = 
             \frac{32}{W_{k}(0)} ( \int_{B_{\frac{\delta}{\varepsilon_k}}}\frac{1}{(1+|z|^2)^2}a_{x_0}(\varepsilon_k z + z_k) \overline{\hat{a_k}}\left(\varepsilon_k z+z_k) e^{-2u_X(\varepsilon_k z+z_k)} \frac{i}{2} dz \wedge d\bar z +o(1)\right)+o_r(1) =\\\\
             \frac{32}{W_{k}(0)} (a_{x_0}(z_k) \overline{\hat{a_k}}(z_k)\int_{B_{\frac{\delta}{\varepsilon_k}}}\frac{1}{(1+|z|^2)^2} e^{-2u_X(\varepsilon_k z+z_k)} \frac{i}{2} dz \wedge d\bar z +o(1))+o_r(1) = \\\\
             =\frac{\pi}{|\hat{a_k}(z_k)|}( a_{x_0}(0)\frac{\overline{ \hat{a_k}(z_k)}}{|\hat{a_k}(z_k)|}
     +o(1))+o_r(1), \,\, k\to +\infty,
             \end{array}
             \end{equation}
and the term $o_r(1)$ can be dropped in case blow-up occurs with the "concentration" property.
\end{proof}
\medskip

When $m_0\geq 2,$  then necessarily: $ x_0 \in Z^{(0)}$ (see \cite{Li_Shafrir},  \cite{Bartolucci_Tarantello}), namely: $\hat{\alpha}_0(x_0)=0,$ or equivalently in local coordinates  $\hat{a}_0(0)=0.$

\medskip
So, by recalling \eqref{zeroalfakappa} and \eqref{zetakappa},  without loss of generality, we may suppose that, for suitable $s \in \{1, \ldots , N \},$ we have:
$$ q_{j,k} \in Z^{(k)}:\, q_{j,k} \to x_0 \in Z^{(0)}\,\mbox{ as }\,k\to +\infty,\,\,\forall j = 1, ..., s. $$
Moreover, by letting $\hat{p}_{j,k}$ the local expression of $q_{j,k}$ in the given holomorphic z-coordinates at $x_0,$ then by recalling  \eqref{cappuccio a} we have:

\begin{equation}\label{psi}
\begin{array}{l}
\hat{a}_{k}(z) = \prod_{j=1}^s (z-\hat{p}_{j,k})^{n_j}\psi_{k} (z) \to  \hat{a}_{0}(z) = z^{n}\psi_{0}(z),\,\, 

\mbox{uniformly on} \,\, \Omega_r\,\, \\ \\ n_{x_0}:=\sum_{j=1}^s n_j; \qquad \hat{p}_{j,k} \to 0 \quad \mbox{as}\ k \to +\infty,
\end{array}
\end{equation}
where $ \psi_k, \,\,  \psi_0$ are holomorphic functions never vanishing in $\bar B_{\delta},$ and in view of Lemma \ref{C_k} there holds:
\begin{eqnarray}\label{conv_psi}
 \psi_{k} \to \psi_0 \,\, \mbox{ uniformly in $\bar B_{\delta}$ } \,\,\, \text{ as } \,\, k\to +\infty.
 \end{eqnarray}

Therefore, for
\begin{eqnarray}\label{conv_pj}
p_{j,k}:= \hat{p}_{j,k} - z_k \to 0 \,\,\,\text{ as } \,\, k\to +\infty,
\end{eqnarray} 
we find:
\begin{eqnarray}\label{W}
W_{k}(z)=32 (\prod_{j=1}^s |z-p_{j,k}|^{2n_j})h_k(z) e^{-2u_X(z+z_k)},\,\, h_k(z)=|\psi_{k} (z+z_k)|^2
\end{eqnarray} 
in $ \bar B_{\delta}.$
In particular, we have: 
$$
    0<b_1\leq h_k(z) \leq b_2, ~ |\n h_k|\leq A \text{ and } h_k\to h_0 := |\psi_0|^2  ~ \text{ uniformly in } \bar B_{\delta},
$$ 
with suitable constants $0<b_1\leq b_2$ and $A>0.$

To simplify notations (and without loss of generality) from now on we shall use the normalization:
 \begin{eqnarray}\label{h_0}
   h_0(0)=1.
\end{eqnarray}

Again, without loss of generality we may let,
$$
    0\leq |p_{1,k}|\leq |p_{2,k}|\leq \cdots \leq |p_{s,k}| \to 0, \quad \text{ as } k\to +\infty.
$$
Our main effort in the sequel will be to identify, for the blow up point $x_0,$ the corresponding integer $N_{x_0}$ satisfying \eqref{vincoli} as claimed in Theorem \ref{mainth}. To illustrate its origin, we point out that, when $1\leq n_{x_0} \leq 2(m_0 -1),$ then we can simply take: $N_{x_0}=n_{x_0}.$ 
Indeed, we have: $1\leq m_0 \leq \mathfrak{g} -1$ and so in this case we are in position to use the "approximation" Lemma \ref{approximation} with the devisor $D_k :=\sum_{j=1}^sn_jq_{j,k}\to D:=n_{x_0}x_0\,\,\mbox{as } k\to +\infty,$ and conclude: 
\begin{equation}\label{approximation3}
\begin{array}{l}
\forall \alpha \in Q(D)\quad \exists \,\, \alpha_{k} \in Q(D_k): 
\alpha_{k} \to \alpha, \,\,\mbox{ as } k \to \infty.\\
\end{array}
\end{equation}
In particular, in local $z$-coordinates at $x_0$ we have:
\begin{equation}\label{applocal}
\begin{array}{l}
\alpha_k=a_{k,x_{0}}(z)dz^2 \,\,\mbox{and }\, \alpha=a_{x_0}(z)dz^2\quad
(a_{k,x_0}(z) \,\,\mbox{and }\,\, a_{x_0}(z) \,\,\mbox{holomorphic in}\,\, \Omega_r)\\  
a_{k,x_{0}} \to a_{x_0} \,\,\mbox{uniformly in}\,\, \Omega_r\,\, \mbox{as}\,\,  k \to +\infty.
\end{array}
\end{equation}
Moreover, we use standard notation and let: $a_{x_0}^{(n)}$ denote the $n$-complex derivative of the function $a_{x_0}.$

Thus, for the case: 
\begin{equation}\label{casobuono}
m_0\geq2\quad  1\leq n_{x_0} \leq 2(m_0 -1),\quad (n_{x_0}\,\, in \,\, \eqref{psi}).
\end{equation}
we obtain the following asymptotic expression:
\begin{proposition}\label{local2}
Let $x_0 \in \mathcal{S}$ with blow-up mass $m_0$ and suppose that \eqref{casobuono} holds.
Then for the divisors: $D_k :=\sum_{j=1}^sn_jq_{j,k}\to D:=n_{x_0}x_0$ in $X^{(n_{x_0})}$ and for $\alpha \in Q(D) $ let $\alpha_{k} \in Q(D_k)$
as given by \eqref{approximation3} and \eqref{applocal}.  The following holds:
\begin{eqnarray}\label{total asymp behavior}
    \int_{B(x_0; r)} e^{\xi_k} <\a_k\, \, , \,\widehat \a_k > dA= \pi m_0 \frac{a_{x_0}^{(n_{x_0})}(0)}{n_{x_0}!} \overline{ \psi_0(0)} +o(1)~ \text{ as } k\to +\infty, 
\end{eqnarray}
for $r>0$ sufficiently small. 
\end{proposition}
\begin{proof}
To simplify notations, we set:
$$n:=n_{x_0}.$$
We observe that, under the given assumption \eqref{casobuono}, the blow-up of $\xi_k$ at $0$ must occurs with  the "concentration" property.\\ 
\medskip
Indeed, if by contradiction, we assume that (along a subsequence):  $\xi_k \to \xi$ in $C^2_{loc}(B_\delta \setminus \{0\}),$ with $\xi$ satisfying: 
$$\left\{ \begin{array}{l} - \Delta \xi = 32|z|^{2n}h_0e^{-2u_X} e^{\xi} + 8 \pi m_0 \delta_0   \ \mbox{in} \  B_\delta \\ \int_{B_\delta} |z|^{2n} h_{0}e^{-2u_X}e^{\xi}\frac{i}{2}dz \wedge d\bar{z} < C \end{array}\right. $$ 
then, the presence of the Dirac singularity at the origin implies that, $\xi(z) = 4m_0 \log\left(\frac{1}{|z|}\right) + O(1) \ \mbox{as} \ z \to 0,$ with  $n+1 < 2m_0$ by \eqref{casobuono}, and this is impossible as it violates the integrability of $|z|^{2n} e^{\xi}$ around the origin.  

In other words, by recalling \eqref{W} and \eqref{h_0}, we have:
\begin{eqnarray}\label{W_k}
 32\prod_{j=1}^s |z-p_{j,k}|^{2n_j}e^{\xi_k} \rightharpoonup 8\pi m_0 \delta_0, ~~ \text{ weakly in the sense of measures. } 
    \end{eqnarray} 

Moreover, since $\alpha_{k} \in Q(D_k) \,\,\mbox{and}\,\,\alpha \in Q(D)$ (locally) we have:
\begin{eqnarray}\label{app_alpha2} 
 a_{k,x_0}(z)=\prod_{j=1}^s(z - \hat p_{j,k})^{n_j}C_{k}(z) \,\,\mbox{ and }  a_{x_0}(z)=z^{n}C(z),
\end{eqnarray}  
with $C_{k}(z)\,\,\mbox {and } C(z)$ holomorphic in  $\bar{B_\delta},$ and according to Lemma \ref{C_k} we find:
\begin{eqnarray}\label{convunif}
C_k \to C\,\,\mbox{ uniformly in} \, \bar {B}_\delta,\,\mbox { as } k \to \infty,\,\mbox{ and }\,  C(0)=\frac{a_{x_0}^{(n)}(0)}{n!}. 
\end{eqnarray} 

As a consequence, by using \eqref{W_k}, \eqref{app_alpha2}, \eqref{convunif} we find:
\begin{equation}
\begin{array}{l}
\int_{B(x_0; r)} e^{\xi_k} <\a_k\, \, , \,\widehat \a_k > dA=
              \\ \\=4 \int_{B_\delta} e^{\xi_k} a_{k}(z + z_k) (\bar{z}- \overline{p_k})^n \overline{\psi}_k(z+z_k) e^{-2u_X(z+z_k)} \frac{i}{2} dz \wedge d\bar z +o(1) \\ \\
             =\;4\int_{B_\delta} e^{\xi_k}\prod_{j=1}^s|z - p_{j,k}|^{2n_j} C_{k}
             (z +z_k)\overline{\psi}_k(z + z_k) e^{-2u_X(z + z_k)} \frac{i}{2} dz \wedge d\bar z \\\\ + o(1) = \pi m_0 
             \,C(0)\overline{ \psi_0(0)} +o(1)= \pi m_0 \,\frac{a^{(n)}(0)}{n!} \overline{ \psi_0(0)} +o(1),~ \text{ as } k\to +\infty.
\end{array}
\end{equation}
          as claimed.

\end{proof}
From now on we assume that, 
\begin{equation}\label{n+1}
m_0 \geq 2\,\,\mbox{ and }\,\, n:=n_{x_0} > 2(m_0 -1),
\qquad (n_{x_0} \,\,\mbox{in} \, \eqref{psi}).
\end{equation}

Notice that when \eqref{n+1} holds then (in view of \eqref{sigma_q}) necessarily $s\geq2$ and the blow-up point $x_0$ must be a zero for $\hat{\alpha}_0$ of "collapsing" type, namely: $x_0 \in \mathcal{S} \cap Z_0 $.
As a consequence, the ``concentration" property for $\xi_k$ is no longer ensured and alternative (ii)-b) of Theorem \ref {thm_blow_up_global_from_part_1} could hold. Hence now we can write:
\begin{eqnarray*}\label{2.1*}
\int_{(\Omega_{r} - z_k)}  \prod_{j=1}^s |z-p_{j,k}|^{2n_j} h_k(z) e^{-2u_X(z + z_k)} e^{\xi_k} \frac{i}{2} dz \wedge d\bar z \\ =  \int_{B_\delta}  \prod_{j=1}^s |z-p_{j,k}|^{2n_j} h_k(z) e^{-2u_X(z + z_k)} e^{\xi_k} \frac{i}{2} dz \wedge d\bar z  +  o_{r}(1),
\end{eqnarray*}
for any $\delta >0$ sufficiently small and where $ o_{r}(1) \to 0 \,\, \mbox{ as} \,\, r\to 0^+, \, \mbox{ uniformly in $k.$} $ \\

Also note that when \eqref{n+1} holds then the divisors $D_k$ and $D$ in Proposition \ref{local2} are no longer suitable for our purposes. Indeed now we cannot control their degree as required by Lemma \ref{approximation}, and so the "approximation" property   \eqref{approximation3} is no longer ensured.

Thus, we need to dig into the blow-up profile of $\xi_k,$ in order to find suitable replacements of the integer $n$ and corresponding divisors $D_k\, \mbox{and} \, D$ for which Lemma \ref{approximation} applies.

To this purpose let,
$$
    \tau_{k}^{(1)}:=|p_{s,k}|\to 0^+         ~~\text{ as } k\to +\infty,
$$
and consider,
\begin{equation}\label{phikappa}
 \v_k^{(1)}(z):=\xi_k(\tau_{k}^{(1)} z) +2(n+1) \ln \tau_{k}^{(1)}\quad \mbox{and} \qquad p_{j,k}^{(1)}:=\frac{p_{j,k}}{\tau_{k}^{(1)}},\quad j=1,\cdots, s. 
\end{equation}
Since, $0\leq |p_{1,k}^{(1)}|\leq |p_{2,k}^{(1)}|\leq \cdots \leq |p_{s,k}^{(1)}| =1,$ we can also assume that, 
\begin{equation}\label{p_j}
p_{j,k}^{(1)}\to p_{j}^{(1)}, \text{ as } k\to +\infty,
\end{equation}
(possibly along a subsequence) with suitable points, $p_{j}^{(1)},\,\, j=1,\ldots, s.$  Thus setting:
$$
 h_{1,k}(z):=h_k(\tau_{k}^{(1)} z) e^{-2u_X(\tau_{k}^{(1)} z + z_k)}  ~~\text{  and }  g_{1,k}(z):=(\tau_{k}^{(1)})^2  g_k( \tau_k z),
$$
we  have:
\begin{equation}\label{eq of varphi-k}
\left\{
\begin{array}{lc}
-\Delta \v_k^{(1)} =32\prod\limits_{j=1}^s |z-p_{j,k}^{(1)}|^{2n_j} h_{1,k}(z) e^{\v_k^{(1)}} -g_{1,k}(z)  & \text{ in } \Omega_{k,\delta}:=\{ |z| < {\frac{\delta}{\tau_{k}^{(1)}}}\} \\
 \int_{\Omega_{k,\delta}} \prod\limits_{j=1}^s |z-p_{j,k}^{(1)}|^{2n_j}  h_{1,k}( z)e^{\v_k^{(1)}} \frac{i}{2} dz \wedge d\bar z\leq C \end{array} \right.
\end{equation}
with
$$
 h_{1,k}(z) \to h_0(0)=1\text{ and } g_{1,k} \to 0, \text{ uniformly in } C_{loc}(\RR^2), \text{ as } k \to +\infty.
$$

Let, 
\begin{eqnarray}\label{lambda-varphi}
 \lambda_{\varphi^{(1)}}:=\lim_{R\to +\infty} \lim_{k\to +\infty} 32\int_{B_R} \prod_{j=1}^s |z-p_{j,k}^{(1)}|^{2n_j} h_{1,k}( z)e^{\v_k^{(1)}(z)} \frac{i}{2} dz \wedge d\bar z, 
\end{eqnarray}
and we easily check that: $\lambda_{\varphi^{(1)}}\leq 8\pi m_0.$\\ 
Also from \cite{Lee_Lin_Wei_Yang} we know that the following identity holds:
\begin{eqnarray}\label{s0 identity}
 \s_0^2 -\lambda_{\varphi^{(1)}}^2 =8\pi (n+1) (\s_0-\lambda_{\varphi^{(1)}}) 
\end{eqnarray}
(see also the Appendix in \cite{Tar_1} for a detailed proof of \eqref{s0 identity}).

Hence, from  \eqref{s0 identity} we obtain that, $\lambda_{\varphi^{(1)}}\in 8\pi \NN\cup \{0\}$, and
\begin{equation}\label{mzero}
\begin{array}{l}
\text{either } \lambda_{\varphi^{(1)}}=\s_0=8\pi m_0\,
\text{ or }\, 8\pi m_0=8\pi (n+1)-\lambda_{\varphi^{(1)}}\\
\mbox{and if the latter case holds then necessarily: }\,\,2m_0-1\geq(n+1).
\end{array} 
\end{equation} 
The following holds:
\begin{proposition}\label{violate}
If \eqref{n+1} holds then:
$$ \begin{array}{l} (i) \  \lambda_{\varphi^{(1)}} = 8 \pi m_0, \\ (ii) \ \varphi_k^{(1)}(0) = max_{\Omega_k^{(1)}} \varphi_k^{(1)} \to + \infty, \ \mbox{as $ k \to + \infty$  (up to subsequences),}\\ \mbox{i.e. $\varphi_k^{(1)}$ blows up at $0$}. \\  (iii) \ \mbox{Blow- up occurs with the  "concentration" property, namely:}\\
\prod_{j=1}^{s} |z- p_{j,k}^{(1)}|^{2 n_j} h_{1,k}(z)e^{\varphi_k^{(1)}(z)}  \rightharpoonup
 \sum_{y \in \mathcal{S}_{x_0}^{(1)}} 8 \pi m_{y}^{(1)} \delta_y, \\ \mbox{weakly in the sense of measure,} 
\mbox{ where $\mathcal{S}_{x_0}^{(1)}$ denotes the blow-up set of $\varphi_k^{(1)}$}\\ \mbox{and  $m_{y}^{(1)} \in \mathbb{N}$ is the blow-up mass at the point $y \in \mathcal{S}_{x_0}^{(1)}$} \mbox{ and moreover: }\\ \sum_{y \in \mathcal{S}_{x_0}^{(1)}} m_{y}^{(1)} = m_0.  \end{array}. $$ 

\end{proposition} 

\begin{proof}
In view of our assumption and \eqref{mzero} we easily deduce (i).
In order to establish (ii) we argue by contradiction and suppose that $\varphi_k^{(1)}(0) \leq C.$ Thus, by setting $\varepsilon_k = e^{\frac{- \xi_k(0)}{2(n+1)}}$ we have that:
$$0 < \frac{\tau_k^{(1)}}{\varepsilon_k} = \tau_k^{(1)} e^{\frac{ \xi_k(0)}{2(n+1)}} = e^{\frac{- \varphi_k(0)}{2(n+1)}} \leq C.$$
Let, 
$$ \Phi_k(z) := \xi_k(\varepsilon_k z) - \xi_k(0) = \xi_k(\varepsilon_k z) + 2(n+1) \log(\varepsilon_k), \,\, \ |z| < \frac{r}{\varepsilon_k} := R_k, $$

satisfying: 

$$ \left\{ \begin{array}{l}  - \Delta \Phi_k  =32 \prod_{j=1}^{s} |z - \frac{p_{j,k}}{\varepsilon_k}|^{2 n_j} h_k(\varepsilon_k z)e^{-2u_X(\varepsilon_k z + z_k)} e^{\Phi_k} - g_{2,k}(z) \ \mbox{in} \  B_{R_k}  \\ \\ \Phi_k(0) = max_{B_{R_k}} \Phi_k = 0 \\  \\ \int_{B_{R_k}}  \prod_{j=1}^{s} |z - \frac{p_{j,k}}{\varepsilon_k}|^{2 n_j} h_k(\varepsilon_k z)e^{-2u_X(\varepsilon_k z + z_k)} e^{\Phi_k}\frac{i}{2}dz \wedge d\bar{z}  \leq C,   \end{array} \right.$$
with $g_{2,k}(z):=\varepsilon_k^2  g_k( \varepsilon_k z).$
Since, $ \frac{|p_{j,k}|}{\varepsilon_k}   \leq \frac{\tau_k^{(1)}}{\varepsilon_k} \leq C,$ we may assume (up to a subsequence)  that, 
\begin{equation}\label{pq2} \frac{p_{j,k}}{\varepsilon_k} \to \hat{q}_{j},\,\, 1 \leq j \leq s. \end{equation} 

Moreover, by well known Harnack type inequalities valid for $\Phi_k,$ see \cite{Brezis_Merle} and \cite{Bartolucci_Tarantello}, we find that $\Phi_k$ is uniformly bounded in $C^{2,\gamma}_{loc}.$ Therefore, by taking a subsequence if necessary, we find that $\Phi_k \to \Phi$ in $C^{2,\gamma}_{loc} \ \mbox{as} \ k \to + \infty, $ with $\Phi$ satisfying:

$$  \left\{ \begin{array}{l}  - \Delta \Phi  =32 \prod_{j=1}^{s} |z - \hat{q_j}|^{2 n_j} e^{\Phi} \ \mbox{in} \  \mathbb{R}^2  \\ \\ \Phi(0) = max_{\mathbb{R}^2} \Phi = 0 \\  \\ \int_{\mathbb{R}^2}  32\prod_{j=1}^{s} |z - \hat{q}_j|^{2 n_j} e^{\Phi} \frac{i}{2}dz \wedge d\bar{z}:= \lambda_{\Phi} \leq 8 \pi m_0,
  \end{array} \right. $$ 
with
 $$\lambda_{\Phi}:= \lim_{R \to + \infty} \lim_{k \to + \infty}  \int_{B_{R_k}} 32 \prod_{j=1}^{s} |z - \frac{p_{j,k}}{\varepsilon_k}|^{2 n_j} h_k(\varepsilon_k z) e^{-2u_X(\varepsilon_k z + z_k)}e^{\Phi_k} \frac{i}{2}dz \wedge d\bar{z}.   $$
Again, we have:
\begin{equation}\label{Lin-Yang}  (8 \pi m_0)^2 - \lambda_{\Phi}^2 = 8\pi(n+1)(8\pi m_0- \lambda_{\Phi}), \end{equation}
(cf. \cite{Lee_Lin_Yang_Zhang} and appendix in 
\cite{Tar_1}). As in \eqref{mzero}, the relation \eqref{Lin-Yang} together with the given assumption \eqref{n+1}  implies that necessarily: 
$\lambda_{\Phi} = 8 \pi m_0.$
At this point, we can use Theorem 2 of \cite{Chen_Li_2} to find that,
$$ \Phi(z) = \frac{\lambda_{\Phi}}{2 \pi} \log \left(\frac{1}{|z|} \right) + O(1) =  4m_0 \log \left(\frac{1}{|z|} \right) + O(1),  \ \mbox{for} \,\,
|z| \geq 1$$  
and the integrability of the term: $ \prod_{j=1}^{s} |z - \hat{q}_j|^{2 n_j} e^{\Phi} $ in $\mathbb{R}^2$ implies that $2m_0 > n+1,$ in contradiction with \eqref{n+1}. 
Thus, we have proved that $\varphi_k^{(1)}(0) \to + \infty$ as $ k \to + \infty.$ So the sequence $\varphi_k^{(1)}$ admits a (non empty) blow-up set  $\mathcal{S}_{x_0}^{(1)}.$ In particular $0 \in \mathcal{S}_{x_0}^{(1)},$ and for $y \in \mathcal{S}_{x_0}^{(1)}$ we denote by $m_y^{(1)} \in \mathbb{N}$ the corresponding blow-up mass. If by contradiction we suppose that blow-up occurs \underline{without} the concentration property, then every blow-up point is of \underline{collapsing} type,  that is, $\mathcal{S}_{x_0}^{(1)} \subseteq \{p_{j}^{(1)},\, j=1,...,s \}$ and:
\begin{equation}\label{J_y}
I_y := \{ j \in \{1, \ldots,s\} : p_{j}^{(1)}=y \}\,\mbox{ satisfies: } |I_y|\geq 2,\quad \forall y \in \mathcal{S}_{x_0}^{(1)}
 \end{equation}
($|I_y|$ is the cardianality of $I_y$).

Moreover (along a subsequence) $\varphi_k^{(1)} \to \varphi^{(1)}$ uniformly in $C^2_{loc}(\mathbb{R}^2 \setminus \mathcal{S}_{x_0}^{(1)})$  with $\varphi^{(1)}$ satisfying: 
$$  \left\{ \begin{array}{l}  - \Delta \varphi^{(1)}  = 32\prod_{j=1}^{s} |z - p_{j}^{(1)} |^{2 n_j} e^{\varphi^{(1)}} + 8\pi \sum_{y \in \mathcal{S}_{x_0}^{(1)}} m_y^{(1)} \delta_y \ \mbox{in} \  \mathbb{R}^2   \\  \\32 \int_{\mathbb{R}^2}  \prod_{j=1}^{s} |z - p_{j}^{(1)} |^{2 n_j} e^{\varphi^{(1)}(z)}\frac{i}{2}dz \wedge d\bar{z} + 8 \pi  \sum_{y \in \mathcal{S}_{x_0}^{(1)} } m_y^{(1)} = \lambda_{\varphi^{(1)}}=8\pi m_0.   \\ \\ 
\end{array} \right.$$   

For $y \in \mathcal{S}_{x_0}^{(1)}$ we set: $n_y = \sum_{j \in I_y} n_j$ and  $I_0 = \{1, \ldots,s\} \setminus \bigcup_{y \in \mathcal{S}_{x_0}^{(1)}} I_y$ (possibly empty) and consider the function:
$$ \Psi^{(1)}(z) = \varphi^{(1)}(z) + 4 \sum_{y \in \mathcal{S}_{x_0}^{(1)}} m_y^{(1)} \log|z-y|.$$  
We see that $\Psi^{(1)}$ extends smoothly at any $y \in \mathcal{S}_{x_0}^{(1)}$ and satisfies:
\begin{equation}\label{maineqPsi1}  \left\{ \begin{array}{l}  - \Delta \Psi^{(1)}  = ( 32\prod_{y \in \mathcal{S}_{x_0}^{(1)}} |z -y|^{2 n_y-4m_y^{(1)})} ) ( \prod_{j \in I_0} |z - p_{j}^{(1)}|^{2 n_j} ) e^{\Psi^{(1)}}  \ \mbox{in} \  \mathbb{R}^2   \\  \\ \int_{\mathbb{R}^2} ( 32\prod_{y \in \mathcal{S}_{x_0}^{(1)}} |z -y|^{2 n_y-4m_y^{(1)})})  ( \prod_{j \in I_0} |z - p_{j}^{(1)}|^{2 n_j}e^{\Psi^{(1)}}\frac{i}{2}dz \wedge d\bar{z} =\\ =8 \pi(m_0- \sum_{y \in \mathcal{S}_{x_0}^{(1)}} m_y^{(1)}), \end{array} \right.  \end{equation}
with the understanding that, if $I_0$ is empty then the corresponding product term included in \eqref{maineqPsi1} must be dropped.\\ 
\medskip
Hence (as above) we find that:
 $\Psi^{(1)}(z)=4( m_0 - \sum_{y \in \mathcal{S}_{x_0}^{(1)}}m_y) \log\left( \frac{1}{|z|} \right)+O(1),$ for $|z| > 1,$
 and again the integrability condition implies that necessarily:  
 $2m_0-1 \geq n+1,$ in contradiction with our assumption.
 
 In conclusion, $\varphi_k^{(1)}$ must blow-up with the "concentration" property, namely:
 $$  \prod_{j=1}^{s} |z- p_{j,k}^{(1)}|^{2 n_j} h_{1,k}(z)e^{\varphi_k^{(1)}(z)}  \rightharpoonup
 \sum_{y \in \mathcal{S}_{x_0}^{(1)}} 8 \pi m_y^{(1)} \delta_y, \ \mbox{weakly in the sense of measure}. $$ Consequently, 
 $ 8 \pi m_0 = \lambda_{\varphi^{(1)}} = 8 \pi \sum_{y \in \mathcal{S}_{x_0}^{(1)}} m_y^{(1)},$ that is:
 $\sum_{y \in \mathcal{S}_{x_0}^{(1)}} m_y^{(1)} =m_0$ as claimed.  \end{proof}

At this point, we are ready to complete the asymptotic description of the local integral terms in \eqref{betawedge} as follows:

\begin{thm}\label{local3}
Let $x_0 \in \mathcal{S}$ admit blow-up mass $m_0$ such that \eqref{n+1} holds.  
There exist a suitable constant $b_{x_0} \in \mathbb{C}\setminus \{0\},$  a sequence $\varepsilon_{k,x_0} \to 0^+,$ an integer $ N_0\in \mathbb{N}:\,\, 1\leq N_0 \leq 2(m_0-1)$ (in particular, $1\leq N_0< n$) and an effective divisor $D_k$ satisfying: $deg\, D_k = N_0 \,\, \mbox{ and }\,\, D_k \to D:=N_0x_0 \,\,\mbox{as }k\to +\infty $ such that, for every $\alpha \in Q(D)\,\, \mbox{and } \alpha_k \in Q(D_k): \,\, \alpha_k \to \alpha \,\,\mbox{as}\,k\to +\infty,$  in terms of the local expression \eqref{coordalfa} for $\alpha$ at $x_0,$ there holds:
\begin{equation}\label{total asymp behavior2}
\int_{B(x_0; r)} e^{\xi_k} <\a_k\, \, , \,\widehat \a_k > dA=\frac{\pi b_{x_0}}{\varepsilon_{k,x_0}}(\frac {a_{x_0}^{(N_0)}(0)}{N_{0}!} \overline{ \psi_0(0)} +o(1)) + o_r(1),
\end{equation}
as $k\to +\infty,$ and $o_r(1) \to 0$ as $r\to 0^+$ uniformly in $k.$
\end{thm}

Observe that, when blow-up for $\xi_k$ occurs with the "concentration" property then the term $o_r(1)$ can be dropped. 
\begin{remark}\label{remm2}
It is important to note that the given information about $N_0$ is crucial and represent the main core of Theorem \ref{local3}.\\ In particular, for the divisor $D_k \to D \,\,\mbox{as}\, k\to +\infty$  in Theorem \ref{local3}, it allows us to  use Lemma \ref{approximation}, so that every $\alpha \in Q(D)$  can be "approximated" by a suitable sequence $\alpha_k \in Q(D_k),$ that is
we can always guarantee \eqref{approximation3} and \eqref{applocal}. 
 \end{remark}
Next, we observe the  following :
\begin{remark}\label{remm1}
It is possible to interpret Proposition \ref{local1} as a particular case of either Proposition \ref{local2} or Theorem \ref{local3}. In other words, if $m_0=1$ then we can recast the expansion \eqref{total asymp behavior1} as a particular case, either of the expansions  \eqref{total asymp behavior} with  $n=0$ or \eqref{total asymp behavior2} with $N_0=0$ respectively and when we take $\alpha_k=\alpha \in C_{2}(X),\, \forall k\in \mathbb{N}.$ 

Indeed, when $m_0=1$ then either $x_0 \notin Z^{(0)},$ and so blow-up occurs with the "concentration" property and $\hat{a}_k(z_k) \to \hat{a}_0(0)=\psi_{0}(0)\neq 0,$ actually $|\hat{a}_0(0)|=1$ in view of the normalization \eqref{h_0}.   Thus in this case, \eqref{total asymp behavior1} just reads
as in \eqref{total asymp behavior} exactly when $n=0\,\,\mbox{and } m_0=1.$ 

Or $x_0 \in Z_{0},$ then  $\hat{a}_{k}(z_k) \to 0 \,\,\mbox{ as }\, k\to +\infty $ and $\hat{a}_{k}(z)\,\mbox{ and } \hat{a}_{0}(z)$ must take the form  \eqref{psi}. In those notation and (possibly along a subsequence) letting: 
$\frac{p_{j,k}}{|p_{j,k}|} \to \tilde{p_j}, \,\,k\to +\infty$ we have: 
\begin{eqnarray}\label{acappuccio}
\frac{\hat{a}_k(z_k)}{|\hat{a}_k(z_k)|} \to \prod_{j=1}^s (-\tilde{p}_j)^{n_j}\frac{\psi_0(0)}{|\psi_0(0)|}.
\end{eqnarray}

Since $|\psi_0(0)|=1,$  in this case \eqref{total asymp behavior1} reads exactly as \eqref{total asymp behavior2} with $\varepsilon_{k,x_0} = |\hat{a_k}(z_k)| \to 0,\,\mbox{ as }\, k \to+\infty, $ $b_{x_0}=\prod_{j=1}^s (-\tilde{p}_j)^{n_j}\neq0$ and $N_0=0.$
\end{remark}
By combining Proposition \ref{local1}, Proposition \ref{local2}, Theorem \ref{local3}  and Remark \ref{remm1}, we arrive at the following conclusion:
\begin{corollary}\label{local4}
    
$\forall x \in \mathcal{S}$ there exist a suitable constant $B_{x} \in \mathbb{C}\setminus \{0\},$ and a uniformly bounded sequence $\varepsilon_{x,k} >0 \,$  such that,\\
(i) if $m_x\geq 2$ then $ \exists\, N_x\in \mathbb{N}:\,\, 1\leq N_x \leq 2(m_x-1)$ and an effective divisor $D_{x,k}$ satisfying:
$deg\, D_{x,k} = N_x \,\, \mbox{ and}\,\, D_{x,k} \to D_x:=(N_x)x$ such that, \\
$\forall\alpha \in Q(D_x)\,\, \exists \, \alpha_k \in Q(D_{x,k}): \,\, \alpha_k \to \alpha \,\,\mbox{as}\,\,k\to +\infty,$ and in terms of the local expression \eqref{coordalfa} for $\alpha$ at $x,$ we have:
\begin{equation}\label{total asymp behavior3}
   \int_{B(x;r)} e^{\xi_k} <\a_k\, \, , \,\widehat \a_k > dA=\frac{B_x}{\varepsilon_{x,k}}(a_x^{(N_x)}(0)+o(1)) + o_r(1).
\end{equation}
\\
(ii) If $m_x=1$ then \eqref{total asymp behavior3} holds with $N_x=0.$ Namely: for any $\alpha,\,\, \alpha_k \in C_2(X):\,\alpha_k \to \alpha $ there holds:
\begin{equation}
 \int_{B(x;r)} e^{\xi_k} <\alpha_k\, \, , \,\widehat \a_k > dA=\frac{B_x}{\varepsilon_{x,k}}( a_x(0) +o(1)) + o_r(1).
\end{equation}
Moreover, if $x \in \mathcal{S}$ is a zero for $\hat\alpha_0$ with multiplicity $n_x$ and $n_x+1 > 2m_x-1,$ then $\varepsilon_{x,k} \to 0.$

\end{corollary}
More importantly, from Corollary \ref{local4} 
we can establish Theorem \ref{mainth} and Theorem \ref{thm1}. More precisely, 
with , we 
The following holds:  
\begin{thm}\label{mainthbis}
Let $\mathcal{S}$ be the blow-up set of $\xi_k.$ For every $x \in \mathcal{S}$ with blow-up mass $m_x \in \mathbb{N},$ there exists $N_x \in \mathbb{N} \cup \{0\}$ with
$N_x  \leq 2(m_x -1),$ so that for the divisor $D := \sum_{x \in \mathcal{S}} (N_x + 1) x$ the following holds:
\begin{equation}\label{orthog} \int_X \beta \wedge \alpha = 0, \quad \forall \alpha \in Q(D).
\end{equation} 
In particular, letting:
\begin{equation}\label{fine}
\begin{array}{l}
\mathcal{S}=\{ x_1, \ldots, x_k\}\quad m_j:=m_{x_{j}}\quad N_j:=N_{x_{j}}+1 \quad j=1\ldots k\\ \\
\textbf{m}=(m_1,\ldots, m_k) \quad \mbox{and }\quad  \textbf{N}=(N_1,\ldots, N_k)
\end{array}
\end{equation} 
then (in the notation of \eqref{Ykmn}) we have:
$1\leq k \leq \mathfrak{g}-1,\quad ({\bf{m}},{\bf{N}}) \in \mathfrak{I}_k$ 
and
$$
\quad[\beta]_{\mathbb{P}} \in \tilde{\Sigma}_{k,\textbf{m},\textbf{N}}$$
\end{thm}

\begin{proof} 
In view of Corollary \ref{local4} and in order to unify notations, we set: $$D_{x,k}=0\,\,\mbox{ for } x \in \mathcal{S}\,\,\mbox{with} \,\, m_x=1,$$  and let: 
$$ D_k:= \sum_{x \in \mathcal{S}}D_{x,k} \to D_0:= \sum_{x \in \mathcal{S}}(N_x)x\,\,\mbox{ as } \, k \to +\infty,$$ where \\\\
$deg\, D_k = deg \,D_0 =\sum_{x \in \mathcal{S}}N_x \leq \sum_{x \in \mathcal{S}}2(m_x -1)\leq2(m-1)\leq2(\mathfrak{g} -2)< 2(\mathfrak{g}-1).$\\

Thus by Lemma \ref{approximation}, for $\alpha \in Q(D_0),$  we find $\alpha_k \in Q(D_k):\,\, \alpha_k \to \alpha,$ \\
as $ k \to +\infty.$
In particular, $\alpha_k \in Q(D_{x,k}), \,\, \forall x \in \mathcal{S}$ and by combining Lemma \ref{lemmaexpand} and Corollary \ref{local4}, we obtain:
\begin{equation}\label{espansione}
\begin{array}{l}
\mbox{for}\, \alpha \in Q(D_0))\,\,\mbox{ and } \alpha_k \in Q(D_k):\,\, \alpha_k \to \alpha, \,\,\mbox{ there holds:}\\\\
\int_X \beta \wedge \alpha  = \int_X \beta \wedge \alpha_k  +o(1)= \\ \\  e^{\frac{-s_k}{2}} \left(\sum_{x \in \mathcal{S}}\frac{B_x}{\varepsilon_{x,k}}  (a_{x}^{(N_{x})}(0) +o(1)) + o_r(1) \right) +o(1).\end{array} 
\end{equation}

\vskip.1cm
CLAIM:
$$ 0 < \frac{e^{\frac{-s_k}{2}}}{\varepsilon_{x,k}} \leq C, \quad \forall x \in \mathcal{S};$$  
with suitable $C>0.$ 
\vskip.1cm
To establish the above estimate,  we let $x_0 \in \mathcal{S}$ such that (up to subsequence): \begin{equation}\label{epsilon}
0< \varepsilon_{x_0,k}\leq \varepsilon_{x,k},\,\, \forall x\in \mathcal{S}.
\end {equation} 
For the given divisor $D = \sum_{x \in \mathcal{S}} (N_x + 1) x,$ we consider:
$$D(x_0):= D - x_0, \, \mbox{and} \, D_k(x_0):= D_k + \sum_{x\in \mathcal{S} \setminus\{x_0\}} x.$$ 
Clearly,  $D_k(x_0) \to D(x_0)$ as $k \to + \infty,$ and $deg\,D_k(x_0)= deg\,D(x_0)$=\\ $= \sum_{x \in \mathcal{S}} (N_x +1) - 1 \leq 2( \mathfrak{g} -2) < 2( \mathfrak{g} -1).$ \\
Therefore, by Corollary \ref{alphaspecial} we find: $\alpha_0 \in Q(D(x_0))$ but $\alpha_0 \notin Q(D).$ In words, by using for $\alpha_0$ the local expression \eqref{coordalfa}, we have: \begin{equation}\label{allzero}
a_{x_0}^{(N_{x_0})}(0) \neq 0\,\,\mbox{ while }\,\,a_{x}^{(N_{x})}(0)=0,\quad\forall x \in \mathcal{S} \setminus \{x_0\}. \end{equation}
Furthermore as above we choose  
$\alpha_k \in Q(D_k(x_0))$ with $\alpha_k  \to \alpha_0,$ as $ k \to +\infty.$

Since $D_k(x_0) \geq D_k$ then $\alpha_k\in Q(D_k)$ and analogously, since $D(x_0) \geq D_0$ then $\alpha_0 \in Q(D_0).$ So we can use \eqref{espansione} and \eqref{allzero} 
to conclude:
\begin{equation}
\begin{array}{l}
\int_X \beta \wedge \alpha_0  = \int_X \beta \wedge \alpha_k  +o(1)= \\ \\   \frac{e^{\frac{-s_k}{2}}}{\varepsilon_{x_0,k}} \left( B_{x_0}a_{x_0}^{(N_{x_0})}(0) +o(1) + o_r(1) \right) +o(1).\end{array} 
\end{equation}
with $B_{x_0}a_{x_0}^{(N_{x_0})}(0) \neq 0.$ Consequently,  $0 < \frac{e^{\frac{-s_k}{2}}}{\varepsilon_{x_0,k}} \leq C,$ and the claim follows in view of \eqref{epsilon}.\\

At this point we consider, 
$$
\hat{D}_k:=D_k+\sum_{x \in \mathcal{S}} x \to D = \sum_{x \in \mathcal{S}} (N_x + 1) x \quad k\to +\infty, $$ with
\begin{equation}\label{grado}
deg ( \hat{D}_k ) = deg (D) = \sum_{x \in \mathcal{S}} (N_x +1) < 2( \mathfrak{g} -1).
\end{equation}
 
Hence, if we take  $\alpha\in Q(D)$ then $a_{x}^{(N_{x})}(0)=0,\,\, \forall x\in \mathcal{S}.$\\ In addition, in view of \eqref{grado} there exist 
$\alpha_k \in Q(\hat{D}_k) \to \alpha \in Q(D),\,\,\mbox{as} \quad k\to +\infty.$ Since in particular, $\alpha_k \in Q(D_k)\,\,\mbox{ and }\,\, \alpha \in Q(D_0),$ we can apply \eqref{espansione} with  $a_{x}^{(N_{x})}(0)=0,\,\, \forall x\in \mathcal{S}.$\\
In this way, we arrive at the desired conclusion by using the Claim and by passing to the limit, first as $k \to +\infty$ and then as $r\to 0^+.$

In view of \eqref{blowupmass} and \eqref{massatotale}, the number of blow-up points $k=|\mathcal{S}|$  satisfies: $1 \leq k\leq \mathfrak{g} -1 $ and consequently the divisor $D\in Y_{ (k,{\bf{m}},{\bf{N}})}.$ Thus, by Remark \ref{sigma} we have $[\beta]_{\mathbb{P}} \in \tilde{\Sigma}_{(k,\bf{m},\bf{N})},$ as claimed . 

\end{proof}

\begin{remark}\label{remarkdivisori}: The sub-variety $\tilde{\Sigma}_{\mathfrak{g}}$ defined in Corollary \ref{dimproj} admits codimension at least $\mathfrak{g}-1$ in 
$\mathbb{P}(V^*).$ Moreover, in view of Theorem \ref{mainthbis}, we have that: \\ if $[\beta]_{\mathbb{P}} \notin \tilde{\Sigma}_{\mathfrak{g}}$ then we can rule out \underline {blow up} and in this way also Theorem \ref{thm1} is established.
\end{remark}

Consequently, we are just left to prove Theorem \ref{local3}.
\\ \\
THE PROOF OF THEOREM \ref{local3}

\begin{proof} 

By the given assumptions, we are in position to use Proposition \ref{violate}. Thus $\varphi_k^{(1)}$ in \eqref{phikappa} admits a (non empty) blow-up set $\mathcal{S}^{(1)}_0:= \mathcal{S}^{(1)}_{x_0},$ such that $0 \in \mathcal{S}^{(1)}_0$ and the properties specified in $(i),(ii),(iii)$ are satisfied. 
In particular, for $R>1$ sufficiently large, there holds:
\begin{equation}\label{resto0} 
\int_{\Omega_{k,\delta} \setminus B_R} \prod_{j=1}^s |z-p_{j,k}^{(1)}|^{2n_j} h_{1,k}( z)e^{\v_k^{(1)}(z)} \frac{i}{2} dz \wedge d\bar z \to 0, \,\, k\to +\infty.
\end{equation}
Consequently, for a given integer $N_0 \in   \{1,...,n-1 \}$ we find :
\begin{equation}\label{resto}
\int_{R\tau_k^{(1)}\leq |z| \leq \delta} e^{\xi_{k}}|z|^{n+N_0}\frac{i}{2} dz \wedge d\bar z = o(\frac{1}{(\tau_k^{(1)})^{n-N_0}}).
\end{equation}

To simplify notation we let:
$$I= \{1,...,s\}.$$
For $ y \in \mathcal{S}_0^{(1)}$ we define: 
\begin{equation}\label{max1} 
z_{k,y}^{(1)} \in B_{\delta}(y) :\,\, \varphi_k^{(1)}(z_{k,y}^{(1)})  = max_{B_{\delta}(y)}   \varphi_k^{(1)} \to + \infty\,\,  \mbox{ and } z_{k,y}^{(1)} \to y, \,\, \mbox{ as } \ k \to + \infty,
  \end{equation}
where we notice in particular that,
\begin{equation}\label{max0}
z_{k,y=0}^{(1)} = 0, \quad \forall k \in \mathbb{N}.
 \end{equation}

Again, the points in \eqref{p_j} may not be distinct, and so we let, $$Z_1^{(0)}\,\,\mbox{the set of \underline{distinct} points in }\, \{p_{1}^{(1)},...,p_{s}^{(1)}\}$$ and define:
\begin{equation}
I_{y}= \{ j \in I :\,\, p_{j}^{(1)} = y  \},\quad \mbox{ for }\, y\in Z_1^{(0)}
\end{equation}
the sets $I_y $ are mutually disjoint and 
$I = \bigcup_{y\in Z_1^{(0)}} I_y.$\\ 
In this way, we can identify the set (possibly empty) of points in $Z_1^{(0)}$ of "collapsing" type (where different points in $\{ p_{j,k}^{(1)}, \quad j= 1,...,s \}$ coalesce at the limit ) as given by:
\begin{equation}
Z_0^{(1)} = \{y\in Z_1^{(0)} : |I_y|\geq 2 \}.
    \end{equation}
Our most delicate task will be to control the asymptotic behavior of $\varphi_k^{(1)}$ around blow-up points in $\mathcal{S}_0^{(1)}\cap Z_0^{(1)}.$

To this purpose we observe that,\\
if $y\in \mathcal{S}_0^{(1)}\setminus Z_1^{(0)},$ then $ m^{(1)}_y =1$ (see \cite{Li_Shafrir}) and we set $ n^{(1)}_y=0,$\\
if $y\in \mathcal{S}_0^{(1)}\cap Z_1^{(0)},$ then we let: $n^{(1)}_y=\sum_{j\in I_y} n_j.$\\

We define: $$\mathcal{S}_*^{(1)}=\{ y\in \mathcal{S}_0^{(1)}:\,\, n^{(1)}_y+1\leq 2m^{(1)}_y -1 \},$$ 
and by \cite{Li_Shafrir} and \cite{Bartolucci_Tarantello}, we have: $\mathcal{S}_0^{(1)}\setminus Z_0^{(1)} \subseteq \mathcal{S}_*^{(1)}.$ 
We let,
$$ \mathcal{S}_1^{(1)} = \mathcal{S}_*^{(1)} \cap Z_1^{(0)}, \quad 
\mathcal{S}_2^{(1)}=\mathcal{S}_0^{(1)}\setminus  \mathcal{S}_*^{(1)} \subseteq Z_0^{(1)}, $$
with the understanding that some of the above sets may be empty. 

Define,
\begin{equation}\label{I_1} 
I_{(1)}=\cup_{y\in \mathcal{S}_1^{(1)}}I_y \subseteq I,
\end{equation}
and notice that actually, $I_{(1)} \neq I.$ 
\medskip
Indeed, if by contradiction we assume that $I_{(1)}=I,$ then,
\begin{equation}
n=\sum_{j\in\ I} n_j = \sum_{j\in\ I_{(1)}} n_j = \sum_{y\in \mathcal{S}_1^{(1)}} n^{(1)}_y \leq  2 \sum_{y\in \mathcal{S}_1^{(1)}}(m^{(1)}_y -1) \leq 2(m_0 -1),
\end{equation}
 in contradiction to the given assumption \eqref{n+1}.  Therefore,
\begin{equation}{\label{I1}}
I \setminus I_{(1)} \neq \emptyset\,\
\mbox{ and }\,\,
p_{j}^{(1)}\notin \mathcal{S}_0^{(1)}, \quad \forall j\in I\setminus I_{(1)} .
\end{equation}

To illustrate our procedure, we start to consider the case where: \begin{equation}\label{empty1}
 \mathcal{S}_2^{(1)} =\emptyset.    
\end{equation}
that is,
\begin{equation}\label{simpleblowup}
\mathcal{S}_0^{(1)}=\mathcal{S}_*^{(1)}=(\mathcal{S}_0^{(1)}\setminus Z_1^{(0)}) \cup \mathcal{S}_1^{(1)}. 
\end{equation}
When \eqref{empty1} holds, then or all $y\in \mathcal{S}_0^{(1)},$ we define:
\begin{equation} \label{xkappay}
\begin{array}{l}

x_{k,y}^{(1)}:=\, \mbox{(unique) point } \in B(x_0,r)\, \mbox{mapped (in the $z$-coordinates at $x_0$) to }\\
\\
z_k + \tau_k^{(1)}z_{k,y}^{(1)}.\quad
\mbox{Hence}\quad x_{k,y}^{(1)} \to x_0, \quad k \to +\infty.
\end{array}
\end{equation}
Thus in this case, we consider the devisor:
\begin{equation}\label{divisore1}
\begin{array}{l}
\hat{D}_{k,x_0}= \sum_{y\in \mathcal{S}_0^{(1)}} x_{k,y}^{(1)} + \sum_{y\in \mathcal{S}_1^{(1)}}(\sum_{j\in\ I_y} n_jq_{j,k}) = \\
\\= \sum_{y\in \mathcal{S}_0^{(1)}} x_{k,y}^{(1)} + \sum_{j\in\ I_{(1)}} n_jq_{j,k}.
\end{array}
\end{equation}

Therefore, $ deg\,(\hat{D}_{k,x_0}) = \sum_{y\in \mathcal{S}_0^{(1)}}(n_{y}^{(1)}+1) := N_0 + 1, $
with $N_0 \in \mathbb{N}$ satisfying:
$$ 2\leq N_0+1 \leq \sum_{y\in \mathcal{S}_0^{(1)}}(2m_{y}^{(1)}-1) \leq 2m_0 -1, $$
and since by assumption: $2m_0 -1 < n+1,$ we find in particular that, $1\leq N_0 < n.$

Furthermore,
\begin{equation}\label{convdiv}
\hat{D}_{k,x_0} \to (N_0 +1)x_0 := \hat{D}_{x_0}, \quad k \to +\infty.  
\end{equation}

Next we recall that, $0 \in \mathcal{S}_0^{(1)} $ and $z_{k,y=0}^{(1)}=0,$  therefore: $x_{k,y=0}^{(1)} = x_k, \, \forall k.$

Hence, when \eqref{empty1} holds, we set:
\begin{equation}\label{divisore1bis}
D_k=D_{k,x_0}= \hat{D}_{k,x_0} - x_k \to (N_0)x_0:=D.
\end{equation}
Notice in particular that we can apply Lemma \ref{approximation} with the above divisors and so, for given $\alpha \in Q(D)$ there always exist $\alpha_k \in Q(D_k)$:  
$\alpha_k \to \alpha \,\, \mbox{as } k \to + \infty. $
Consequently,  
in z-coordinates at $x_0,$ for 
$\alpha_k=a_{k,x_{0}}(z)dz^2$ and  $\alpha=a_{x_0}(z)dz^2$ ($a_{k,x_0}(z),$ and  $a_{x_0}(z)$ holomorphic in $\Omega_r$) we find:
\begin{equation}\label{alpha_coord_1}
\begin{array}{l}
 a_{k,x_0}(z+z_k)= \prod_{j \in I_{(1)} }(z- p_{j,k})^{n_j} \prod_{y \in \mathcal{S}_{0}^{(1)}\setminus{\{0\}}} (z -\tau_k^{(1)}z_{k,y}^{(1)}) C_k(z +z_k),\\ \\
 a_{x_0}(z)=z^{N_0}C(z),
\end{array}
\end{equation}
where the functions $ C_k \,\,\mbox{ and }\,\, C$ satisfy \eqref{convunif} with $n$ replaced by $N_0.$

At this point, by means of \eqref{resto} and \eqref{alpha_coord_1} together with \eqref{psi}, for $r>0$ sufficiently small and $R>1$ sufficiently large,  we can compute:
\begin{equation}\label{approx1}
\begin{array}{l}
\int_{B(x_0; r)} e^{\xi_k} <\a_k\, \, , \,\widehat \a_k > dA= \\ \\ 4 \int_{\{|z|<\tau_k^{(1)}R \}} e^{\xi_k}\overline{\hat{a}_{k,x_0}}(z+z_k)a_{k,x_0}(z+z_k)e^{-u_X(z+z_k)}\frac{i}{2} dz \wedge d \bar{z} 
+ o(\frac{1}{(\tau_k^{(1)})^{n-N_0}}) \\ \\ + o_r(1) = 4 \int_{\{|z|<\tau_k^{(1)}R \}} \left[ e^{\xi_k}  \prod_{j \in I} (\overline{z-p_{j,k}})^{n_j}\overline{\Psi}_k(z +z_k) \prod_{j \in I_{(1)}} (z-p_{j,k})^{n_j}\right.\\
\\\left. \prod_{y \in \mathcal{S}_{0}^{(1)}\setminus{\{0\}}}(z -\tau_k^{(1)} z_{k,y} ^{(1)}) 
  C_k(z+z_k) e^{-u_X(z+z_k)} \frac{i}{2} dz \wedge d \bar{z} \right] + o(\frac{1}{(\tau_k^{(1)})^{n-N_0}})+ o_r(1) \\
  \\ 
=\frac{4}{{(\tau_k^{(1)})^{n-N_0}}} \left(  \int_{\{|z|<R \}} \left[ e^{\varphi^{(1)}_k} \prod_{j \in I_{(1)}}  |z-p_{j,k}^{(1)}|^{2n_j}\prod_{y \in \mathcal{S}_{0}^{(1)}\setminus{\{0\}}}(z - z_{k,y} ^{(1)}) \right. \right. \\
\\\left. \left. \prod_{j \in I\setminus I_{(1)}} (\overline{z-p_{j,k}^{(1)}})^{n_j}\overline{\Psi}_k(\tau_k^{(1)}z +z_k)C_k(\tau_k^{(1)}z +z_k) e^{-u_X(\tau_k^{(1)}z +z_k)}\frac{i}{2} dz \wedge d \bar{z} \right] +o(1) \right)\\ \\ +  o_r(1).
\end{array}
\end{equation}

 According to $(ii)$ of Proposition \ref{violate}, we have:
$$e^{\varphi^{(1)}_k} \prod_{j \in I_{(1)}}  |z-p_{j,k}^{(1)}|^{2n_j} \rightharpoonup  8 \pi  \sum_{y \in \mathcal{S}_{0}^{(1)}} \frac{m_y^{(1)}}{32 |A_y|^2}  \delta_y\,\,\mbox{weakly in the sense of measure,} $$ 
where,  
\begin{equation}\label{Ay}
A_y= \prod_{j\in I\setminus I_{(1)}}(y-p_{j}^{(1)})^{n_j} \neq 0,\,\,y \in \mathcal{S}_{0}^{(1)},
\end{equation}
(recall \eqref{I1}).
As a consequence, from \eqref{max1} and \eqref{approx1}, we conclude:
\begin{equation}\label{approx1bis}
\begin{array}{l}
\int_{B(x_0; r)}e^{\xi_k} <\a_k\, \, , \,\widehat \a_k > dA=\\ \\=\frac{\pi}{(\tau_k^{(1)})^{n-N_0}}\left(m_{y=0}^{(1)}[\, \prod_{y \in \mathcal{S}_{0}^{(1)}\setminus{\left( 0 \right)}} (-y) ] \frac{\bar{A}_{y=0}}{|A_{y=0}|^2} C(0) \bar{\psi} _{0}(0) + o(1) \right)+o_r(1).

\end{array}
\end{equation}

Thus, in this case \eqref{total asymp behavior2} is established with $\varepsilon_{k,x_0} = (\tau_k^{(1)})^{n-N_0} \to 0, \,\, k\to + \infty,$ and $b_{x_0} = m_{y=0}^{(1)}(\prod_{y \in \mathcal{S}_{0}^{(1)}\setminus{\left( 0 \right)}} (-y)) \frac{\bar{A}_{y=0}}{|A_{y=0}|^2} \in \mathbb{C} \setminus {\{ 0 \}}. $

\medskip

Next we consider the case where,
$$ \mathcal{S}_2^{(1)}=\mathcal{S}_0^{(1)}\setminus  \mathcal{S}_*^{(1)} \neq \emptyset. $$
Recall that, every $y \in \mathcal{S}_2^{(1)} $ must be a blow-up point of "collapsing" type, i.e. $y \in \mathcal{S}_2^{(1)} \cap Z_{(1)}^0.$ 

The goal now is to complete the divisor specified above, by considering:
\begin{equation}\label{divisore2}
\hat{D}_{k,x_0}= \hat{D}_{k,x_0} (\mathcal{S}_*^{(1)}) +  \sum_{y\in \mathcal{S}_2^{(1)}} \hat{D}_{k,x_0} (y),
\end{equation}
where, in analogy to \eqref{divisore1} we set:
\begin{equation}
\hat{D}_{k,x_0} (\mathcal{S}_*^{(1)})
= \sum_{y\in \mathcal{S}_*^{(1)}} x_{k,y}^{(1)} + \sum_{j\in\ I_{(1)}} n_jq_{j,k};
\end{equation}
while, for any $y\in \mathcal{S}_2^{(1)},$ we show next how to select (via a further blow-up procedure) the appropriate subset of indices in $I_y,$ and construct a suitable divisor $ \hat{D}_{k,x_0}(y),$ with the desired properties.
\medskip
To proceed further, let us fix $y_0 \in \mathcal{S}_2^{(1)}$ so that,  
$n_{y_0}^{(1)} +1 > 2 m_{y_0}^{(1)} -1.$
For $\delta > 0$ small, we let
\begin{equation}
\begin{array}{l}
\varphi_{k,y_0}^{(1)}(z) := \varphi_{k}^{(1)}(z+ z_{k,y_0}^{(1)}),\,\, z \in B_\delta,  \\
\\
\tilde{p}_{j,k}^{(1)}:= \tilde{p}_{j,k}^{(1)}(y_0)=p_{j,k}^{(1)} - z_{k,y_0}^{(1)},
\end{array}
\end{equation}

and observe that, \begin{equation}\label{qjk1}  \begin{array}{l} \tilde{p}_{j,k}^{(1)} \to 0, \ \mbox{as} \  k \to + \infty;  \quad \forall j \in I_{y_0} := J_0^{(1)} \\ \\   \tilde{p}_{j,k}^{(1)} \to \tilde{p}_{j}^{(1)} \neq 0, \,\,  \forall j \in I \setminus J_0^{(1)},\,\,(\mbox{provided}\,\, I \setminus J_0^{(1)} \neq \emptyset.) \end{array}  \end{equation} 

Therefore, the function:
$$W_{k,y_{0}}^{(1)}(z) = \prod_{j \in I \setminus J_0^{(1)}} |z- \tilde{p}_{j,k}^{(1)}|^{2n_j} h_{1,k}(z + z_{k,y_0}^{(1)}),$$
is uniformly bounded from above and from below away from zero in $\bar{B}_\delta$ and,
 $$ W_{k,y_{0}}^{(1)}(z) \to W_{y_0}^{(1)}(z):=\prod_{j \in I \setminus J_0^{(1)}} |z- \tilde{p}_{j}^{(1)}|^{2n_j}, \,\,\mbox{uniformly in 
$\bar{B}_\delta,$} $$ 
(in particular: $W_{y_0}^{(1)}(0)\neq 0.$ ) It is understood that, $W_{y_0}^{(1)}=1$ when $I = J_0^{(1)}$. 
By letting,
$$n_{0}^{(1)}=n_{y_0}^{(1)}\quad \mbox{ and }  \,\, m_{0}^{(1)}=m_{y_0}^{(1)}$$
for $\delta >0$ sufficiently small, we have:
\begin{equation}\label{P2bis} \left\{ \begin{array}{l}- \Delta \varphi_{k,y_0}^{(1)} =32e^{\varphi_{k,y_0}^{(1)}} (\prod_{j \in J_0^{(1)}} |z- \tilde{p}_{j,k}^{(1)}|^{2n_j})W_{k,y_{0}}^{(1)}(z) - g_{1,k}(z+z_{k,y_0}^{(1)})   \ \mbox{in} \ B_{\delta}  \\ \\
 \varphi_{k,y_0}^{(1)}(0) = max_{B_{\delta}(0)}  \,\varphi_{k,y_0}^{(1)}  \to + \infty, \,\, \mbox{as} \ k \to +\infty,\\ \\  
 \int_{B_{\delta}} e^{\varphi_{k,y_0}^{(1)}} (\prod_{j \in J_0^{(1)}} |z- \tilde{p}_{j,k}^{(1)}|^{2n_j})W_{k,y_{0}}^{(1)}(z) < C.
 \end{array} \right.
 \end{equation} 
\\
In addition, 
\begin{equation}\label{convdebole1}
\begin{array}{l}
e^{\varphi_{k,y_0}^{(1)}} (\prod_{j \in J_0^{(1)}}|z- \tilde{p}_{j,k}^{(1)}|^{2n_j}) \rightharpoonup 
\frac {8\pi m_{0}^{(1)}}{32W_{y_0}^{(1)}(0)} \delta_0 \\ \\ \ \mbox{weakly in the sense of measure in} \,\, B_{\delta}, \\ \\
n^{(1)}_{0}  = \sum_{j \in J_0^{(1)}} n_j > 2(m_{0}^{(1)} - 1).
 \end{array}
\end{equation}
Clearly, problem \eqref{P2bis} for the function $\varphi_{k,y_0}^{(1)}$  is completely analogous to that of $\xi_k$ we started with,  however we have the following improvement:
\begin{lemma}\label{s,m}
Either $|J_0^{(1)}| := s_0 < s$ or $|J_0^{(1)}| = s$ and $2 \leq m_{0}^{(1)}  < m_0,$ where $|J_0^{(1)}|$ is the cardinality of $J_0^{(1)}.$ 
\end{lemma}
\begin{proof}
If $|J_0^{(1)}| = s,$ then $J_0^{(1)}=I$ and so $p_{j}^{(1)}=y_0, \,\, \forall j \in I.$ Hence 
$\mathcal{S}_2^{(1)}=\{y_0\}$ and since $|p_{s}^{(1)}|=1,$ we also know that $|y_0| =1.$ Hence, $ 0 \in\mathcal{S}_{0}^{(1)} \setminus \{y_0\} = \mathcal{S}_0^{(1)}\setminus Z_{1}^{(0)}\neq \emptyset$ and $ \forall y\in\mathcal{S}_{0}^{(1)} \setminus \{y_0\}$ we have: $m_{y}^{(1)}=1.$

As a consequence we find, $m_0=m_{0}^{(1)} +|\mathcal{S}_{0}^{(1)} \setminus \{y_0\}|\geq m_{0}^{(1)}+1,$ 
and necessarily, $1 \leq m_{0}^{(1)} \leq m_{0}-1,$ as claimed.
Next we need to exclude that, $m_{0}^{(1)} = 1.$ Indeed if this was the case, then we would have:  $m_{y}^{(1)}=1, \,\,\forall y\in\mathcal{S}_{0}^{(1)}.$ 
Therefore, around any $y \in\mathcal{S}_{0}^{(1)} $  we could use the pointwise blow-up profile description for $\varphi_{k,y_0}^{(1)}$ (analogous to \eqref{profile}) as given in Corollary 3.1 of \cite{Tar_1}. Consequently, we would find comparable rates on the behavior of $\varphi_{k,y_0}^{(1)}$ away from the blow up set. In particular we could deduce:      
$$ \frac{ [\prod_{j \in I} |z_{k,y_0}^{(1)}- p_{k,j}^{(1)}|^{2n_j}h_{1,k}(z_{k,y_0}^{(1)})]^2 e^{\varphi_k^{(1)}(z_{k,y_0}^{(1)})}}{ [(\prod_{j \in I} | p_{k,j}^{(1)}|^{2n_j}h_{1,k}(0)]^2 e^{\varphi_k^{(1)}(0)}} = O(1), \mbox{ \ as \ } k \to + \infty.$$ 
On the other hand,  $\varphi_k^{(1)}(0) = max_{\Omega_{k,\delta}}  \varphi_k^{(1)},$ and from the estimates above we derive:  
$$ \prod_{j \in I} | p_{k,j}^{(1)}|^{2n_j}
 \leq C  \prod_{j \in I} |z_{k,y_0}^{(1)}- p_{k,j}^{(1)}|^{2n_j} \to 0,  \ \mbox{as} \ k \to + \infty$$ which is impossible, since $  \prod_{j \in I} | p_{k,j}^{(1)}|^{2n_j} \to |y_0|^{2n} = 1, \,\, k \to + \infty.$ 
\end{proof}

Since the analogous of Proposition \ref{violate}
applies to $\varphi_{k,y_0}^{(1)},$ we can iterate the blow-up procedure illustrated above. For this purpose,
let 
$$ \tau_k^{(2)} :=\tau_k^{(2)}(y_0)= max_{j \in J_0^{(1)}} |\tilde{p}_{j,k}^{(1)}| \to 0, \,\, \mbox{ as } k \to + \infty, $$
and define:
\begin{equation}\label{varphik2} \begin{array}{l} \varphi_{k,y_0}^{(2)}(z) := \varphi_{k,y_0}^{(1)}(\tau_{k}^{(2)}z) +2(n^{(1)}_{0} +1)\log(\tau_k^{(2)}),  \quad  z \in \Omega_{k,\delta}^{(1)}:= \{ z \in \mathbb{C} : |z| < \frac{\delta}{\tau_k^{(2)}} \}\\\\  
p_{k,j}^{(2)}:= \frac{\tilde{p}_{j,k}^{(1)}}{\tau_k^{(2)}}, \,\, \mbox{(so that, } \,|p_{k,j}^{(2)}| \leq 1),\quad p_{k,j}^{(2)} \to p_{j}^{(2)}, \quad \forall j\in J_0^{(1)}; \end{array} \end{equation}
(possibly along a subsequence) with suitable points $p_{j}^{(2)}, \,\, j\in J_0^{(1)}.$
Using Proposition \ref{violate} for  $\varphi_{k,y_0}^{(1)},$ we obtain:

\begin{equation}\label{varphik2bis} \varphi_{k,y_0}^{(2)}(0) = max_{\Omega_{k,\delta}^{(1)}} \varphi_{k,y_0}^{(2)} \to + \infty \ \mbox{as} \ k \to + \infty, \end{equation} namely, $\varphi_{k,y_0}^{(2)}$ blows up. Let $\mathcal{S}_0^{(2)}(y_0)$ denote the blow-up set of  
$\varphi_{k,y_0}^{(2)},$ so that, $0 \in \mathcal{S}_0^{(2)}(y_0).$ For $w \in \mathcal{S}_0^{(2)}(y_0),$ we define: 
$$\begin{array}{l}    m_w^{(2)}:= m_w^{(2)}(y_0) = \mbox {blow-up mass of $\varphi_{k,y_0}^{(2)}$ at }\,\, w. \quad \mbox{Hence:}\\ \\  
e^{\varphi_{k,y_0}^{(2)}} (\prod_{j \in J_0^{(1)}} |z- {p}_{j,k}^{(2)}|^{2n_j})\ \rightharpoonup \frac{8\pi}{32W_{y_0}^{(1)}(0)}\sum_{w \in \mathcal{S}_0^{(2)}(y_0)}m_w^{(2)}\delta_{w},\\ \\
\mbox{weakly in the sense of measure,}\\ \\\sum_{w \in \mathcal{S}_0^{(2)}(y_0)} m_w^{(2)} = m_{0}^{(1)}.\end{array}$$
\\
We proceed exactly as above, and for $ w \in \mathcal{S}_0^{(2)}(y_0)$ we set, 
\begin{equation}\label{max1bis}
\begin{array}{l}
z_{k,w}^{(2)}:=z_{k,w}^{(2)}(y_0) \in B_{\delta}(w) :\,\, \varphi_k^{(2)}(z_{k,w}^{(2)})  = max_{B_{\delta}(w)}   \varphi_k^{(2)} \to + \infty,\,\, \\ \\  z_{k,w}^{(2)} \to w, \,\, \mbox{ as } \ k \to + \infty, \quad\mbox{moreover}\,\, z_{k,w=0}^{(2)} = 0, \quad \forall k \in \mathbb{N}.
  
  \end{array} \end{equation}
Again set,  $$Z_2^{(0)}\,\,\mbox{the set of \underline{distinct} points in} \, \{p_{j}^{(2)},\quad \forall j \in J_0^{(1)}\}$$ and consider the subset: 
\begin{equation}
J_{w}= \{ j \in J_0^{(1)} :\,\, p_{j}^{(2)} = w  \},\,\, \mbox{ for }\, w\in Z_2^{(0)} \,
\end{equation} so that, $J_{w}$ are mutually disjoint and
$J_0^{(1)} = \bigcup_{w\in Z_2^{(0)}} J_w.$\\ 

Consequently, the set (possibly empty) of points in $Z_2^{(0)}$ of "collapsing" type, is given by:
\begin{equation}
Z_0^{(2)} = \{w\in Z_2^{(0)} : |J_w|\geq 2 \}.
    \end{equation}
As before we observe that,
if $w\in \mathcal{S}_0^{(2)}(y_0)\setminus Z_2^{(0)}$ then $ m^{(2)}_w =1,$  and in this case conveniently we set: $ n^{(2)}_w=0.$\\
While, if $w\in \mathcal{S}_0^{(2)}(y_0)\cap Z_2^{(0)},$ then we set: $n^{(2)}_w=\sum_{j\in J_w} n_j.$\\

Thus we define, $$\mathcal{S}_*^{(2)}(y_0)=\{ w\in \mathcal{S}_0^{(2)}(y_0):\,\, n^{(2)}_w+1\leq 2m^{(2)}_w -1 \},$$ 
so that, $\mathcal{S}_0^{(2)}(y_0)\setminus Z_0^{(2)} \subseteq \mathcal{S}_*^{(2)}(y_0);$ 
and we consider the (possibly empty) sets,
$$ \mathcal{S}_1^{(2)}(y_0) = \mathcal{S}_*^{(2)}(y_0) \cap Z_2^{(0)}, \quad 
\mathcal{S}_2^{(2)}(y_0)=\mathcal{S}_0^{(2)}(y_0)\setminus  \mathcal{S}_{*}^{(2)}(y_0)\subseteq Z_0^{(2)}. $$
Let,
$$ J_{(2)}(y_0)=\cup_{w\in \mathcal{S}_1^{(2)}(y_0)}J_w \subseteq J_0^{(1)},$$
and as above, we see that:
\begin{equation}{\label{J1}}
J_{0}^{(1)}(y_0) \setminus J_{(2)}(y_0) \neq \emptyset\,\
\mbox{ and }\,\,
p_{j}^{(2)}\notin \mathcal{S}_0^{(2)}(y_0) \quad \forall j\in J_{0}^{(1)}(y_0) \setminus J_{(2)}(y_0).
\end{equation}
Consequently,
\begin{equation}\label{convweek2}
\begin{array}{l}
e^{\varphi_{k,y_0}^{(2)}} (\prod_{j \in J_{(2)}(y_0)} |z- {p}_{j,k}^{(2)}|^{2n_j})\ \rightharpoonup \frac{\pi}{4}\sum_{w \in \mathcal{S}_0^{(2)}(y_0)}M_w^{(2)}(y_0)\delta_{w},\\
\\
\mbox{where, }\\
M_w^{(2)}(y_0) = \frac{8\pi}{32W_{y_0}^{(1)}(0)|A_{w,y_{0}}^{(2)}|^2}\,\, \mbox{with } A_{w, y_{0}}^{(2)}= \prod_{j \in (J_{0}^{(1)}(y_0) \setminus J_{(2)}(y_0))}(w-p_j^{(2)})^{n_j} \neq 0.
\end{array}
\end{equation}

Again, we consider first the case: \begin{equation}\label{empty2}
 \mathcal{S}_2^{(2)}(y_0) =\emptyset,    
\end{equation}
namely: $\mathcal{S}_0^{(2)}(y_0)=\mathcal{S}_*^{(2)}(y_0)=(\mathcal{S}_0^{(2)}(y_0)\setminus Z_2^{(0)}) \cup \mathcal{S}_1^{(2)}(y_0), 
 $ and therefore: \\ $n_w^{(2)}+1 \leq 2m_w^{(2)}-1,\,\, \forall w\in \mathcal{S}_0^{(2)}(y_0). $\\
In this case, for all $w\in \mathcal{S}_0^{(2)}(y_0),$  we define the point:
\begin{equation} \label{xkappa2}
\begin{array}{l}
x_{k,w}^{(2)} = x_{k,w}^{(2)}(y_0)\in B(x_0,r)\,\, \mbox{mapped (in the $z$-coordinates at $x_0$) to the point:}\\
\\ 
z_k + \tau_k^{(1)}(z_{k,y_0}^{(1)} + \tau_k^{(2)}z_{k,w}^{(2)}):=\zeta_{k,w},\quad
\mbox{and so}\quad x_{k,w}^{(2)} \to x_0, \quad k \to +\infty.
\end{array}
\end{equation}
Thus, when \eqref{empty1} holds, then (in analogy to the previous step) we take the devisor:
\begin{equation}\label{divisory_0}
\begin{array}{l}
\hat{D}_{k,x_0}(y_0)=\sum_{w\in \mathcal{S}_0^{(2)}(y_0)} x_{k,w}^{(2)} + \sum_{w\in \mathcal{S}_1^{(2)}(y_0)}(\sum_{j\in\ J_y} n_jq_{j,k})=\\ \\=\sum_{w\in \mathcal{S}_0^{(2)}(y_0)}x_{k,w}^{(2)}+\sum_{j\in J_{(2)}(y_0)}n_j q_{j,k},
\end{array}
\end{equation}
and therefore,
\begin{equation}\label{degreedivisory_0}
\begin{array}{l}
deg (\hat{D}_{k,x_0}(y_0)) = \sum_{w\in \mathcal{S}_0^{(2)}(y_0)}(n_w^{(2)}+1)=: N_{y_0}^{(1)} +1,\\
\\
\mbox{with: }\,2\leq N_{y_0}^{(1)} +1\leq \sum_{w\in \mathcal{S}_0^{(2)}(y_0)}(2m_w^{(2)}-1) \leq 2m_{0}^{(1)}-1 \,\, \mbox{and }\, 1\leq N_{y_0}^{(1)} < n_{0}^{(1)}; \\
\\
\hat{D}_{k,x_0}(y_0) \to \hat{D}_{x_0}(y_0):=(N_{y_0}^{(1)} +1)x_0, \quad k \to +\infty.
\end{array}
\end{equation}
At this point, if we assume:
\begin{equation}\label{vuoto}
 \mathcal{S}_2^{(2)}(y) =\emptyset, \,\,  \forall y\in \mathcal{S}_2^{(1)}\end{equation} 
 then it is clear to take as divisor:  
\begin{equation}\label{divisore2bis}
\begin{array}{l}
\hat{D}_{k,x_0}= \sum_{y\in \mathcal{S}_*^{(1)}} x_{k,y}^{(1)} + \sum_{j\in\ I_{(1)}} n_jq_{j,k} + \\\
\\+ \sum_{y\in \mathcal{S}_2^{(1)}} \left( \sum_{w\in \mathcal{S}_0^{(2)}(y)}x_{k,w}^{(2)}(y) + \sum_{j\in J_{(2)}(y)}n_j q_{j,k} \right). 
\end{array}
\end{equation}
Moreover we recall that, $\forall y\in \mathcal{S}_2^{(1)}$ we have: $z_{k,w=0}^{(2)}(y)=0$ and so: $x_{k,w=0}^{(2)}(y)=x_{k,y}^{(1)}.$ Therefore, by setting,
$$ I_{0}=(I_{(1)})\cup_{y\in \mathcal{S}_2^{(1)}}J_{(2)}(y) \subsetneq I,$$
(recall \eqref{J1}) we obtain,
\begin{equation}\label{divisore1bisbis}
\hat{D}_{k,x_0}=\sum_{j\in I_{0}} n_jq_{j,k} + \sum_{y\in \mathcal{S}_{0}^{(1)}} x_{k,y}^{(1)}+\sum_{y\in \mathcal{S}_2^{(1)}} \left( \sum_{w\in \mathcal{S}_0^{(2)}(y) \setminus{\{0\}}}x_{k,w}^{(2)}(y)\right).
\end{equation}
Thus, we find:
\begin{equation}\label{degreedivisory1}
\begin{array}{l}
deg (\hat{D}_{k,x_0}) = \sum_{y\in \mathcal{S}_*^{(1)}}(n_y^{(1)}+1)+ \sum_{y\in \mathcal{S}_2^{(1)}} (N_{y}^{(1)} +1)=: N_{0}+1,\\
\\ 2\leq N_{0}+1 \leq \sum_{y\in \mathcal{S}_0^{(1)}} (2m_y^{(1)}-1) \leq 2m_{0}-1, \,\, \mbox{and so:}\,\, 1\leq N_0 < n;\\
\\
\hat{D}_{k,x_0} \to \hat{D}_{x_0}:=(N_{0} +1)x_0, \quad k \to +\infty.
\end{array}
\end{equation}
Let us fix $y_*\in \mathcal{S}_2^{(1)},$ such that (possibly along a subsequence) there holds:
\begin{equation}\label{minimo}
(\tau_k^{(2)}(y_*))^{n_{y_*}^{(1)}-N_{y_*}^{(1)}} \leq (\tau_k^{(2)}(y))^{n_{y}^{(1)}-N_{y}^{(1)}}, \quad \forall y \in \mathcal{S}_2^{(1)}.
\end{equation}
We are going to show that the appropriate "approximation" devisor in this case is given by:
$$ D_k = \hat{D}_{k,x_0} - x_{k,y_*}^{(1)},$$
Namely, 
\begin{equation}\label{divisore3}
\begin{array}{l}
D_k=\sum_{j\in I_{0}} n_jq_{j,k} + \sum_{y\in \mathcal{S}_{0}^{(1)}\setminus{\{y_*\}}} x_{k,y}^{(1)}+ \sum_{y\in \mathcal{S}_2^{(1)}} \left(\sum_{w\in \mathcal{S}_0^{(2)}(y)\setminus{\{0\}}}x_{k,w}^{(2)}(y) \right),\\\\
deg (D_k) = N_0, \quad D_k \to D:= N_{0}x_{0}, \,\, k\to +\infty.
\end{array}
\end{equation}
Hence, if we take $\alpha \in Q({D}),\,\, \alpha_k \in Q(D_{k}): \,\, \alpha_k \to \alpha $ 
(given by Lemma \ref{approximation})
then, in local $z-$coordinates at $x_0,$ we have:\\ \\ 
$\alpha_k = a_{k,x_0}dz^2, \quad \alpha = a_{x_0}dz^2  $ with $ a_{x_0}(z) = z^{N_{0}} C(z), $ while $a_{k,x_0}$ takes the following form:
\begin{equation}\label{QDk}
\begin{array}{l}
 a_{k,x_0}(z+z_k) = \left[\left( \prod_{j \in I_{0} }(z- p_{j,k})^{n_j} \right) \prod_{y \in \mathcal{S}_{0}^{(1)}\setminus{\{y_*\}}}(z-  \tau_k^{(1)}z_{k,y}^{(1)})\right.\\\\ \left.  \prod_{ y \in \mathcal{S}_{2}^{(1)}} \prod_{w \in \mathcal{S}_{0}^{(2)}(y)\setminus{\{0\}}} \left(z-  \tau_k^{(1)}(z_{k,y}^{(1)} + \tau_{k}^{(2)}(y)z_{k,w}^{(2)}(y)) \right)  C_k(z+z_k) \right] 
 \end{array} \end{equation} 
 with $C_k \ \, \mbox{and} \,\, C \,\,\mbox{holomorphic and }\, C_k \to C \,\, \mbox{in }\,\,B_\delta, \,\, k \to +\infty.$\\
Thus, by recalling \eqref{psi}, we compute:
\begin{equation}\label{approx2}
\begin{array}{l}
\int_{B(x_0; r)} e^{\xi_k} <\a_k\, \, , \,\widehat \a_k > dA=\\ \\ 4 \int_{\{|z|<\tau_k^{(1)}R \}} e^{\xi_k}\overline{\hat{a}_{k,x_0}}(z+z_k)a_{k,x_0}(z+z_k)e^{-2u_X(z+z_k)}\frac{i}{2} dz \wedge d \bar{z}\\ 
\\+ o(\frac{1}{(\tau_k^{(1)})^{n-N_0}}) + o_r(1) = 4 \int_{\{|z|<\tau_k^{(1)}R \}} e^{\xi_k} \left[\prod_{j \in I} (\overline{z-p_{j,k}})^{n_j} \prod_{j \in I_{0}} (z-p_{j,k})^{n_j}\right.\\
\\\left. \prod_{y \in\mathcal{S}_{0}^{(1)}\setminus{\{y_*\}}}(z -\tau_k^{(1)} z_{k,y} ^{(1)})\prod_{ y \in \mathcal{S}_{2}^{(1)}} \prod_{w \in \mathcal{S}_{0}^{(2)}(y)\setminus{\{0\}}} \left(z-  \tau_k^{(1)}(z_{k,y}^{(1)} + \tau_{k}^{(2)}(y) z_{k,w}^{(2)}(y))\right) \right.\\\\ \left.   \overline{\Psi}_k(z +z_k)C_k(z+z_k) e^{-2u_X(z+z_k)} \frac{i}{2} dz \wedge d \bar{z} \right] + o(\frac{1}{(\tau_k^{(1)})^{n-N_0}}) + o_r(1) =\\
  \\
\frac{4}{{(\tau_k^{(1)})^{n-N_0}}}  \sum_{ y \in \mathcal{S}_{0}^{(1)}}\left \{  \int_{B_{\delta}(z_{k,y} ^{(1)})} e^{\varphi^{(1)}_k} \left[ \prod_{j \in I_{0}}  |z-p_{j,k}^{(1)}|^{2n_j}\prod_{y \in\mathcal{S}_{0}^{(1)}\setminus{\{y_*\}}}(z - z_{k,y} ^{(1)}) \right. \right. \\
\\ \left. \left.\prod_{j \in I\setminus I_{0}} (\overline{z-p_{j,k}^{(1)}})^{n_j} \prod_{ y \in \mathcal{S}_{2}^{(1)}} \prod_{w \in \mathcal{S}_{0}^{(2)}(y)\setminus{\{0\}}} \left(z-  (z_{k,y}^{(1)} + \tau_{k}^{(2)}(y) z_{k,w}^{(2)}(y))\right)\right. \right. \\ \\ \left.\overline{\Psi}_k(\tau_k^{(1)}z +z_k)C_k(\tau_k^{(1)}z +z_k)  \left. e^{-2u_X(\tau_k^{(1)}z +z_k)}\frac{i}{2} dz \wedge d \bar{z} \right] +o(1) \right \} +  o_r(1).
\end{array}
\end{equation}

Now, we analyze each of the integral terms above, and start by taking $y \in \mathcal{S}_{*}^{(1)}.$ We find:\\
\begin{equation}\label{intIa} \begin{array}{l}
\int_{B_{\delta}(z_{k,y} ^{(1)})}e^{\varphi^{(1)}_k}\left [ ... \right] = \int_{B_{\delta}}e^{\varphi^{(1)}_{k,y}}\prod_{j \in I_{y}}  |z-\tilde{p}_{j,k}^{(1)}|^{2n_j}\left[z H_{k,y}(z)\right.\\\\ \left.\overline{\Psi}_k(\tau_k^{(1)}(z +z_{k,y}^{(1)}) +z_k)C_k(\tau_k^{(1)}(z +z_{k,y}^{(1)}) +z_k) e^{-2u_X(\tau_k^{(1)}(z +z_{k,y}^{(1)})+z_k)}\frac{i}{2} dz \wedge d \bar{z} \right]\\ \\ \\=C(0)\overline{\Psi}_{0}(0)\left( \int_{B_{\delta}}e^{\varphi^{(1)}_{k,y}}\prod_{j \in I_{y}}  |z-\tilde{p}_{j,k}^{(1)}|^{2n_j}z H_{k,y}(z)\frac{i}{2} dz \wedge d \bar{z} +o(1)\right)\end{array} \end{equation}\\  with suitable sequence of functions $H_{k,y}(z)$ satisfying: $H_{k,y} \to H_{y} \,\,\mbox{as }\, k \to +\infty,$ uniformly in $B_{\delta}$ and $H_{y}(0) \neq 0.$ 
Moreover, by applying \eqref{convdebole1} to the blow- up point $y \in \mathcal{S}_{*}^{(1)}$, we conclude that,
\begin{equation}\label{intIb}
\int_{B_{\delta}}e^{\varphi^{(1)}_{k,y}}\prod_{j \in I_{y}}  |z-\tilde{p}_{j,k}^{(1)}|^{2n_j}z H_{k,y}(z)\frac{i}{2} dz \wedge d \bar{z} = o(1), \quad k \to +\infty.
\end{equation}

Next, we consider $y \in \mathcal{S}_{2}^{(1)}\setminus{\{y_*\}},$ 
and as above we find:\\
\begin{equation}
\begin{array}{l}
\int_{B_{\delta}(z_{k,y} ^{(1)})}e^{\varphi^{(1)}_k}\left [ ... \right] = \\ \\ = \int_{B_{\delta}}e^{\varphi^{(1)}_{k,y}}\prod_{j \in J_{(2)}(y)} |z-\tilde{p}_{j,k}^{(1)}(y)|^{2n_j}\left[ z\prod_{w \in \mathcal{S}_{0}^{(2)}(y)\setminus{\{0\}}} \left( z- \tau_{k}^{(2)}(y) z_{k,w}^{(2)}(y)\right) \right. \\ \\  H_{k,y}(z)\overline{\Psi}_k(\tau_k^{(1)}(z +z_{k,y}^{(1)})  +z_k)C_k(\tau_k^{(1)}(z +z_{k,y}^{(1)}) +z_k)\\ \\ e^{-2u_X(\tau_k^{(1)}(z +z_{k,y}^{(1)}) +z_k)}  \left.\frac{i}{2} dz \wedge d \bar{z},\right] 
\end{array}
\end{equation}
where again, $H_{k,y} \to H_{y} \,\,\mbox{as }\, k \to +\infty,$ uniformly in $B_{\delta}$ and $H_{y}(0) \neq 0.$  
Consequently, after a further scaling , we have:
\begin{equation}\label{intII}
\begin{array}{l}
\int_{B_{\delta}(z_{k,y} ^{(1)})}e^{\varphi^{(1)}_k}\left [ ... \right] = \frac{1}{(\tau_{k}^{(2)}(y))^{n_y^{(1)} - N_y^{(1)}-1} }\int_{\Omega_{k,\delta} ^{(1)}}\left[e^{\varphi^{(2)}_{k,y}}\prod_{j \in J_{(2)}(y)} |z-p_{j,k}^{(2)}(y)|^{2n_j} \right.\\ \\ \left. z\prod_{w \in \mathcal{S}_{0}^{(2)}(y)\setminus{\{0\}}} \left( z-  z_{k,w}^{(2)}(y)\right) \  H_{k,y}(\tau_{k}^{(2)}(y)z)\overline{\Psi}_k(\tau_k^{(1)}(\tau_{k}^{(2)}(y)z +z_{k,y}^{(1)})  +z_k)\right.\\ \\ C_k(\tau_k^{(1)}(\tau_{k}^{(2)}(y)z +z_{k,y}^{(1)}) +z_k)e^{-2u_X(\tau_k^{(1)}(\tau_{k}^{(2)}(y)z +z_{k,y}^{(1)}) +z_k)}  \left.\frac{i}{2} dz \wedge d \bar{z},\right] = \\ \\
\frac{C(0)\overline{\Psi}_0(0)H_{y}(0)}{(\tau_{k}^{(2)}(y))^{n_y^{(1)} - N_y^{(1)}-1} }\left(\int_{B_R}\left[e^{\varphi^{(2)}_{k,y}}\prod_{j \in J_{(2)}(y)} |z-p_{j,k}^{(2)}(y)|^{2n_j} \right. \right. \\ \\ \left. \left. z\prod_{w \in \mathcal{S}_{0}^{(2)}(y)\setminus{\{0\}}} \left( z-  z_{k,w}^{(2)}(y)\right)\frac{i}{2} dz \wedge d \bar{z}\right] + o(1)\right),
\end{array}
\end{equation}
and, by using for the blow-up point $y$ the analogous convergence property stated in \eqref{convweek2} for $y_0,$  we obtain that:
\begin{equation}\label{intIIbis}
\int_{B_R}e^{\varphi^{(2)}_{k,y}}\prod_{j \in J_{(2)}(y)} |z-p_{j,k}^{(2)}(y)|^{2n_j}  \prod_{w \in \mathcal{S}_{0}^{(2)}(y)\setminus{\{0\}}} z\left( z-  z_{k,w}^{(2)}(y)\right)\frac{i}{2} dz \wedge d \bar{z} =o(1),  
\end{equation}
as $k \to +\infty.$\\
Finally, arguing as above for $y_* \in \mathcal{S}_{2}^{(1)},$ we find a suitable sequence of functions $H_{k,y_*}$ satisfying: $H_{k,y_*} \to H_{y_*}\,\,\mbox{as }\, k \to +\infty,$ uniformly in $B_{\delta}$ and $H_{y_*}(0) \neq 0,$ and such that:
\begin{equation}\label{intIII}
\begin{array}{l}
\int_{B_{\delta}(z_{k,y_*} ^{(1)})}e^{\varphi^{(1)}_k}\left [ ... \right] = \int_{B_{\delta}}e^{\varphi^{(1)}_{k,y_*}}\left[ \prod_{w \in \mathcal{S}_{0}^{(2)}(y_*)\setminus{\{0\}}} \left( z- \tau_{k}^{(2)}(y_*) z_{k,w}^{(2)}(y_*)\right) \right. \\ \\ \left. \prod_{j \in J_{(2)}(y_*)} |z-\tilde{p}_{j,k}^{(1)}(y_*)|^{2n_j} H_{k,y_*}(z)\overline{\Psi}_k(\tau_k^{(1)}(z +z_{k,y}^{(1)})  +z_k)\right. \\ \\ \left.C_k(\tau_k^{(1)}(z +z_{k,y}^{(1)}) +z_k)e^{-2u_X(\tau_k^{(1)}(z +z_{k,y}^{(1)}) +z_k)}\frac{i}{2} dz \wedge d \bar{z} \right] =\\ \\ 
\frac{1}{(\tau_{k}^{(2)}(y_*))^{n_{y_{*}}^{(1)} - N_{y_{*}}^{(1)}}}\int_{\Omega_{k,\delta} ^{(1)}}\left[e^{\varphi^{(2)}_{k,y_{*}}}\prod_{j \in J_{(2)}(y_{*})} |z-p_{j,k}^{(2)}(y_{*})|^{2n_j}  \prod_{w \in \mathcal{S}_{0}^{(2)}(y_{*})\setminus{\{0\}}} \left( z-  z_{k,w}^{(2)}(y_{*})\right) \ \right.\\ \\ \left. H_{k,y_{*}}(\tau_{k}^{(2)}(y_*)z) \overline{\Psi}_k(\tau_k^{(1)}(\tau_{k}^{(2)}(y_*)z +z_{k,y}^{(1)})  +z_k)C_k(\tau_k^{(1)}(\tau_{k}^{(2)}(y_*)z +z_{k,y}^{(1)}) +z_k)\right.\\ \\ \left.e^{-2u_X(\tau_k^{(1)}(\tau_{k}^{(2)}(y_*)z +z_{k,y}^{(1)}) +z_k)}
\frac{i}{2} dz \wedge d \bar{z} \right]= \\ \\ 
\frac{C(0)\overline{\Psi}_0(0)H_{y_*}(0)}{(\tau_{k}^{(2)}(y_*))^{n_{y_{*}}^{(1)} - N_{y_{*}}^{(1)}}} \left(\int_{B_R}e^{\varphi^{(2)}_{k,y_{*}}}\prod_{j \in J_{(2)}(y_{*})}\left[ |z-p_{j,k}^{(2)}(y_{*})|^{2n_j} \right. \right. \\ \\ \left. \left. \prod_{w \in \mathcal{S}_{0}^{(2)}(y_{*})\setminus{\{0\}}} \left( z-  z_{k,w}^{(2)}(y_{*})\right)\frac{i}{2} dz \wedge d \bar{z} \right] + o(1) \right) \quad k \to +\infty. 
\end{array}
\end{equation}
\\
Hence, by using the analog of \eqref{convweek2} for the blow- up point $y_* \in \mathcal{S}_{2}^{(1)},$ we conclude:
\begin{equation}\label{intIIIbis}
\begin{array}{l}
\int_{B_R}e^{\varphi^{(2)}_{k,y_{*}}}\prod_{j \in J_{(2)}(y_{*})} |z-p_{j,k}^{(2)}(y_{*})|^{2n_j} \prod_{w \in \mathcal{S}_{0}^{(2)}(y_{*})\setminus{\{0\}}} \left( z-  z_{k,w}^{(2)}(y_{*})\right)\frac{i}{2} dz \wedge d \bar{z} = \\ \\
\frac{\pi}{4}{ M_w^{(2)}}(y_*)(\prod_{w\in\mathcal{S}_{0}^{(2)}(y_{*})\setminus{\{0\}}}(-w )) + o(1) ,\quad k \to +\infty
\end{array}
\end{equation}
\\
At this point, by recalling that: $C(0)=\frac{a_{x_0}^{(N_0)}(0)}{{N_{0}!}}$  then, we can use \eqref{minimo} together with \eqref{intIa}, \eqref{intIb}, \eqref{intII}, \eqref{intIIbis} and \eqref{intIII}, \eqref{intIIIbis} into \eqref{approx2}, to conclude that,
\begin{equation}\label{approx3}
\begin{array}{l}
\int_{B(x_0; r)} e^{\xi_k} <\a_k\, \, , \,\widehat \a_k > dA = \frac{\pi b_{x_0}}{(\tau_k^{(1)})^{n-N_0}(\tau_{k}^{(2)}(y_*))^{n_{y_{*}}^{(1)} - N_{y_{*}}^{(1)}}}(\frac {a_{x_0}^{(N_0)}(0)}{N_{0}!} \overline{ \psi_0(0)} +o(1)) \\ \\ + o_r(1), \quad
\mbox{with} \quad  b_{x_0} = M_w^{(2)}(y_*)\prod_{w\in\mathcal{S}_{0}^{(2)}(y_{*})\setminus{\{0\}}}(-w )H_{y_{*}}(0) \neq 0.
\end{array}
\end{equation}
Thus, when \eqref{vuoto} holds, we have proved \eqref{total asymp behavior2} with,
$$\varepsilon_{k,x_0} = (\tau_k^{(1)})^{n-N_0}(\tau_{k}^{(2)}(y_*))^{n_{y_{*}}^{(1)} - N_{y_{*}}^{(1)}} \to 0, \,\, k \to +\infty.$$

In case, for some $y \in \mathcal{S}_{2}^{(1)}$ we have that  $\mathcal{S}_{2}^{(2)}(y) \neq \emptyset,$ then around any blow-up point $w \in \mathcal{S}_{2}^{(2)}(y),$ (necessarily of "collapsing" type)  we can apply to the sequence $\varphi^{(2)}_{k,y}$ a further procedure of blow-up and obtain a new sequence to which the analogous of Proposition \ref{violate} applies. Again, in analogy to Lemma \ref{s,m}, for such new sequence we would have reduced either the number of "collapsing" zeroes conveging towards $w,$ (with respect to those "collapsing" zeroes converging towards $y$) or the corresponding value of the blow- up mass (with respect to $m_y^{(2)}$). In this way as above  we would obtain an additional "approximation" term to include into the sequence of divisors constructed above. By continuing in this way, and since such a procedure must stop after finitely many steps, we would end up with the appropriate sequence of divisors with the desired properties and such that \eqref{total asymp behavior2} holds. 

\end{proof}
\vskip.1cm
{\bf{Funding}} 
\vskip.1cm
The authors have been partially supported by MIUR Excellence Department Projects awarded to the Department of Mathematics, University of Rome Tor Vergata, 2018-2022 CUP E83C18000100006, and 2023-2027 CUP E83C23000330006. G.Tarantello is partially supported  by PRIN \emph{Variational and Analytical aspects of Geometric PDEs} n.2022AKNSE4. \\ S.Trapani is partially supported by PRIN \emph{Real and Complex Ma\-ni\-folds: Topology, Geometry and holomorphic dynamics} n.2017JZ2SW5. 
 \vskip.2cm
 {\bf{Authors and Affiliation}}
 
   \noindent{\sc Universit\`a di Roma TorVergata, Dipartimento di Matematica
  Via della Ricerca Scientifica 1, 00133 Roma, Italy.}\\
Gabriella Tarantello and  Stefano Trapani \\

{\tt tarantel@mat.uniroma2.it}

{\tt trapani@mat.uniroma2.it}

\end{document}